\documentclass[aos,preprint]{imsart}

\RequirePackage{amsthm,amsmath,amsfonts,amssymb}
\RequirePackage[authoryear]{natbib}
\RequirePackage[colorlinks,citecolor=blue,urlcolor=blue]{hyperref}
\RequirePackage{graphicx}
\usepackage{stmaryrd}
\usepackage[ruled,vlined]{algorithm2e}
\usepackage{enumerate}
\usepackage{mathabx}

\usepackage{booktabs}
\usepackage{multirow}

\usepackage{tikz}
\usetikzlibrary{shapes ,arrows ,positioning, automata}

\startlocaldefs
\theoremstyle{plain}

\newtheorem{theorem}{Theorem}[section]
\newtheorem{corollary}[theorem]{Corollary}
\newtheorem{lemma}[theorem]{Lemma}

\newtheorem{proposition}[theorem]{Proposition}
\theoremstyle{remark}
\newtheorem{remark}[theorem]{Remark}

\def \E{I\!\!E}
\def \P{I\!\!P}

\def \Z{\mathbb Z}

\newcommand{\indiq}{{{\mathbf 1}}}

\def\renewaltime{R}

\newcommand{\R}{\mathbb {R}}
\newcommand{\N}{\mathbb {N}}

\DeclareMathOperator{\var}{\mathop{Var}}
\DeclareMathOperator{\cov}{\mathop{Cov}}

\DeclareMathOperator{\diag}{\mathop{diag}}
\DeclareMathOperator{\d0}{\mathop{d_0}}

\renewcommand{\vec}{\operatorname{vec}}

\usepackage
	{todonotes}

\newcommand{\modif}[1]{{#1}}

\endlocaldefs

\begin{document}

\begin{frontmatter}
\title{Community detection for binary graphical models in high dimension} 
\runtitle{Community detection for binary graphical models}

\begin{aug}
\author[A]{\fnms{Julien} \snm{Chevallier}\ead[label=e1]{julien.chevallier1@univ-grenoble-alpes.fr}\orcid{0000-0002-0736-8487}},
\author[B]{\fnms{Guilherme} \snm{Ost}\ead[label=e3]{guilhermeost@im.ufrj.br}\orcid{0000-0003-0887-9390}}
\address[A]{Univ. Grenoble Alpes, CNRS, Inria, Grenoble INP\footnote{Institute of Engineering Univ. Grenoble Alpes}, LJK, 38000 Grenoble, France,
\printead{e1}
}
\address[B]{Institute of Mathematics, Federal University of Rio de Janeiro,
\printead{e3}}

\end{aug}

\begin{abstract}
Let $N$ components be partitioned into two communities, denoted ${\cal P}_+$ and ${\cal P}_-$, possibly of different sizes. Assume that they are connected via a directed and weighted Erd\"os-R\'enyi (DWER) random graph  with unknown parameter $ p \in  (0, 1).$ The weights assigned to the existing connections are of mean-field-type, scaling as $N^{-1}$. At each time \modif{step}, we observe the state of each component: either it sends some signal to its successors (in the directed graph) or remains silent otherwise. 
In this paper, we show that it is possible to find the communities ${\cal P}_+$ and ${\cal P}_-$ based only on the activity of the $N$ components observed over $T$ time units. More specifically, we propose 
\modif{
two simple methods, an aggregated method and a spectral method, 
whose {\it misclassification rates} vanish as long as $T \gg N$ (up to log terms).
This condition is proved to be near-optimal in the minimax sense. Moreover, under the stronger condition $T \gg N^2$ (up to log terms), the aggregated method is shown to achieve {\it exact recovery} with probability tending to $1$. 
}
Interestingly, these simple
\modif{methods} do not require any prior knowledge of the other model parameters (e.g. the edge probability $p$).
The key step in our analysis is to derive an asymptotic approximation of the
1-lagged
covariance matrix associated to the states of the $N$ components, as $N$ diverges. This asymptotic approximation relies on the study of the behavior of the solutions of a \modif{Stein-type} matrix equation satisfied by the simultaneous (0-lagged) covariance matrix associated to the states of the components. 
This study is challenging, especially because the simultaneous covariance matrix is random since it depends on the underlying DWER random graph.
\end{abstract}



\end{frontmatter}

\section{Introduction}
In this paper, we consider a system of $N$ interacting $\{0,1\}$-valued chains $X=\{X_{i,t},t\in\Z, 1\leq i\leq N\}$ that are coupled via a random matrix $\theta=(\theta_{ij})_{1\leq i,j\leq N}$ of i.i.d entries distributed as $\operatorname{Ber}(p)$ with ${\modif{0<} p}\leq 1$. Let $X_{t}=(X_{1,t},\ldots, X_{N,t})$ denote the configuration of the system $X$ at time $t$. The dynamics of the model is as follows. Conditionally on $\theta$, the system $X$ evolves as a stationary Markov chain on the 
state space $\{0,1\}^N$ in which the conditional distribution of $X_t$ given that $X_{t-1}=x$ is that of $N$ independent $\text{Ber}(p_{\theta,i}(x))$ random variables, $i=1,\ldots,N$, where
\begin{equation}
\label{def:transition_prob_2}
p_{\theta,i}(x)= \mu +(1-\lambda)\left(\frac{1}{N} \sum_{j\in\mathcal{P}_+} \theta_{ij}x_{j}+\frac{1}{N}\sum_{j\in\mathcal{P}_-} \theta_{ij}(1-x_{j})\right), \ x=(x_1,\ldots,x_N).
\end{equation}
The sets $\mathcal{P}_+$ and $\mathcal{P}_-$ form a partition of $[N]:=\{1,\ldots, N\}$ and the values $0 <  \lambda < 1 $ and $0 \modif{<} \mu \leq \lambda$ are fixed.
Neither the sets $\mathcal{P}_+$ and $\mathcal{P}_-$ nor the values $\mu$ and $\lambda$ depend on the random environment $\theta$. 

In this model, the transition probabilities $p_{i,\theta}(x)$ defined in \eqref{def:transition_prob_2} are increasing with respect to $x_j$ for $j\in \mathcal{P}_+$ and decreasing for $j\in \mathcal{P}_-$, so that the components 
$j\in \mathcal{P}_+$ have an excitatory role, whereas those belonging to $\mathcal{P}_-$ have an inhibitory role. In that respect, the sets $\mathcal{P}_+$ and $\mathcal{P}_-$ partitioning the set of components $[N]$ of the system $X$ introduce a {\it community structure} into the model: components in the same partition
have the same role.
Hereafter, for $a\in\{-,+\}$, we denote $r^N_a=|\mathcal{P}_a|/N$ the fraction of components of the system that belong to the community ${\cal P}_a$,
and assume that there are values $0<r_+,r_-<1$ satisfying $r_++r_-=1$ such that 
$|r^N_+-r_+|\vee |r^N_--r_-|\leq KN^{-1}$ for some constant $K$. In particular, we do not assume that the communities are balanced, i.e., the choice $r_+\neq r_-$ is allowed.

\paragraph*{Problem formulation} 
The goal of this paper is to address the following {\it community detection problem}.
By observing a sample $X_{1},\ldots, X_{T+1}$ of the system $X$, can we find the communities $\mathcal{P}_+$ and $\mathcal{P}_-$?
We do not assume any prior knowledge of $\mu, \lambda, p, r_+, r_-$ and $\theta$.

\subsection{Application motivation}

Part of our motivation \modif{to investigate this community detection problem} comes from Neuroscience. Nowadays, the simultaneous activity of large neuronal networks can be routinely recorded and understanding which structural properties of these networks can be inferred from the recordings is an important question. 
In several cases, neuronal recordings have been modeled by discrete-time $\{0,1\}$-valued stochastic chains that share some similarities with our model (see for example \citep{GalEva:13,duarte2019estimating} and references therein). In these cases, the time interval in which the neuronal network has been recorded is binned in small intervals and in each of these intervals, and for each recorded neuron, we assign the symbol 1 if the neuron has fired an action potential in that time interval and the symbol 0 otherwise. 
\modif{In this context, the sets $\mathcal{P}_+$
and $\mathcal{P}_-$ correspond respectively to the excitatory and inhibitory neurons of the network, and the community detection problem formulated above asks whether these two populations can be discriminated using only the recorded activity.}

\subsection{Related literature}

The model considered here was recently introduced in \citep{Chevallier2024inferring}. In their main result, the authors show that it is possible to estimate the connectivity parameter $p$ from the sample $X_{1},\ldots, X_{T}$ of the system $X$ with rate $N^{-1/2}+N^{1/2}/T + (\log(T)/T)^{1/2}$, \modif{provided that} the asymptotic fraction of components in each community (the values $r_+$ and $r_-$) is known a priori. As a consequence, the proposed estimator is shown to be consistent provided that $N^{1/2}/T\to 0$.
In this paper, we complement these results by showing that one can also find the communities defined by the sets $\cal{P}_+$ and $\cal{P}_-$.
More specifically, we propose 
\modif{
two simple methods, an aggregated method and a spectral method, for which we establish theoretical guarantees on exact recovery and misclassification rate control.
}
Interestingly, these simple \modif{methods}
do not require the knowledge of the asymptotic fraction of components that belong to each community.

Community detection problems have been extensively investigated over the past decades. A large part of the known results were established for the stochastic block model (SBM) - also known as planted partition model -  \citep{Emmanuel2016Exact,Elchanan2015Reconstruction,Mossel2016Belief} and its generalizations such as the degree-corrected SBM \citep{Zhao2012Consistency,2015LeiConsistency} and  mixed-membership SBM \citep{Airoldi2008Mixed}. We refer the reader to \citep{Abbe2018Community} for a recent review of the classical results in this framework.  
The problem of detecting the communities in such stochastic block models is different from ours for at least two reasons. First, for such models, one has access to the adjacency matrix with a block structure in  terms of the edge probabilities and uses a function of this matrix to find the communities. 
In our framework, the adjacency matrix counterpart is the random matrix $A^N$ defined in \eqref{def:signed_random_enviroment}. This matrix is a normalized and signed version of the random matrix $\theta$ in which all columns in $\mathcal{P}_-$ have a negative sign. Since the matrix $A^N$ depends on the matrix $\theta$ that is not observed, one cannot use the matrix $A^N$ in our detection problem. Instead, we can only use the information present in the sample $X_1,\ldots, X_{T+1}$ which only gives us a skewed view of $A^N$ (this is better justified in \eqref{eq:theorem:2:1:informal}).
Second, given the community assignments, there is no interaction between the components of those models: they behave independently of each other.
Here, on the contrary, the components of the model do interact and these interactions are crucial for finding the communities. 

More recently, the problem of community detection has been also addressed in the context of graphical models \citep{Berthet2019Exact} and G-latent models \citep{Bunea2020Model}. 
In both papers, the observations are assumed to be i.i.d. copies of some $N$-dimensional random vector $Z$. In \citep{Berthet2019Exact}, the authors suppose that the distribution of $Z$ is that of an Ising model without external field in which the interaction strength is $\beta N^{-1}$ between sites of the same block and $\alpha N^{-1}$ otherwise, where the real values $\alpha$ and $\beta$ satisfy $\beta>\alpha$. In particular, the block structure in this model is imposed directly on the interactions, similar to our model.   
In the framework of G-latent models considered in \citep{Bunea2020Model}, the community structure is imposed directly on the covariance matrix of the random vector $Z$, unlike the model considered in \citep{Berthet2019Exact} and ours. In the present paper, we go beyond the assumption that the observations are i.i.d by addressing the community detection problem for dependent variables. 
In this setting, part of the literature has focused on dynamic SBM \citep{Pensy2019Spectral,Lei2024Dynamic,Matias2016Statistical} - a framework fundamentally different from ours - where one observes sequences of adjacency matrices that slowly vary their community structures over time and the goal is to estimate these time varying communities. 
Closer to our setting, we are only aware of the works of \citep{Peixoto2019Network} and \citep{Hoffmann2020Community}. Both papers adopt a Bayesian approach for detecting communities from time series data. Their focus is on discussing the properties
of the proposed methods using synthetic and empirical data, rather than developing theoretical guarantees leading to 
exact recovery \modif{or controlling the misclassification rate}. 
Establishing such results in a similar context is the goal of our paper.

\subsection{Our contributions}

\modif{
In this paper, we first establish structural results on the model. These
are then used to derive theoretical guarantees for two community detection methods.
}

\subsubsection{Structural results}
The key step in our analysis is to derive an asymptotic approximation for
the 1-lagged
covariance matrix $\Sigma^{(1)}$ whose $(i,j)$-entries are defined as
\begin{equation}
\label{def:lag_1_Cov_matrix_signed_random_enviroment}
\Sigma^{(1)}_{ij}=\cov_{\theta}(X_{i,1},X_{j,0}), \ i,j\in [N],
\end{equation}
where $\cov_{\theta}(X_{i,1},X_{j,0})$ denotes the covariance between the random variables $X_{i,1}$ and $X_{j,0}$, conditionally on $\theta$. 
\modif{To describe this approximation, let $A^N$ be the random matrix}
whose $(i,j)$-entries are given by  
\begin{equation}
\label{def:signed_random_enviroment}
A^N_{ij}=
\begin{cases}
N^{-1}\theta_{ij}, \ \text{if} \ j\in{\cal{P}_+} \\
-N^{-1}\theta_{ij}, \ \text{if} \ j\in{\cal{P}_-} 
\end{cases}, \ i,j\in [N].
\end{equation}
\modif{With this notation,} we establish in Theorem \ref{thm:concentration:Sigma} 
that there are explicitly known constants $c_1 >0$ and $c_2$ (both depending on $\lambda, p,$ and $r_+$) such that, 
\modif{with high probability,
\begin{equation}
\label{eq:theorem:2:1:informal}
\max_{i,j\in [N]}\big|\Sigma^{(1)}_{ij}-(c_1A^N_{ij}+c_2N^{-1})\big|\leq K \sqrt{\log(N)/N^3},
\end{equation}
where the constant $K>0$ depends on $\lambda,p$ and $r_+$. 
This means that all entries of the matrix $\Sigma^{(1)}$ are uniformly close to those of the matrix $c_1A^N+c_2N^{-1}1_N1_N^{\top}$ (with high probability), where 
$1_N$ is the $N$-dimensional vector full of ones so that $1_N1_N^{\top}$ is the $N$-by-$N$ matrix full of ones.}
The term $c_2N^{-1}$ can be interpreted as a bias term present in all the entries of the matrix $\Sigma^{(1)}$. 
We note that the sign of this bias term is the same as the sign of $r_+-r_-$. In particular, the bias is $0$ if and only if the communities are balanced, i.e. $r_+=r_-=1/2.$

\modif{Once the entry-wise approximation \eqref{eq:theorem:2:1:informal} is established, we then derive an approximation for $\Sigma^{(1)}$ in the spectral norm, in terms of the matrix}
\modif{
\begin{equation}
\label{def:BarSigma}    
\overline{\Sigma}^{(1)}:=c_1\mathbb{E}[A^N]+c_2N^{-1}1_N1_N^{\top},
\end{equation}
\modif{where $\mathbb{E}[A^N]$ denotes the expectation of $A^N$, taken with respect to the distribution of the random matrix $\theta$.
Specifically}, we establish in Corollary \ref{cor:concentration:Sigma:operator:norm} that, with high probability,
\begin{equation}
\label{eq:cor:spectral:informal}
\vvvert \Sigma^{(1)} - \overline{\Sigma}^{(1)} \vvvert_2 \leq K \sqrt{\log(N)/N},
\end{equation}
for some constant $K>0$ depending on $\lambda,p$ and $r_+$.
}

The matrix $A^N$ defined in \eqref{def:signed_random_enviroment} is a normalized and signed version of $\theta$ in which the entries associated with columns in ${\cal P}_-$ have a negative sign. Hence the columns of $A^N$\modif{, and also of $\mathbb{E}[A^N]$,} carry information about the communities $\mathcal{P}_+$ and $\mathcal{P}_-$. 
The approximations in \eqref{eq:theorem:2:1:informal} \modif{and \eqref{eq:cor:spectral:informal}} suggest that the same is true for the columns of the random matrix $\Sigma^{(1)}$ for large values of $N$. 
\modif{
To formalize this, let $\sigma^{\rm ag}$ and $\overline{\sigma}^{\rm ag}$ be the vectors defined as 
\begin{equation}
\label{def:sigmaN:and:sigma:aggregated}
\sigma^{\rm ag}:=(\Sigma^{(1)})^{\top}1_N \ \text{and} \ \overline{\sigma}^{\rm ag}:=(\overline{\Sigma}^{(1)})^{\top}1_N.
\end{equation}
The superscript ${\rm ag}$ in both vectors $\sigma^{\rm ag}$ and $\overline{\sigma}^{\rm ag}$ stands for {\it aggregated}, since these vectors are obtained by summing over the columns of the matrix  $\Sigma^{(1)}$ and  $\overline{\Sigma}^{(1)}$, respectively.
} 

\modif{
Note that the random vector $(A^N)^{\top}1_N$ is expected to concentrate around its mean $\mathbb{E}[(A^N)^{\top}]1_N$ sufficiently fast (this is formalized in Inequality \eqref{eq:concentration:C:main}). Combining this fact with Equation \eqref{eq:theorem:2:1:informal}, we show in Corollary \ref{cor:concentration:sigma:sup:norm} that the vectors $\sigma^{\rm ag}$ and $\overline{\sigma}^{\rm ag}$ satisfy,
with high probability,
\begin{equation}
\label{infor:cor:concentration:sigma:sup:norm}
\max_{i\in [N]}|\sigma^{\rm ag}_i-\overline{\sigma}^{\rm ag}_i|\leq K\sqrt{\log(N)/N},  
\end{equation}
for some constant $K>0$ depending on $\lambda$, $p$ and $r_+$. 

Using \eqref{def:BarSigma} and the fact that the random variables $\theta_{ij}$'s are i.i.d with distribution
$\text{Ber}(p)$, one can check that
the vector $\overline{\sigma}^{\rm ag}$ defined in \eqref{def:sigmaN:and:sigma:aggregated}
can be rewritten as
\begin{equation}
\label{def:alternative:sigma:ag}
\overline{\sigma}^{\rm ag}=\overline{\sigma}_+ 1_{{\cal P_+}}+\overline{\sigma}_-1_{{\cal P_-}},
\end{equation}
where $1_{{\cal P}_+}$ (resp.  $1_{{\cal P}_-}$ ) is the $N$-dimensional vector whose entries are equal to 1 in ${\cal P}_+$ (resp.  ${\cal P}_-$) and $0$ elsewhere,
and the real values $\overline{\sigma}_+$ and $\overline{\sigma}_-$ are given by 
\begin{equation}
\label{def:sigmaplus:sigmaminus}
\overline{\sigma}_+=c_2 + c_1p\ \text{and} \ \overline{\sigma}_-=c_2-c_1p.    
\end{equation}

Since $\overline{\sigma}_+-\overline{\sigma}_-=2c_1p>0$, it follows from \eqref{def:alternative:sigma:ag} that the vector $\overline{\sigma}^{\rm ag}$ discriminates the communities ${\cal P_+}$ and ${\cal P_-}$, in the sense that the entries in  ${\cal P}_+$ take a larger value than those in $ {\cal P}_-$.
}

\modif{
\subsubsection{Aggregated method}
In view of the upper bound \eqref{infor:cor:concentration:sigma:sup:norm}, it is natural to expect that the vector $\sigma^{\rm ag}$ also discriminates these communities in the limit as $N$ diverges. 
However, the vector $\sigma^{\rm ag}$ cannot be used in practice, since the matrix $\Sigma^{(1)}$ is not observed.
In light of \eqref{def:sigmaN:and:sigma:aggregated}, we propose to estimate $\sigma^{\rm ag}$ by 
\begin{equation}
\label{def:hat:sigma:ag}
\widehat{\sigma}^{\rm ag}=(\widehat{\Sigma}^{(1)})^{\top}1_N,    
\end{equation}
where $\widehat{\Sigma}^{(1)}$, defined below in \eqref{def:empirical:Sigma1}, is the empirical estimate of $\Sigma^{(1)}$.   
Given the vector $\widehat{\sigma}^{\rm ag}$, it is then natural 
to estimate the communities ${\cal P}_+$ and ${\cal P}_-$ by clustering its coordinates. 
We call the resulting method {\it aggregated method}. 

Its theoretical guarantees are studied in Section \ref{sec:aggregated:method}. 
There, we first prove the convergence of $\widehat{\sigma}^{\rm ag}$ towards $\overline{\sigma}^{\rm ag}$, both in the sup-norm (Theorem \ref{thm:conver:hat:sigma:ag:towards:bar:sigma:ag}) and $\ell_2$-norm (Theorem \ref{thm:conver:hat:sigma:ag:towards:bar:sigma:ag:in:l2}).
The control in the sup-norm relies on \eqref{infor:cor:concentration:sigma:sup:norm} and \cite[Theorem 3]{ost2020sparse}, whereas 
the one in $\ell_2$ 
follows from \eqref{eq:cor:spectral:informal} and a sharp estimate for the distance between $\Sigma^{(1)}$ and its empirical estimate $\widehat{\Sigma}^{(1)}$, stated in Theorem \ref{thm:error:between:empirical:Sigma1:and:Sigma1}.
Then, using Theorem \ref{thm:conver:hat:sigma:ag:towards:bar:sigma:ag}, we show in Corollary \ref{cor:extact_recovery} that the aggregated method achieves exact recovery with high probability as long as $T\gg N^2$ (up to log factors).
In Appendix \ref{sec:lower:bound:exact:recovery}, we study the optimality of this condition in a similar statistical setting in which conditionally on $\theta$, the system $X$ starts from fixed initial conditions (and therefore is not stationary). In this setting, in Proposition \ref{prop:lower:bound:exact:recovery}, we derive an information-theoretic lower bound showing that exactly recovering the communities ${\cal P_+}$ and ${\cal P_-}$ is impossible in the minimax sense  if $T/(N \log(N))\to 0$. This result suggests that the condition $T\gg N^2$ may not be optimal for exact recovery.

Finally, as a consequence of Theorem \ref{thm:conver:hat:sigma:ag:towards:bar:sigma:ag:in:l2}, we show in Theorem \ref{thm:missclasification:rate:aggregated:method} that the misclassification rate of the aggregated method vanishes as long as $T\gg N$ (up to log factors). Combined with the information-theoretic lower bound stated in Theorem \ref{thm:minmax:lower:bound}, this result shows that the aggregated method is near-optimal with respect to the misclassification rate.
}

\modif{
\subsubsection{Spectral method}
In this paper, we also consider a {\it spectral method} for finding the communities ${\cal P_+}$ and ${\cal P_-}$.
One key observation behind our spectral method is that the matrix $\overline{\Sigma}^{(1)}$ has rank 1 and that the vector $\overline{\sigma}^{\rm ag}$ (up to renormalization) is a leading right singular vector of it (see Equation \eqref{SVD:bar:sigma1}).
Since singular vectors are only defined up to a $\pm 1$ factor, any leading singular vector of $\widehat{\Sigma}^{(1)}$ can  estimate $\overline{\sigma}^{\rm ag}$ only up to its sign, and clustering the entries of such a vector is prone to label ambiguity as it is well known in the spectral methods literature \citep{Chen2021Spectral}.
Interestingly, we show that in our case this ambiguity can be resolved by using 
the estimator $\widehat{\sigma}^{\rm ag}$ to determine the correct sign, 
yielding an improved spectral method presented in Section \ref{sec:spectral:method}.
}
\modif{
In Theorem \ref{thm:missclasification:rate:spectral:method}, we prove that the  misclassification rate 
of our spectral method vanishes as long as $T\gg N$ (up to log factors). This result together with Theorem \ref{thm:minmax:lower:bound} shows that our spectral method is also near-optimal with respect to the misclassification rate.
To prove Theorem \ref{thm:missclasification:rate:spectral:method}, we combine the perturbation theory of matrices with the approximation results stated in Equation \eqref{eq:cor:spectral:informal} and Theorem \ref{thm:error:between:empirical:Sigma1:and:Sigma1}. 
}

\subsubsection{Proof strategy of our key step}

The asymptotic approximation \eqref{eq:theorem:2:1:informal} is achieved by first showing that the \modif{1-lagged} covariance matrix $\Sigma^{(1)}$ can be computed from the simultaneous (0-lagged) covariance matrix $\Sigma^{(0)}$ whose $(i,j)$-entries are defined as 
\begin{equation}
\Sigma^{(0)}_{ij}=\cov_{\theta}(X_{i,0},X_{j,0}), \ i,j\in [N].
\end{equation}  
Next, we show that the simultaneous covariance matrix $\Sigma^{(0)}$ satisfies some Stein matrix equation (see Proposition \ref{prop:Sigma:0:matrix:equation}). To conclude, we then study the asymptotic behavior of the solutions of the matrix equation satisfied by $\Sigma^{(0)}$ as $N\to\infty$. This analysis is challenging, specially because the matrix $\Sigma^{(0)}$ depends on the random matrix $\theta$ and, hence, it is itself a random matrix.
The analysis involves two main ingredients: identifying an "almost" eigenvector of the Kronecker product $A^N \otimes A^N$, and a control 
\modif{of averages over entries of the matrix $\theta$ that are obtained via Hoeffding's inequality}.

\subsection{Concluding remarks, extensions and future work}
The present paper focuses on the case of two communities ${\cal P}_+$ and ${\cal P}_-$.
Let us stress that 
\modif{
our proposed methods
}
for finding these communities do not rely on the exact values of the two limits $\overline{\sigma}_+$ and $\overline{\sigma}_-$ defined in \eqref{def:sigmaplus:sigmaminus};
\modif{knowing that they are well seperated is sufficient.}

From a theoretical viewpoint, the case of two communities is a natural and widely studied setting in the community detection literature \cite{Abbe2018Community,Emmanuel2016Exact}.
Moreover, as noted by \cite{Berthet2019Exact} in the context of the Ising blockmodel, extending to multiple communities is particularly challenging when the population covariance matrix cannot be directly computed from the model parameters. A similar difficulty arises in our setting: the 1-lagged covariance matrix $\Sigma^{(1)}$ is not directly given in terms of the model parameters, and we must derive a suitable approximation for it (Theorem \ref{thm:concentration:Sigma}). Deriving such an approximation is by itself a difficult task, because the environment encoded in the matrix $\theta$ is random and observations are not i.i.d. in our model. 

In a suitably modified version of our model with $K>2$ communities, we expect the vector $\overline{\sigma}^{\rm ag}$ to have $K$ distinct values, so that it should be possible to extend our analysis of the aggregated method to the case of $K$ communities, though the details remain to be worked out.  
However, we conjecture that the analysis of the spectral method would be more challenging.
Indeed, with $K>2$ communities, we expect the matrix $\overline{\Sigma}^{(1)}$ to have rank $K-1$ and we would need to track $K-1$
singular vectors simultaneously, instead of only one as in our case. Moreover, the resolution of the label permutation ambiguity, which in our case reduces to a simple sign ambiguity, would become significantly more challenging for $K>2$. 

As mentioned earlier, a key property in our analysis is that the community 
structure only affects the columns of $A^N$, and in particular its rows 
share a common distribution. The aggregated method heavily relies on this property 
and should perform poorly if it is violated. 
Although the spectral method also relies on this property, 
we believe that a spectral approach is more robust to generalizations of the present model.  

Furthermore, throughout our analysis we assume that the parameter $p$ does not scale with the number of chains $N$ in the system $X$.
Despite the fact that some probabilistic tools used in our proofs are available for cases where $p$ scales with $N$ (e.g. Hoeffding's inequality), extending our results for such cases is not straightforward for at least two reasons. First, the model has to be suitably modified, since the normalization factor $N^{-1}$ is no longer appropriate. Second, and more importantly, our proof of Theorem \ref{thm:error:between:empirical:Sigma1:and:Sigma1}
relies on Lemma 3.8 of our companion paper \citep{Chevallier2024inferring}, which has been proved only for $p$ constant. This lemma controls the covariances between components of our model and is a key ingredient to control a variance term (this is done in Lemma \ref{lemma:bound:on:sigma}) appearing in the Bernstein-type concentration inequality derived in Lemma \ref{lem:sigmatilde:minus:sigma:in:operator:norm}, which, in turn, is used to prove Theorem \ref{thm:error:between:empirical:Sigma1:and:Sigma1}. 
Extending Lemma 3.8 to regimes where $p$ scales with $N$ would require a substantially different analysis and is beyond the scope of the present paper.
The study of the case $p$ scaling with $N$ is particularly relevant for applications where the underlying graph is sparse and dynamics of the model is sensitive to local graph structures (contrary to the mean-field assumption in \eqref{def:transition_prob_2}).

To conclude, we note that both aggregated and spectral methods are near-optimal in terms of the misclassification rate and that the aggregated method achieves exact recovery. 
Whether the spectral method also achieves exact recovery remains an open question.
A related open question is to understand whether the condition $T\gg N^2$ required by the aggregated method is optimal for exact recovery. Our lower bound in Proposition \ref{prop:lower:bound:exact:recovery} suggests that the sharp threshold may indeed be $T\gg N\log(N)$. Settling this issue is left for future work. 

\subsection{Outline of the paper}
The paper is organized as follows. In Section \ref{sec:results}, we define our model rigorously, introduce most of the mathematical notation used in the text and state our main results. 
\modif{
Section \ref{sec:proofs:structural} is dedicated to the proofs of the structural results of the model, namely Theorem \ref{thm:concentration:Sigma} and Corollaries \ref{cor:concentration:sigma:sup:norm} and \ref{cor:concentration:Sigma:operator:norm}. 
Sections \ref{sec:proof:exact:recovery:aggregated:method} and \ref{sec:proof:missclassification:rate} contain, respectively, the proofs of all results regarding exact recovery and misclassification rate.
In Section \ref{sec:simulation}, we illustrate the performance of our community detection methods through simulations.
Appendix \ref{sec:proof:max:inequality} contains concentration results derived from Hoeffding inequality that are used in our proofs. 
In Appendix \ref{sec:proof:conv:hat:sigma:ag:towards:sigma:ag}, we prove Theorem \ref{thm:conv:hat:sigma:ag:towards:sigma:ag}, on the convergence of $\widehat{\sigma}^{\rm ag}$ towards $\sigma^{\rm ag}$ in the sup-norm (the key step to prove Theorem \ref{thm:conver:hat:sigma:ag:towards:bar:sigma:ag}).  Section \ref{sec:proof:simult:cov:matrix} is devoted to the proof of the asymptotic approximation of the simultaneous covariance matrix, i.e. Proposition \ref{prop:concentration:asymp_approximation_simulta_cov_matrix}. 
In Appendix \ref{sec:proof:of:thm:error:between:empirical:Sigma1:and:Sigma1}, we prove Theorem \ref{thm:error:between:empirical:Sigma1:and:Sigma1}, which plays an important role in the analysis of the misclassification rate of both methods. 
In Appendix \ref{sec:coupling:concentration}, we establish a Bernstein-type concentration inequality for dependent matrices (stated in Lemma \ref{lem:sigmatilde:minus:sigma:in:operator:norm}) that is used in the proof of Theorem \ref{thm:error:between:empirical:Sigma1:and:Sigma1}. 
Appendix \ref{sec:proof:lower:bounds} is devoted to the proofs concerning the minimax lower bounds for the misclassification rate (Theorem \ref{thm:minmax:lower:bound}) and exact recovery (Proposition \ref{prop:lower:bound:exact:recovery}).
In Appendix \ref{app:covariance:lemma}, we recall Lemma 3.8 from our companion paper \cite{Chevallier2024inferring}, which is central in the proofs regarding the misclassification rate. 
Appendix \ref{sec:remarks:on:kmeans:ckustering:1d} contains an auxiliary result used in the proof concerning the exact recovery of the aggregated method. In Appendix \ref{app:notation:tables}, we present two tables summarizing the notation of the main objects in our analysis.  
Finally, Appendix \ref{app:plots} provides additional information about the performance of the aggregated method on simulated data.
}

\section{Model definition, notation and main results}
\label{sec:results}

\subsection{Model definition}
We consider a system of $N$ interacting chains $X=\{X_{i, t} , t \in \Z , 1 \le i \le N\} $ taking values in $ \{0, 1 \} $ denoting the presence or the absence of a signal at a given time. This system evolves in a random environment which is given by the realization of a directed Erd\"os-R\'enyi random graph via the selection of $N^2 $ i.i.d.\! random variables $ \theta_{ij}, 1 \le i, j \le N$, distributed as a Bernoulli of parameter $p$ \modif{with} $0 \modif{<} p\leq 1$.
Conditionally on $\theta$, the evolution of the system $X$ is that of a stationary Markov chain on the hypercube $\{0,1\}^N$ with transition probabilities (which depend on $\theta$) given as follows.
Writing $X_t=(X_{1,t},\ldots, X_{N,t})$ and $x,y\in \{0,1\}^{N}$, we have, for all $t\in\Z$,
\begin{equation}
\label{def:transition_prob_1}
\P_{\theta}(X_{t}=y  | X_{t-1}=x) =\prod_{i=1}^{N}(p_{\theta,i}(x))^{y_i}(1-p_{\theta,i}(x))^{(1-y_i)}, 
\end{equation}
where $p_{\theta,i}(x)=\P_{\theta}( X_{i,t}=1| X_{t-1}=x)$ is defined in \eqref{def:transition_prob_2}.
The existence and uniqueness of a stationary version of the Markov chain having transition probabilities as defined in \eqref{def:transition_prob_1} and  \eqref{def:transition_prob_2} is granted by \citep[Theorem 3.3]{Chevallier2024inferring}. Only the stationary regime is considered throughout this paper. 

\subsection{Notation}
Throughout the paper, we use the letters $t,s$ to denote some time indices whereas the letters $i,j$ denote some spatial values, that is the index of the corresponding component. In agreement with the left-hand side of \eqref{def:transition_prob_1}, we write $\P_{\theta}$ to denote the probability measure under which the environment $\theta$ is kept fixed and the process $X$ is distributed as the unique stationary version of the Markov chain having transition probabilities as defined in \eqref{def:transition_prob_1} and  \eqref{def:transition_prob_2}. The expectation taken with respect to $\P_{\theta}$ is denoted by $\E_{\theta}$. The variance and covariance computed from $\E_{\theta}$ are denoted $\var_{\theta}$ and $\cov_{\theta}$ respectively.   
Moreover, we write $\P$ to denote the probability measure under which the random environment $\theta=(\theta_{ij})_{i,j\in [N]}$ is distributed as a collection of i.i.d. random variables with distribution $\operatorname{Ber}(p)$ and the conditional distribution of the process $X$ given $\theta$ is that of the process $X$ under $\P_{\theta}$, i.e., the following identity holds $\P(X\in \cdot|\theta)=\P_{\theta}(X\in \cdot)$.  Finally, we denote $\E$ the expectation taken with respect to the probability measure $\P$, and $\var$ and $\cov$ the variance and covariance, respectively, computed from $\E$. 
\modif{When we need to emphasize that the partition $({\cal P}_+,{\cal P}_-)=(S,S^c)$, where $S\subset [N]$ with $|S|= \lceil Nr_+ \rceil$, we add an additional subscript to the previous notation. For example, $\P_{S,\theta}$ denotes the probability measure under which the environment $\theta$ is fixed and the process $X$ is distributed as the unique stationary version of the Markov chain having transition probabilities given by  \eqref{def:transition_prob_1} and  \eqref{def:transition_prob_2}, where the communities $({\cal P}_+,{\cal P}_-)=(S,S^c).$ 
The corresponding expectation is denoted by $\E_{S,\theta}$.}

For any $N$-by-$N$ matrix $M$, we shall denote $\vvvert M \vvvert_{p}$, with $1\leq p\leq \infty$, the matrix norm induced by the vector norm $\| \cdot \|_p$, and  $\vvvert M \vvvert_{\max}=\max_{i,j\in [N]}|M(i,j)|$. Recall that, although the matrix norm $\vvvert\cdot \vvvert_{\max}$ is not sub-multiplicative, it satisfies  
\begin{equation}
\label{eq:max:norm:submultiplicative}
    \vvvert AB \vvvert_{\max} \leq \min\left\{ \vvvert A \vvvert_{\infty} \vvvert B \vvvert_{\max}, \vvvert A \vvvert_{\max} \vvvert B \vvvert_{1} \right\},
\end{equation}
for any $N$-by-$N$ matrices $A$ and $B$. We write $I_N$ to denote the $N$-by-$N$ identity matrix.

Hereafter, for any subset $V$ of $[N]$, we write $1_V$ to indicate the $N$-dimensional vector having value $1$ in each coordinate belonging to $V$ and value $0$ in the remaining coordinates. To alleviate the notation, we will simply write $1_N$ instead of $1_{[N]}$, to indicate the vector full of ones. \modif{Along the text, the letter $K$ denotes a constant that is independent of $N$ and $T$ and that may change from line to line. Also, for any real values $a,b$, we let $a\wedge b$ and $a\vee b$ respectively denote the minimum and the maximum between $a$ and $b$. Finally, all the logarithms are in base $e$.}

\subsection{Structural results}

Let us denote
\begin{equation}
\label{def:mean_activity}
m=m(\mu,\lambda,p)=\frac{\mu+(1-\lambda)pr_-}{1-\modif{D}} \ \text{with} \ \modif{D = D(\lambda,p,r_+)=(1-\lambda)p(r_+-r_-)}
\end{equation}
which was proved in \citep[Theorem 2.1]{Chevallier2024inferring} to be the limit, in probability, of the spatio-temporal mean $(TN)^{-1}\sum_{i=1}^N\sum_{t=1}^T X_{i,t}$ as $T \wedge N \to\infty$. 
\modif{
Furthermore, denote
\begin{equation}
    \label{eq:def:c1:c2}
    \begin{cases}
        c_1 = (1-\lambda) m (1-m),\\
        c_2 = \frac{(1-\lambda)^2 p^2 (r_+ - r_-)}{1-D^2} c_1.
    \end{cases}
\end{equation}
For later use, let us observe that $c_1>0$ since $p>0$ and $\lambda<1$ and that $c_2$ is non-null as long as the network is unbalanced (i.e. $r_+\neq r_-$).
}

Recall the random matrices $\Sigma^{(1)}$ and $A^N$ defined in \eqref{def:lag_1_Cov_matrix_signed_random_enviroment} and \eqref{def:signed_random_enviroment} respectively. The matrix $\Sigma^{(1)}$ is the 
1-lagged
covariance matrix and $A^N$ is a normalized and signed version of random matrix $\theta$. 
Remark that $\Sigma^{(1)}$ also depends on $N$ but we chose to not highlight this dependence in the notation. Moreover, note that this matrix is not symmetric. 
\modif{
In our first structural result, namely Theorem \ref{thm:concentration:Sigma} below, we show that with high probability the matrix $\Sigma^{(1)}$ 
is close to the sum of two matrices: the first one is the matrix $c_1A^N$ and the second one is the matrix $c_2N^{-1}1_N1_N^{\top}$ which is interpreted a bias term.
The proof of Theorem \ref{thm:concentration:Sigma} is given in Section \ref{sec:proofs:structural}.

\begin{theorem}
    \label{thm:concentration:Sigma}
    There exists a positive constant $K>0$ depending on $\lambda, p$ and $r_+$ such that for all $\delta\in (0,1]$ and $N\geq 2$,
    $$
    \P\left(\left\vvvert \Sigma^{(1)} - \left( c_1 A^N + \frac{c_2}{N} 1_N 1_N^{\top} \right) \right\vvvert_{\max} \geq K\sqrt{\frac{\log(N/\delta)}{N^3}} \right) \leq \delta,
    $$
    where the matrices $\Sigma^{(1)}$ and $A^N$ are defined in \eqref{def:lag_1_Cov_matrix_signed_random_enviroment} and \eqref{def:signed_random_enviroment} respectively, and the real values $c_1$ and $c_2$ are defined in \eqref{eq:def:c1:c2}.
\end{theorem}
}%
\begin{remark}
\label{rem:constants}
    The constant $K$ appearing in Theorem \ref{thm:concentration:Sigma} remains bounded for fixed values of $\lambda$ (i.e. as functions of $p$ and $r_+$ only). However, this constant diverges when $\lambda \to 0$ (the ergodicity of dynamics breaks) or $\lambda \to 1$ (the dependence between the components of the model becomes very weak). 
    In such regimes, the result stated in Theorem \ref{thm:concentration:Sigma} is vacuous.
    The same remark applies to several other results, although we will not mention this explicitly.   
\end{remark}

\modif{
Recall the matrix $\overline{\Sigma}^{(1)}$ defined in \eqref{def:BarSigma}. Our second structural result establishes  that the vectors $\sigma^{\rm ag}=(\Sigma^{(1)})^{\top}1_N$ and $\overline{\sigma}^{\rm ag}=(\overline{\Sigma}^{(1)})^{\top}1_N$ are close in the sup norm, with high probability. 
This result is the basis for the aggregated method presented in Section \ref{sec:aggregated:method}. Its proof, given also in Section \ref{sec:proofs:structural}, follows from Theorem \ref{thm:concentration:Sigma} and concentration properties of the random vector $(A^N)^{\top}1_N$ around its mean $\E[(A^N)^{\top}]1_N$. 

\begin{corollary}
\label{cor:concentration:sigma:sup:norm}
There exists a positive constant $K>0$ depending on $\lambda, p$ and $r_+$ such that for all $\delta\in (0,1]$ and $N\geq 2$,
\[
\P\left(\left\| \sigma^{\rm ag} - \overline{\sigma}^{\rm ag} \right\|_{\infty} \geq K\sqrt{\frac{\log(N/\delta)}{N}} \right) \leq \delta,
\]
where the vectors $\sigma^{\rm ag}$ and $\overline{\sigma}^{\rm ag}$ are defined in \eqref{def:sigmaN:and:sigma:aggregated}.
\end{corollary}

Our final structural result shows  that the matrices $\Sigma^{(1)}$ and  $\overline{\Sigma}^{(1)}$ are close in the operator norm, with high probability. This result will be important in the analysis of the misclassification rate for both aggregated and spectral methods.
Its proof, presented in Section \ref{sec:proofs:structural}, follows from a combination of Theorem \ref{thm:concentration:Sigma} with concentration properties of the random matrix $A^N$ around its expectation $\mathbb{E}[A^N]$.

\begin{corollary}
\label{cor:concentration:Sigma:operator:norm}
There exists a positive constant $K>0$ depending on $\lambda, p$ and $r_+$ such that for all $\delta\in (0,1]$ and $N\geq 2$,
$$
    \P\left(\left\vvvert \Sigma^{(1)} - \overline{\Sigma}^{(1)} \right\vvvert_{2} \geq K\sqrt{\frac{\log(N/\delta)}{N}} \right) \leq \delta,
$$
where the matrices $\Sigma^{(1)}$ and $\overline{\Sigma}^{(1)}$ are defined in \eqref{def:lag_1_Cov_matrix_signed_random_enviroment} and \eqref{def:BarSigma} respectively.
\end{corollary}
}

\subsection{Aggregated method}
\label{sec:aggregated:method}

Recall the vectors $\sigma^{\rm ag}=(\Sigma^{(1)})^{\top}1_N$ and  $\overline{\sigma}^{\rm ag}=(\overline{\Sigma}^{(1)})^{\top}1_N.$ Observing that $\overline{\sigma}^{\rm ag}$ can be rewritten as $\overline{\sigma}_+1_{{\cal P}_+}+\overline{\sigma}_-1_{{\cal P}_-}$ (recall Equation \eqref{def:alternative:sigma:ag})
where $\overline{\sigma}_+-\overline{\sigma}_-=2c_1p>0$, we immediately see that the vector $\overline{\sigma}^{\rm ag}$ discriminates the communities ${\cal P}_+$ and ${\cal P}_-$. By Corollary \ref{cor:concentration:sigma:sup:norm}, we then expect that the communities can also be discriminated by the vector $\sigma^{\rm ag}$, with high probability.
In practice, however, the vector $\sigma^{\rm ag}$ cannot be used to that end because the matrix $\Sigma^{(1)}$ is not known. To circumvent this issue, one can try to find these sets by using the empirical estimate of $\sigma^{\rm ag}$.

Let $X_{1},\ldots, X_{T+1}$ be a sample of a stationary Markov chain with transition probabilities given by \eqref{def:transition_prob_1} and \eqref{def:transition_prob_2}, associated with some realization $\theta$ of the random environment which is not observed. Given the sample, we propose to estimate the vector 
$\sigma^{\rm ag} = (\Sigma^{(1)})^\top 1_N$ by the vector $\widehat{\sigma}^{\rm ag}=(\widehat{\Sigma}^{(1)})^\top 1_N$, where $\widehat{\Sigma}^{(1)}$ denotes the empirical estimate of the matrix $\Sigma^{(1)}$, which is defined as 
\begin{equation}
\label{def:empirical:Sigma1}
\widehat{\Sigma}^{(1)}=\frac{1}{T}\sum_{t=1}^{T}(X_{t+1}-\widehat{m})(X_{t}-\widehat{m})^{\top}, \ \text{with} \ \widehat{m}=\frac{1}{T}\sum_{t=1}^{T}X_{t}.
\end{equation}
In Theorem \ref{thm:conv:hat:sigma:ag:towards:sigma:ag}, stated and proved in Appendix \ref{sec:proof:conv:hat:sigma:ag:towards:sigma:ag}, we show that the estimator $\widehat{\sigma}^{\rm ag}$ 
concentrates around
$\sigma^{\rm ag}$ in the sup-norm, as both $T$ and $N$ diverge provided that $T\gg N^2$ (up to log factors).
Combining this result with Corollary \ref{cor:concentration:sigma:sup:norm}, we can show that the vector $\widehat{\sigma}^{\rm ag}$ also converges to $\overline{\sigma}^{\rm ag}$ in the sup-norm, provided that $T\gg N^2$ (up to log terms). This is the content of the result below.
\begin{theorem}
\label{thm:conver:hat:sigma:ag:towards:bar:sigma:ag}
Let $\widehat{\sigma}^{\rm ag}$ and $\overline{\sigma}^{\rm ag}$ be the vectors defined in \eqref{def:hat:sigma:ag} and \eqref{def:sigmaN:and:sigma:aggregated} respectively. There exist positive constants $K_1,K_2$ and $K_3$ depending on $\lambda$ such that 
$$
\P\left(\|\widehat{\sigma}^{\rm ag}-\overline{\sigma}^{\rm ag}\|_{\infty} \geq K_1\Bigg(\frac{N\log(TN)}{\sqrt{T}}+\sqrt{\frac{\log(N)}{N}}\Bigg)\right) \leq  K_2\Bigg(\frac{1}{T}+\frac{1}{N}\Bigg),
$$
for all pairs $(N,T)$ satisfying $N\geq 2$ and $T\geq K_3\log(N)$.    
\end{theorem}
Theorem \ref{thm:conver:hat:sigma:ag:towards:bar:sigma:ag} is the key ingredient in establishing the exact recovery guarantees of the aggregated method defined below. Its proof is given in Section \ref{sec:proof:exact:recovery:aggregated:method}.  

To control the misclassification rate, we need a control similar to that of Theorem \ref{thm:conver:hat:sigma:ag:towards:bar:sigma:ag} with $\ell_2$-norm replacing the sup-norm. 
A naive approach that combines the crude bound $\| x \|_2 \leq \sqrt{N}\| x \|_\infty$ with Theorem \ref{thm:conver:hat:sigma:ag:towards:bar:sigma:ag} leads to a sub-optimal estimate for the convergence rate of $\widehat{\sigma}^{\rm ag}$ towards $\overline{\sigma}^{\rm ag}$ in the $\ell_2$-norm.  
To achieve optimal convergence rate, we use Corollary \ref{cor:concentration:Sigma:operator:norm} and a sharp estimate 
for the distance between the 1-lagged covariance matrix $\Sigma^{(1)}$ and its empirical estimate $\widehat{\Sigma}^{(1)}$ (Theorem \ref{thm:error:between:empirical:Sigma1:and:Sigma1}). 
\begin{theorem}
\label{thm:conver:hat:sigma:ag:towards:bar:sigma:ag:in:l2}
Let $\widehat{\sigma}^{\rm ag}$ and $\overline{\sigma}^{\rm ag}$ be the vectors defined in \eqref{def:hat:sigma:ag} and \eqref{def:sigmaN:and:sigma:aggregated} respectively. 
Then there exist positive constants $K_1,K_2$ and $K_3$ depending on $\lambda$ such that 
$$
\P\left(\|\widehat{\sigma}^{\rm ag}-\overline{\sigma}^{\rm ag}\|_{2} \geq K_1\Bigg(N\sqrt{\frac{\log(T)\log(N\log(T))}{T}}+\sqrt{\log(N)}\Bigg)\right) \leq  K_2\Bigg(\sqrt{\frac{N}{T}}+\frac{1}{N}\Bigg),
$$
for all pairs $(N,T)$ satisfying $N\geq 2$ and 
\begin{equation}
\label{def:1:over:SNRn}
T\geq (K_3N\log(T)\log(N\log(T)))\vee 2.
\end{equation}    
\end{theorem}
The proof of Theorem \ref{thm:conver:hat:sigma:ag:towards:bar:sigma:ag:in:l2} is given in Section \ref{sec:proof:missclassification:rate}. Note that Condition \eqref{def:1:over:SNRn} requires the time horizon $T$ to be sufficiently large when compared to the number $N$ of interacting chains composing the system. 

In light of the properties of the vector $\overline{\sigma}^{\rm ag}$ and of Theorems \ref{thm:conver:hat:sigma:ag:towards:bar:sigma:ag} and \ref{thm:conver:hat:sigma:ag:towards:bar:sigma:ag:in:l2}, it is natural to estimate the communities ${\cal P}_+$ and ${\cal P}_-$ by clustering the entries of $\widehat{\sigma}^{\rm ag}$. Although our approach extends to several clustering procedures, in what follows we analyze in detail only two of them. 

The first is the {\it $k$-means} clustering with $k=2$, which computes the global minimizer of the $k$-means criterion.
Since the entries of $\widehat{\sigma}^{\rm ag}$ are real numbers, this can be done efficiently (in time linear in $N$) via the dynamic programming algorithm developed in \cite{Wu1991Optimal} (see also \cite{gronlund2018fast} for a review of the existing approaches to one-dimensional $k$-means clustering).
Note that, in general, its output is different from that of Lloyd's algorithm.
We denote by $(\widehat{\cal P}_+^{\rm ag,km},\widehat{\cal P}_-^{\rm ag,km})$ the estimated communities obtained by applying $k$-means clustering to the vector $\widehat{\sigma}^{\rm ag}$.

The second procedure, which we refer to as {\it mean threshold} clustering, estimates the communities ${\cal P_+}$ (resp. ${\cal P_-}$) as the 
set of indices $i\in [N]$ for which $\widehat{\sigma}^{\rm ag}_i$ is above (resp. below) 
the mean value $N^{-1}\sum_{j=1}^N\widehat{\sigma}_j^{\rm ag}$. We denote the resulting 
communities by $(\widehat{\cal P}_+^{\rm ag,mt},\widehat{\cal P}_-^{\rm ag,mt}).$   

Our next result provides the theoretical guarantees for these clustering procedures with respect to exact recovery.
\begin{corollary}
\label{cor:extact_recovery}
Let $\widehat{\sigma}^{\rm ag}$ be the vector defined in \eqref{def:hat:sigma:ag} and denote by $(\widehat{\mathcal{P}}^{\rm ag,\bullet}_+, \widehat{\mathcal{P}}^{\rm ag,\bullet}_-)$ the estimated communities obtained by applying $k$-means or mean threshold clustering to $\widehat{\sigma}^{\rm ag}$. Then there exist an integer $N_0$ and positive constants $K_1$ and $K_2$ depending only on $\lambda, p,r_+$ such that  
$$
\P\left((\widehat{{\cal P}}_{+}^{\rm ag,\bullet},\widehat{{\cal P}}_{-}^{\rm ag,\bullet})\neq ({\cal P_{+}},{\cal P_{-}})\right)\leq K_2\Big(\frac{1}{N}+\frac{1}{T}\Big)
$$
for all $(N,T)$ such that $N\geq N_0$ and 
\begin{equation}
T\geq K_1N^2\log^2(TN).    
\end{equation}
\end{corollary}
\begin{remark}
 As our proof reveals, the conclusion of Corollary \ref{cor:extact_recovery} also holds if the $k$-means clustering is replaced by Lloyd's $k$-means algorithm, initialized with the means $m_1=\min_{i\in [N]}\widehat{\sigma}^{\rm ag}_i$ and $m_2=\max_{i\in [N]}\widehat{\sigma}^{\rm ag}_i$. Indeed, whenever $\widehat{\sigma}^{\rm ag}$ is sufficiently close to $\overline{\sigma}^{\rm ag}$ in the sup-norm, we show that Lloyd's algorithm converges to the
 true partition $({\cal P}_+, {\cal P}_-)$ in one iteration.

The analysis for $k$-means clustering is more involved as it requires to show that the global minimizer of the $k$-means criterion is the true $({\cal P}_+, {\cal P}_-)$,  provided that the vector $\widehat{\sigma}^{\rm ag}$ and $\overline{\sigma}^{\rm ag}$ are close in the sup-norm. This is done in Lemma \ref{lem:structural:properties:kmeans} of Appendix \ref{sec:remarks:on:kmeans:ckustering:1d}.   
\end{remark}
Like Theorem \ref{thm:conver:hat:sigma:ag:towards:bar:sigma:ag}, the proof of Corollary \ref{cor:extact_recovery} is given in Section \ref{sec:proof:exact:recovery:aggregated:method}. It follows from Corollary \ref{cor:extact_recovery} that we can exactly recover the communities ${\cal P}_+$ and ${\cal P}_-$ by clustering the vector $\widehat{\sigma}^{\rm ag}$ via either $k$-means or mean threshold, as long as $T\gg N^2$ (up to log factors).
In Section \ref{sec:simulation}, we perform a numerical study to compare the performance of these two clustering procedures (and also include hierarchical clustering in this study).
We note that the information-theoretic lower bound derived in Proposition \ref{prop:lower:bound:exact:recovery} of Appendix \ref{sec:lower:bound:exact:recovery}  shows, in a slightly different setting, that exactly recovering the communities ${\cal P_+}$ and ${\cal P_-}$ is impossible in the minimax sense if $T/(N \log(N))\to 0$. This suggests that the condition $T\gg N^2$ may not be optimal, and resolving this issue is left for future work.

In what follows, we focus on the performance of the aggregated method with
$k$-means or mean threshold clustering in terms of its misclassification rate. Here, 
the misclassification rate of an estimated partition $(\widehat{{\cal P}}_+,\widehat{{\cal P}}_-)$ is defined by 
\begin{equation}
\label{def:discrepancy:measure}
\operatorname{MR}(\widehat{{\cal P}}_+,\widehat{{\cal P}}_-)=N^{-1}(|\widehat{{\cal P}}_+\cap {\cal P}_-|+|\widehat{{\cal P}}_-\cap {\cal P}_+|).
\end{equation}

\begin{theorem}
\label{thm:missclasification:rate:aggregated:method}
Let $\widehat{\sigma}^{\rm ag}$ be the vector defined in \eqref{def:hat:sigma:ag} and denote by $(\widehat{\mathcal{P}}^{\rm ag,\bullet}_+, \widehat{\mathcal{P}}^{\rm ag,\bullet}_-)$ the estimated communities obtained by applying $k$-means or mean threshold clustering to $\widehat{\sigma}^{\rm ag}$. 
Then there exist an integer $N_0$ and positive constants $K_1$, $K_2$ and $K_3$ depending only on $\lambda, p$ and $r_+$ such that  
\[
\P\Big(\operatorname{MR}(\widehat{\cal P}^{\rm ag, \bullet}_+,\widehat{\cal P}^{\rm ag, \bullet}_-)\geq K_1\Big(\frac{N\log(T)\log(N\log(T))}{T}+\frac{\log(N)}{N}\Big)\Big)\leq K_2\Big(\sqrt{\frac{N}{T}}+\frac{1}{N}\Big),
\]
for all pairs $(T,N)$ for which $N\geq N_0$ and \eqref{def:1:over:SNRn} is satisfied.
\end{theorem}
The proof of Theorem \ref{thm:missclasification:rate:aggregated:method} is provided in Section \ref{sec:proof:missclassification:rate}. The key step in the proof is to obtain an estimate for missclassification rate in terms of $\ell_2$-distance between $\widehat{\sigma}^{\rm ag}$ and $\overline{\sigma}^{\rm ag}$. This is the purpose of Lemma \ref{lem:l2norm:controls:MR}, stated and proved in Section \ref{sec:proof:missclassification:rate}. Given this estimate, we then conclude by Theorem \ref{thm:conver:hat:sigma:ag:towards:bar:sigma:ag:in:l2}.

Finally, let us observe that since condition \eqref{def:1:over:SNRn} matches the condition $T \gg N$ up to log terms, the minimax lower bound stated in Theorem \ref{thm:minmax:lower:bound} of Section \ref{sec:lower:bounds} implies that the estimated partitions $(\widehat{\mathcal{P}}^{\rm ag,\bullet}_+, \widehat{\mathcal{P}}^{\rm ag,\bullet}_-)$ 
are both near-optimal, with respect to the misclassification rate. 
An interesting open question is to understand whether the log factors in condition \eqref{def:1:over:SNRn} can be removed.

\subsection{Spectral method}
\label{sec:spectral:method}

Recall the constants $c_1$ and $c_2$ defined in \eqref{eq:def:c1:c2}.
Using that $\overline{\Sigma}^{(1)}=c_1\E[A^N]+c_2N^{-1}1_N1_N^{\top}$ and that $\E[A^N]=pN^{-1}1_N(1_{{\cal P}_+}-1_{{\cal P}_-})^{\top}$, we can rewrite $\overline{\Sigma}^{(1)}$ as
\begin{equation}
\label{SVD:bar:sigma1}
\overline{\Sigma}^{(1)}= N^{-1} 1_N (\overline{\sigma}^{\rm ag})^{\top},
\end{equation}
where $\overline{\sigma}^{\rm ag}$ is the vector given by \eqref{def:alternative:sigma:ag}. 

Starting from \eqref{SVD:bar:sigma1}, one can verify after simple calculations that 
\begin{equation}
\label{def:sigma:sp}
\overline{\sigma}^{\rm sp}:=\frac{\overline{\sigma}^{\rm ag}}{\|\overline{\sigma}^{\rm ag}\|_2},    
\end{equation}
is a right singular vector of $\overline{\Sigma}^{(1)}$ with corresponding singular value
\begin{equation}\label{def:leading:singular:value:bar:sigma1}
    \sigma_1\big(\overline{\Sigma}^{(1)}\big)=\sqrt{(\overline{\sigma}_+)^2r^N_++(\overline{\sigma}_-)^2r^N_-}.
\end{equation}
Except in trivial cases, $\sigma_1\big(\overline{\Sigma}^{(1)}\big)>0$ so that $\overline{\Sigma}^{(1)}$ has rank 1 and $\sigma_1\big(\overline{\Sigma}^{(1)}\big)$ is its unique non-null singular value.

Next observe that Corollary \ref{cor:concentration:Sigma:operator:norm} and Theorem \ref{thm:error:between:empirical:Sigma1:and:Sigma1} together imply that the matrices $\widehat{\Sigma}^{(1)}$ and   $\overline{\Sigma}^{(1)}$ are close in operator norm with high probability. In view of the above properties of $\overline{\Sigma}^{(1)}$ and of Wedin's $\sin \Theta$ theorem (see for instance Inequality  \eqref{ineq:1:proof:thm:community:detection}), it then follows that, with high probability, any leading right singular vector $\widecheck{v}$ of the matrix $\widehat{\Sigma}^{(1)}$ is close to $\overline{\sigma}^{\rm sp}$, up to sign.
The sign ambiguity can be resolved using $\widehat{\sigma}^{\rm ag}$ as follows. Let $\widecheck{{\cal P}}$ be a preliminary estimated community defined as 
\begin{equation}\label{eq:def:preliminary:partition}
\widecheck{{\cal P}} = \left\{i\in[N]: \widecheck{v}_i \ge N^{-1} \sum_{j=1}^N \widecheck{v}_j\right\}.
\end{equation}
Depending on the ambiguous sign, $\widecheck{{\cal P}}$ is expected to be close to ${\cal P}_+$ or ${\cal P}_-$,   and consequently the average of the coordinates of $\widehat{\sigma}^{\rm ag}$ over $\widecheck{{\cal P}}$ is expected to be close to $\overline{\sigma}_+$ or $\overline{\sigma}_-$. Hence, we propose to estimate the vector $\overline{\sigma}^{\rm sp}$ by the vector
\begin{equation}\label{eq:def:hat:sigma:spectral}
    \widehat{\sigma}^{\rm sp} =
    \begin{cases}
        \widecheck{v} & \text{, if } |\widecheck{{\cal P}}|^{-1}\sum_{i\in \widecheck{{\cal P}}}\widehat{\sigma}_i^{\rm ag} \geq |[N]\setminus \widecheck{{\cal P}}|^{-1}\sum_{i\in [N]\setminus \widecheck{{\cal P}}}\widehat{\sigma}_i^{\rm ag},\\
        -\widecheck{v} & \text{, else}.
    \end{cases}
\end{equation}

\begin{theorem}
\label{thm:estimation:spectral:method}
Let $\widehat{\sigma}^{\rm sp}$ be defined in \eqref{eq:def:hat:sigma:spectral}. 
There exist an integer $N_0$ and positive constants $K_1$, $K_2$ and $K_3$ depending only on $\lambda, p$ and $r_+$ such that
\[
\P\Big( \| \widehat{\sigma}^{\rm sp} - \overline{\sigma}^{\rm sp} \|_2 \geq K_1\Big(\sqrt{\frac{N\log(T)\log(N\log(T))}{T}} + \sqrt{\frac{\log(N)}{N}}\Big)\Big)\leq K_2\Big(\sqrt{\frac{N}{T}}+\frac{1}{N}\Big),
\]
for all pairs $(T,N)$ for which $N\geq N_0$ and \eqref{def:1:over:SNRn} is satisfied.
\end{theorem}

Theorem \ref{thm:estimation:spectral:method} is proved in Section \ref{sec:proof:missclassification:rate}. As anticipated above, its proof follows from the perturbation theory of matrices (Wedin's Sin $\Theta$ theorem), Corollary \ref{cor:concentration:Sigma:operator:norm} and Theorem \ref{thm:error:between:empirical:Sigma1:and:Sigma1}. The most challenging part in the proof is to show that the sign of $\widehat{\sigma}^{\rm sp}$ is correct with high probability. 

Finally, the estimated partition is defined as
\begin{equation}\label{def:partition:via:spectral:method}
    \widehat{{\cal P}}^{\rm sp}_+ = \left\{i\in[N]: \widehat{\sigma}^{\rm sp}_i \ge N^{-1} \sum_{j=1}^N \widehat{\sigma}^{\rm sp}_j\right\}
    \quad \text{and} \quad
    \widehat{{\cal P}}^{\rm sp}_- = [N] \setminus \widehat{{\cal P}}^{\rm sp}_+,
\end{equation}
and its misclassification rate is upper bounded as follows.

\begin{theorem}
\label{thm:missclasification:rate:spectral:method}
Let $(\widehat{\cal P}^{\rm sp}_+,\widehat{\cal P}^{\rm sp}_-)$ be the partition defined in \eqref{def:partition:via:spectral:method}. 
There exist an integer $N_0$ and positive constants $K_1$, $K_2$ and $K_3$ depending only on $\lambda, p$ and $r_+$ such that
\[
\P\Big(\operatorname{MR}(\widehat{\cal P}^{\rm sp}_+,\widehat{\cal P}^{\rm sp}_-)\geq K_1\Big(\frac{N\log(T)\log(N\log(T))}{T}+\frac{\log(N)}{N}\Big)\Big)\leq K_2\Big(\sqrt{\frac{N}{T}}+\frac{1}{N}\Big),
\]
for all pairs $(T,N)$ for which $N\geq N_0$ and \eqref{def:1:over:SNRn} is satisfied.
\end{theorem}
The proof of Theorem \ref{thm:missclasification:rate:spectral:method} is provided in Section \ref{sec:proof:missclassification:rate} as well. Similarly to the aggregated method,
the control of the misclassification rate of the spectral method relies on Lemma \ref{lem:l2norm:controls:MR} and Theorem \ref{thm:estimation:spectral:method}. Moreover, the minimax lower bound stated in Theorem \ref{thm:minmax:lower:bound} implies that the random partition defined in \eqref{def:partition:via:spectral:method} is near-optimal.

\subsection{Lower bound for the misclassification rate}
\label{sec:lower:bounds}
In this section, we establish lower bounds on the sample size $T$ needed to recover the partition $({\cal P}_+,\cal{P}_-)$, with respect to the either misclassification rate. 
For each $K>0$ and $u\geq 0$, denote $\tau_K(u)=\sqrt{1-\exp(-Ku)}$. Note that $\tau_{K}(u)\to 1$ as $u\to \infty$, whereas $\tau_{K}(u)\to 0$ as $u\to 0$.
\begin{theorem}
    \label{thm:minmax:lower:bound}
There exist positive constants $K_1$ and $K_2$ depending only on $\mu$ and $\lambda$ such that for all $N \geq 2$ and $T\geq 1$, 
\begin{equation*}
\min_{(\widehat{\cal P}_+,\widehat{\cal P}_-)}\max_{S\in{\cal S}}\E_{S}\Big[\operatorname{MR}(\widehat{\cal P}_+,\widehat{\cal P}_-)\Big]\geq \frac{(r_{\min} - N^{-1})}{2}\Bigg(1-\tau_{K_1}\Big(\frac{T}{N}\Big)-2\tau_{K_2}\Big(\frac{\log(2N)}{N}\Big)\Bigg),   \end{equation*}
where $r_{\min}=r_+\wedge r_-$, ${\cal S}=\{S\subset [N]: |S|= \lceil Nr_+ \rceil\}$ and the infimum is taken over all random partitions $(\widehat{\cal P}_+,\widehat{\cal P}_-)$ of $[N]$ which are measurable with respect to $\sigma(X_1,\ldots, X_{T+1}).$
\end{theorem}
The proof of Theorem \ref{thm:minmax:lower:bound} is presented in Appendix \ref{sec:lower:bound:missclassification}. The main step in the proof is to 
upper bound the variation of the distribution of observations when the true partition $(\cal P_+,\cal P_-)$ is slightly modified.
Once this is done, we conclude by applying Assouad's Lemma \cite[Theorem 2.12]{Tsybakov2009Introduction}. 

Note that if both $T=T_N$ and $N$ diverge but $T_N/N\leq K$ for some constant $K>0$, then Theorem \ref{thm:minmax:lower:bound} implies that
\[
\liminf_{N\to\infty}\min_{(\widehat{\cal P}_+,\widehat{\cal P}_-)}\max_{S\in{\cal S}}\E_{S}\Big[\operatorname{MR}(\widehat{\cal P}_+,\widehat{\cal P}_-)\Big]\geq \frac{r_+\wedge r_-}{2}\big(1-\tau_{K_1}\big(K\big)\big)>0,
\]
showing that recovering the communities ${\cal P}_+$ and ${\cal P}_-$ is impossible in the minimax sense unless $T \gg N$.

\section{Proofs of the structural results}
\label{sec:proofs:structural}

\subsection{Some vector and matrix notation}
\label{sec:vector:matrix:notration:from:first:paper}

The aim of this section is to introduce 
\modif{additional notation 
and simple results already appearing in \citep{Chevallier2024inferring}. 
}

Let $L^{N,\bullet -}$ be the vector of row-wise sums over the columns corresponding to $\mathcal{P}_-$ of the matrix $A^N$, i.e. $L^{N,\bullet -} = A^N 1_{\mathcal{P}_-}$. Then, for all $x\in \{0,1\}^N$, the vector $p_{\theta}(x) = (p_{\theta,1}(x),\dots,p_{\theta,N}(x))$ can be written as,
\begin{equation}\label{eq:p:theta:matrix:notation}
    p_{\theta}(x) = \mu 1_N + (1-\lambda)\left( A^N x - L^{N,\bullet -} \right).
\end{equation}
Now, let us denote $m^N$ the (conditional) mean vector given by
\begin{equation}\label{eq:def:mean:vector}
m^N=(m^N_1,\ldots,m^N_N) = \E_{\theta}(X_0) = \E_{\theta}[p_{\theta}(X_0)],
\end{equation}
where we used the stationarity of the process. With this notation, let us remark that $\Sigma^{(1)}$ rewrites as
\begin{equation}\label{eq:Sigma1:matrix:notation}
    \Sigma^{(1)} = \E_{\theta}\left[ (X_1 - m^N)(X_0 - m^N)^{\top} \right].
\end{equation}
Taking expectation in both sides of Equation \eqref{eq:p:theta:matrix:notation}, one has
\begin{equation}\label{eq:mN:matrix:notation}
    m^N=\mu 1_{N}+(1-\lambda)\left(A^Nm^N-L^{N,\bullet -}\right),
\end{equation}
which in particular gives that the centered version of $p_{\theta}(X_0)$ is
\begin{equation}\label{eq:centered:p:theta}
    p_{\theta}(X_0) - m^N = (1-\lambda) A^N (X_0 - m^N).
\end{equation}
Moreover, using the matrix $Q^N = (I_N - (1-\lambda)A^N)^{-1}$ and the notation for its row-wise sums $\ell^N = Q^N 1_N$ and $\ell^{N,\bullet -} = Q^N 1_{\mathcal{P}_-}$, 
\modif{
the solution $m^N$ to Equation \eqref{eq:mN:matrix:notation} reads as
\begin{equation}\label{eq:mN:ell:ell-}
    m^N = \mu \ell^N + 1_{\mathcal{P}_-} - \ell^{N,\bullet -}.
\end{equation}
}
\modif{Let us remark that the inverse matrix $Q^N$ is controlled by \citep[Section F.3]{chevallier2026supp}:
\begin{equation}\label{eq:control:QN}
    \vvvert Q^N \vvvert_{\infty} \leq \lambda^{-1}.
\end{equation}
}
Finally, let us mention that all coordinates of the vector of temporal means $m^N$ are expected to converge to the asymptotic spatio-temporal mean $m$ defined in \eqref{def:mean_activity}. This statement is formalized in Equation \eqref{eq:concentration:m}.

\modif{The present paper builds upon the modeling framework and some technical results of \cite{Chevallier2024inferring}. Nevertheless, except from the preliminary results detailed above, the only additional dependence is \cite[Lemma 3.8]{Chevallier2024inferring} which is included in Appendix \ref{app:covariance:lemma}. All other proof strategies and arguments are entirely new.}

\subsection{Asymptotics for the simultaneous covariance matrix}
Let us denote $v^N$ the variance vector given by
\begin{equation}\label{eq:def:variance:vector}
v^N=(m^N_1(1-m^N_1),\ldots,m^N_N(1-m^N_N)).
\end{equation}
According to the end of previous section, we expect that $v^N$ is close to $m(1-m) 1_N$. This is formalized in Equation \eqref{eq:concentration:v:main} below.

Let us denote $\Sigma^{(0)} = \E_\theta[(X_0 - m^N)(X_0 - m^N)^{\top}]$ the simultaneous covariance matrix whose $(i,j)$-entry is 
\begin{equation}
\label{def:sigma_0}
\Sigma^{(0)}_{ij} = \cov_{\theta}(X_{i,0},X_{j,0}), \ i,j\in [N].
\end{equation}
The main objective of this section is to provide an asymptotic approximation of the matrix $\Sigma^{(0)}$ as $N\to\infty$ (see Proposition \ref{prop:concentration:asymp_approximation_simulta_cov_matrix}).

The reason to study the simultaneous covariance matrix is that it serves as a proxy for getting the asymptotics of the 1-lagged
covariance matrix $\Sigma^{(1)}$. Indeed, we have the following relation:

\begin{proposition}\label{prop:Sigma1:Sigma0}
    $\Sigma^{(1)}=(1-\lambda)A^N \Sigma^{(0)}$.
\end{proposition}
\begin{proof}
    By conditioning with respect to $\mathcal{F}_0$ in Equation \eqref{eq:Sigma1:matrix:notation} and then using \eqref{eq:centered:p:theta}, one has
    \begin{equation*}
        \Sigma^{(1)} = \E_{\theta}\left[ (p_\theta(X_0) - m^N)(X_0 - m^N)^{\top} \right] = \E_{\theta}\left[ (1-\lambda)A^N(X_0 - m^N)(X_0 - m^N)^{\top} \right],
    \end{equation*}
    which ends the proof.
\end{proof}

For a $N$-dimensional vector $u=(u_1,\ldots, u_N)$, we write $\diag(u)$ to denote the $N$-by-$N$ diagonal matrix having value $u_i\in \R$ associated with its $i$-th diagonal entry, i.e. $\diag(u)_{ii} = u_i$.

For a $N$-by-$N$ matrix $M$, we denote  $\d0(M)$ the same matrix where the diagonal terms are replaced by zeros, i.e. $\d0(M)_{ij} = M_{ij} \mathbf{1}_{i\neq j}$. 
The simultaneous covariance matrix $\Sigma^{(0)}$ satisfies the following Stein-type matrix equation.
\begin{proposition}\label{prop:Sigma:0:matrix:equation}
    The simultaneous covariance matrix $\Sigma^{(0)}$ is equal to $v^N$ on the diagonal and satisfies $\Sigma^{(0)} = (1-\lambda)^2 A^N\Sigma^{(0)}(A^N)^{\top}$ out of the diagonal, i.e.
    \begin{equation*}
        \Sigma^{(0)}=(1-\lambda)^2 \d0\left( A^N\Sigma^{(0)}(A^N)^{\top} \right) + \diag(v^N).
    \end{equation*}
\end{proposition}
\begin{proof}    
    Let $i\in [N]$. Under the stationary distribution, $X_{i,0}$ is a Bernoulli variable with parameter $m^N_i$. Hence, $\Sigma^{(0)}_{ii} = \var_{\theta}(X_{i,0}) = m^N_i(1-m^N_i) = v^N_i$.

    Let $i,j\in [N]$ with $i\neq j$. By conditional independence between $X_{i,0}$ and $X_{j,0}$ given $X_{-1}$ and stationarity of the process, we have that 
    \begin{equation*}
        \Sigma^{(0)}_{ij} = \cov_{\theta}(X_{i,0},X_{j,0}) = \cov_{\theta}(p_{\theta,i}(X_{-1}),p_{\theta,j}(X_{-1})) = \cov_{\theta}(p_{\theta,i}(X_{0}),p_{\theta,j}(X_{0})).
    \end{equation*}
    Yet, the centered version of $p_{\theta,i}(X_{0})$ is $p_{\theta,i}(X_{0}) - m^N_i$ so that using Equation \eqref{eq:centered:p:theta} it is easy to get
    \begin{equation*}
        \Sigma^{(0)}_{ij} = (1-\lambda)^2 \mathbb{E}\left[ A^N (X_0 - m^N)(X_0 - m^N)^{\top} (A^N)^{\top} \right]_{ij} = (1-\lambda)^2 [A^N\Sigma^{(0)}(A^N)^{\top}]_{ij}.
    \end{equation*}
\end{proof}

Let us denote $a^N = 1_N (1_N)^{\top} - I_N$ the $N$-by-$N$ matrix full of ones except on the diagonal. Here is our main asymptotic result regarding the simultaneous covariance matrix.
\modif{
\begin{proposition}\label{prop:concentration:asymp_approximation_simulta_cov_matrix}
    There exists a constant $K>0$ depending only on $\lambda, p$ and $r_+$ such that for all $\delta\in (0,1]$ and $N\geq 2$,
    \begin{equation*}
        \P\left(\left\vvvert \Sigma^{(0)} - \left( \diag(v^N) + \frac{(1-\lambda)^2 p^2 m(1-m)}{1 - D^2} \frac{a^N}{N} \right) \right\vvvert_{\max} \geq K \sqrt{\frac{\log(N/\delta)}{N^3}} \right)\leq \delta,
    \end{equation*}
where the real values $m$ and $D$ are defined in \eqref{def:mean_activity}.
\end{proposition}
}

The proof of Proposition \ref{prop:concentration:asymp_approximation_simulta_cov_matrix}, postponed to Section \ref{sec:proof:simult:cov:matrix}, is interesting in its own right. 
It relies on 
\modif{the Stein-type matrix equation satisfied by $\Sigma^{(0)}$}
(see Proposition \ref{prop:Sigma:0:matrix:equation}). To exploit this structure, we introduce the \emph{vectorization} operator and the \emph{Kronecker product}, along with their useful properties, in Section \ref{sec:proof:simult:cov:matrix}.

\subsection{Proofs}

Finally, the last main ingredients to describe the asymptotics of $\Sigma^{(1)}$ are the asymptotics of the vector of row-wise and column-wise sums of $A^N$, denoted by $L^N = A^N 1_N$ and $C^N=(A^N)^{\top}1_N$ respectively, and the asymptotics of $v^N$.

\begin{proposition}\label{prop:asymptotics:concentration}
    There exists a constant $K>0$ depending only on $\lambda$ such that for all $\delta\in (0,1]$ and $N\geq 2$,
    \begin{gather}
        \P\left(\left\| L^{N} - p(r_+ - r_-)1_N  \right\|_{\infty} \geq K\sqrt{\frac{\log(N/\delta)}{N}}\right)\leq \delta,
        \label{eq:concentration:L:main}\\
        \P\left(\left\| C^{N} - p(1_{\mathcal{P}_+} - 1_{\mathcal{P}_-})  \right\|_{\infty} \geq K\sqrt{\frac{\log(N/\delta)}{N}}\right)\leq \delta
        \label{eq:concentration:C:main}\\
        \P\left(\left\| v^N - m(1-m)1_N   \right\|_{\infty} \geq K \sqrt{\frac{\log(N/\delta)}{N}}  \right)\leq \delta.
        \label{eq:concentration:v:main}
    \end{gather}
\end{proposition}
Proposition \ref{prop:asymptotics:concentration} is a particular case of Lemma \ref{lem:asymptotics:concentration} whose proof is given in Section \ref{sec:proof:max:inequality}.

In the remaining of this section, we prove the structural results, namely, Theorem \ref{thm:concentration:Sigma} and Corollaries \ref{cor:concentration:Sigma:operator:norm} and \ref{cor:concentration:sigma:sup:norm}. We start with the proof of Theorem \ref{thm:concentration:Sigma}. 

\begin{proof}[Proof of Theorem \ref{thm:concentration:Sigma}]
Let $0<\delta\leq 1$. First of all, remind that $\Sigma^{(1)}=(1-\lambda)A\Sigma^{(0)}$ (see Proposition \ref{prop:Sigma1:Sigma0}). By Proposition \ref{prop:concentration:asymp_approximation_simulta_cov_matrix} and using \eqref{eq:max:norm:submultiplicative} with the fact that $\vvvert A^N \vvvert_{\infty} \leq 1$, we have,
\begin{equation}\label{eq:control:Sigma1}
    \P\Big(\left\vvvert \Sigma^{(1)} - A^N \left( \frac{c_1}{m(1-m)}\diag(v^N) + \frac{c_2}{p(r_+ - r_-)} \frac{a^N}{N} \right) \right\vvvert_{\max} \geq K \sqrt{\frac{\log(3N/\delta)}{N^3}} \Big)\leq \frac{\delta}{3}.
\end{equation}

On the one hand, Equation \eqref{eq:concentration:v:main} implies that
$$
\P\left(\left\vvvert \diag(v^N) - m(1-m)I_N \right\vvvert_{1} \geq K \sqrt{\frac{\log(3N/\delta)}{N}} \right)\leq \delta/3,
$$
so that using \eqref{eq:max:norm:submultiplicative} with the fact that $\vvvert A^N \vvvert_{\max} \leq 1/N$, we get
\begin{equation}\label{eq:control:AN:vN}
    \P\left(\vvvert A^N\diag(v^N) - m(1-m)A^N \vvvert_{\max} \geq K \sqrt{\frac{\log(3N/\delta)}{N^3}} \right)\leq \delta/3,
\end{equation}

On the other hand, for all $i,j\in [N]$, $(A^Na^N)_{ij} = \sum_{k=1 : k\neq j}^{n} A^N_{ik} = L^N_i - N^{-1}\theta_{ij}$. Hence, Equation \eqref{eq:concentration:L:main} implies that
\begin{equation}\label{eq:control:AN:aN}
    \P\left(\left\vvvert A^N \frac{a^N}{N} - \frac{p(r_+ - r_-)}{N} 1_N 1_N^{\top} \right\vvvert_{\max} \geq K \sqrt{\frac{\log(3N/\delta)}{N^3}} + N^{-2}\right)\leq \delta/3.
\end{equation}

Hence, combining Equations \eqref{eq:control:Sigma1}-\eqref{eq:control:AN:aN} and adjusting the constant $K$, we deduce that  
 $$
    \P\left(\left\vvvert \Sigma^{(1)} - \left( c_1 A^N + \frac{c_2}{N} 1_N 1_N^{\top} \right) \right\vvvert_{\max} \geq K\Bigg(\sqrt{\frac{\log(3N/\delta)}{N^3}}+N^{-2}\Bigg) \right) \leq \delta.
    $$
The result follows by combining the above inequality with the fact that $\log(3N/\delta)\geq \log(6)\geq 1/2\geq 1/N$ since $0<\delta\leq 1$ and $N\geq 2$ 
(this kind of argument will be used several times throughout the paper to simplify the lower bounds at the cost of adjusting the constant $K$, and will usually be omitted in subsequent proofs).
\end{proof}

Now we prove Corollary \ref{cor:concentration:Sigma:operator:norm}.

\begin{proof}[Proof of Corollary \ref{cor:concentration:Sigma:operator:norm}]
    Let $0<\delta\leq 1$. Let us first decompose
    $$
    \Sigma^{(1)} - \left( c_1 \E\left[ A^N \right] + \frac{c_2}{N} 1_N 1_N^\top \right) = c_1 \left( A^N - \E\left[ A^N \right] \right) + \Sigma^{(1)} - \left( c_1 A^N + \frac{c_2}{N} 1_N 1_N^\top \right).
    $$
    
    On the one hand, \cite[Theorem 4.4.3]{vershynin2026high} applied to the matrix $A^N - \E[ A^N ]$ with $t =\sqrt{(\log(2/\delta)}$ (and the value $K$ there can be taken as $(N\sqrt{\log(2)})^{-1}$ whatever $p$ is) gives
    $$
    \P\left(\left\vvvert A^N - \E\left[ A^N \right] \right\vvvert_{2} \geq K \left[ \frac{1}{\sqrt{N}} + \frac{1}{N} \sqrt{\log(2/\delta)} \right]\right)\leq \delta.
    $$

    On the other hand, remark that $\vvvert M \vvvert_{2} \leq N \vvvert M \vvvert_{\max}$ for any $N$-by-$N$ matrix so that Theorem \ref{thm:concentration:Sigma} implies
    $$
    \P\left(\left\vvvert \Sigma^{(1)} - \left( c_1 A^N + \frac{c_2}{N} 1_N 1_N^\top \right) \right\vvvert_{2} \geq K \sqrt{\frac{\log(N/\delta)}{N}} \right) \leq \delta.
    $$

    The union bound concludes the proof upon adjusting the constant $K$ to get a simpler form for the lower bound.
\end{proof}

Finally, we prove Corollary \ref{cor:concentration:sigma:sup:norm}.

\begin{proof}[Proof of Corollary \ref{cor:concentration:sigma:sup:norm}]
By Theorem \ref{thm:concentration:Sigma} and the definition of the vector $\sigma^{\rm ag} =(\Sigma^{(1)})^{\top}1_N$, one can deduce that
\begin{equation*}
    \P\left(\left\| \sigma^{\rm ag} - \left( c_1 C^N + c_2 1_N \right) \right\|_{\infty} \geq K\sqrt{\frac{\log(N/\delta)}{N}} \right) \leq \delta,
\end{equation*}
and the desired result follows from \eqref{eq:concentration:C:main} and the union bound.
\end{proof}

\section{Exact recovery for the aggregated method}
\label{sec:proof:exact:recovery:aggregated:method}

In this section, we prove Theorem \ref{thm:conver:hat:sigma:ag:towards:bar:sigma:ag}
and Corollary \ref{cor:extact_recovery}. We start proving Theorem  \ref{thm:conver:hat:sigma:ag:towards:bar:sigma:ag}. 

\begin{proof}[Proof of Theorem \ref{thm:conver:hat:sigma:ag:towards:bar:sigma:ag}]
Using first the triangle inequality and then the union bound, we obtain 
$$
\P\left(\|\widehat{\sigma}^{\rm ag}-\overline{\sigma}^{\rm ag}\|_{\infty} \geq \varepsilon_1+\varepsilon_2\right)\leq \P\left(\|\widehat{\sigma}^{\rm ag}-\sigma^{\rm ag}\|_{\infty} \geq \varepsilon_1\right)+\P\left(\|\sigma^{\rm ag}-\overline{\sigma}^{\rm ag}\|_{\infty} \geq \varepsilon_2\right),
$$
for all $\varepsilon_1,\varepsilon_2>0$. 
To conclude, take $\varepsilon_1=K_1N\log(TN)T^{-1/2}$ with $K_1>0$ the constant given in Theorem \ref{thm:conv:hat:sigma:ag:towards:sigma:ag}  
and $\varepsilon_2=K\sqrt{\log(N^2)/N}$ with $K>0$ the constant given in Corollary \ref{cor:concentration:sigma:sup:norm}, and apply Theorem \ref{thm:conv:hat:sigma:ag:towards:sigma:ag}  and Corollary \ref{cor:concentration:sigma:sup:norm} with $\delta=N^{-1}$.

\end{proof}

Now, we prove Corollary \ref{cor:extact_recovery}.

\begin{proof}[Proof of Corollary \ref{cor:extact_recovery}]
Throughout the proof, denote $F^N_{\varepsilon}=\{\| \sigma^{\rm ag} - 
    \overline{\sigma}^{\rm ag} 
\|_{\infty}\leq \varepsilon\}$ and $G^N_{\varepsilon}=\{\|\widehat{\sigma}^{\rm ag}-\sigma^{\rm ag }\|_{\infty}\leq \varepsilon\}$, for each $\varepsilon>0$. 
Combining Corollary \ref{cor:concentration:sigma:sup:norm} (with $\delta=N^{-1}$) and Theorem \ref{thm:conv:hat:sigma:ag:towards:sigma:ag}, one can check that 
\begin{equation}
\label{proof:cor:extact_recovery:eq:1}    
\P(G^N_{\varepsilon}\cap F^N_{\varepsilon})\geq 1-(N^{-1}+K_2T^{-1}),
\end{equation}
whenever $\varepsilon\geq \max\{K\sqrt{\log(N^2)/N},K_1N\log(TN)T^{-1/2}\}:=\varepsilon_{N,T}$.
Moreover, it follows from the triangle inequality that 
that 
\begin{equation}
\label{proof:cor:extact_recovery:eq:1.5}  
F^N_{\varepsilon}\cap G^N_{\varepsilon}\subseteq \{\|\widehat{\sigma}^{\rm ag}-\overline{\sigma}^{\rm ag}\|\leq 2\varepsilon \}:=E^N_{2\varepsilon}.    
\end{equation}

\smallskip
{\bf Claim 1.}  Recall that $(\widehat{\mathcal{P}}^{\rm ag,km}_+,\widehat{\mathcal{P}}^{\rm ag,km}_-)$ denotes the two clusters returned by $k$-means clustering applied to $\widehat{\sigma}^{\rm ag}$ with $k=2$. Then, $(\widehat{\cal P}_+^{\rm ag, km},\widehat{\cal P}_-^{\rm ag, km})=({\cal P}_+,{\cal P}_-)$, on the event $E^N_{\varepsilon}$ for all $\varepsilon\leq c_1pr_{\rm min}/(1+r_{\rm min})$.

\smallskip
{\it Proof of Claim 1.} Lemma \ref{lem:structural:properties:kmeans} implies that $({\cal P}_+,{\cal P}_-)$ is the unique global minimizer of the $k$-means criterion on the event $E^N_{\varepsilon}$ for all $\varepsilon\leq c_1pr_{\rm min}/(1+r_{\rm min})$. Hence, under this event, we must have $(\widehat{\cal P}_+^{\rm ag, km},\widehat{\cal P}_-^{\rm ag, km})=({\cal P}_+,{\cal P}_-)$, concluding the proof.  

\smallskip
{\bf Claim 2.} Let $(\widehat{\mathcal{P}}^{\rm ag,lkm}_+,\widehat{\mathcal{P}}^{\rm ag,lkm}_-)$ be the two clusters returned by Lloyd's $k$-means algorithm applied to $\widehat{\sigma}^{\rm ag}$ with $k=2$, initialized with the means $m_1 = \min_j \widehat{\sigma}^{\rm ag}_j$ and $m_2 = \max_j \widehat{\sigma}^{\rm ag}_j$. Let $\varepsilon< pc_1/2$. Then, on the event $E^N_{\varepsilon}$, we also have $(\widehat{\mathcal{P}}^{\rm ag,lkm}_-,\widehat{\mathcal{P}}^{\rm ag,lkm}_+)=({\cal P}_+,{\cal P}_-).$  

\smallskip
{\it Proof of Claim 2.}
Under the assumption that $\varepsilon < pc_1$, one can easily check that $\max_{j\in {\cal P}_-}\widehat{\sigma}^{\rm ag}_j\leq \overline{\sigma}_-+\varepsilon<\overline{\sigma}_+-\varepsilon\leq \min_{j\in {\cal P}_+}\widehat{\sigma}^{\rm ag}_j$ which, in turn, implies that $m_1=\min_{j\in {\cal P}_-}\widehat{\sigma}^{\rm ag}_j$ and $m_2=\max_{j\in {\cal P}_+}\widehat{\sigma}^{\rm ag}_j$. Furthermore, since $\varepsilon < pc_1/2$, we have $\overline{\sigma}_+-\overline{\sigma}_->4\varepsilon$. As a consequence, $|\widehat{\sigma}^{\rm ag}_i-m_1|\leq 2\varepsilon<\overline{\sigma}_+-\overline{\sigma}_--2\varepsilon\leq |\widehat{\sigma}^{\rm ag}_i-m_2|$ for all $i\in {\cal P}_-$. Arguing similarly, we also have that $|\widehat{\sigma}^{\rm ag}_i-m_2|< |\widehat{\sigma}^{\rm ag}_i-m_1|$ for all $i\in {\cal P}_+$. This implies that the assignment step in the first iteration of Lloyd's $k$-means algorithm returns  the partition made of $S_1 = \mathcal{P}_-$ and $S_2 = \mathcal{P}_+$.
Finally, using similar arguments as above, we can also check that this partition 
is stable and so $\widehat{\mathcal{P}}^{\rm ag,lkm}_- = \mathcal{P}_-$ and $\widehat{\mathcal{P}}^{\rm ag,lkm}_+ = \mathcal{P}_+$, concluding the proof. 

\smallskip
{\bf Claim 3.} Recall that  $(\widehat{\mathcal{P}}^{\rm ag,mt}_+, \widehat{\mathcal{P}}^{\rm ag,mt}_-)$ denotes the two clusters returned by mean threshold clustering applied to $\widehat{\sigma}^{\rm ag}$. Let $\varepsilon < c_1pr_{\min}$. Then, on the event $E^N_{\epsilon}$, we must have $(\widehat{\mathcal{P}}^{\rm ag,mt}_+, \widehat{\mathcal{P}}^{\rm ag,mt}_-)=({\cal P}_+,{\cal P}_-)$. 

\smallskip
{\it Proof of Claim 3.}
First, denote $\tilde{\sigma}^{\rm ag}=\widehat{\sigma}^{\rm ag}-N^{-1}\langle \widehat{\sigma}^{\rm ag},1_N\rangle 1_N$ the centered version of the vector $\widehat{\sigma}^{\rm ag}$ and observe that 
    \begin{equation}
    \label{proof:cor:extact_recovery:eq:2}
      (\widehat{\mathcal{P}}^{\rm ag,mt}_+, \widehat{\mathcal{P}}^{\rm ag,mt}_-)=(\{i\in [N]: \tilde{\sigma}^{\rm ag}_i\geq 0\}, \{i\in [N]: \tilde{\sigma}^{\rm ag}_i< 0\}).   
    \end{equation}
 Next, check using first the triangle inequality and then Hölder inequality that
\begin{align}
 \|\tilde{\sigma}^{\rm ag} - 2c_1p(r_-1_{{\cal P}_+}-r_-1_{{\cal P}_-})\|_{\infty}&=\|\tilde{\sigma}^{\rm ag}-\overline{\sigma}^{\rm ag} -N^{-1}\langle \tilde{\sigma}^{\rm ag}-\overline{\sigma}^{\rm ag},1_N\rangle 1_N\|_{\infty}\nonumber \\
 &\leq \|\tilde{\sigma}^{\rm ag}-\overline{\sigma}^{\rm ag}\|_{\infty}+N^{-1}|\langle \tilde{\sigma}^{\rm ag}-\overline{\sigma}^{\rm ag},1_N\rangle|\nonumber \\
 &\leq \|\tilde{\sigma}^{\rm ag}-\overline{\sigma}^{\rm ag}\|_{\infty}+N^{-1}\|\tilde{\sigma}^{\rm ag}-\overline{\sigma}^{\rm ag}\|_{\infty}\|1_N\|_1\nonumber \\
  \label{proof:cor:extact_recovery:eq:3}
 &=2\|\tilde{\sigma}^{\rm ag}-\overline{\sigma}^{\rm ag}\|_{\infty}.
\end{align}
Denoting $r_{\rm min}=r_+\wedge r_-$, we then deduce from \eqref{proof:cor:extact_recovery:eq:2} and \eqref{proof:cor:extact_recovery:eq:3} that    
\begin{align*}
\{ (\widehat{\cal P}_+^{\rm ag, mt},\widehat{\cal P}_-^{\rm ag, mt})\neq ({\cal P}_+,{\cal P}_-)\}&=\{\|\tilde{\sigma}^{\rm ag} - 2c_1p(r_-1_{{\cal P}_+}-r_-1_{{\cal P}_-})\|_{\infty}\geq 2c_1pr_{\rm min} \}\\
&\subseteq
\|\widehat{\sigma}^{\rm ag} -\overline{\sigma}^{\rm ag}\|_{\infty}\geq c_1pr_{\rm min} \},
\end{align*}
which, in turn, implies that 
\[
E^N_{\varepsilon}\subset \{\|\widehat{\sigma}^{\rm ag} -\overline{\sigma}^{\rm ag}\|_{\infty}< c_1pr_{\rm min} \}\subseteq \{ (\widehat{\cal P}_+^{\rm ag, mt},\widehat{\cal P}_-^{\rm ag, mt})=({\cal P}_+,{\cal P}_-)\}, 
\]
for all $\epsilon<c_1p r_{\rm min},$ concluding the proof.

\smallskip
Hence, we have just proved that exact recovery is achieved for all clustering procedures on $E^N_{\varepsilon}$
as long as $\varepsilon<\min\{c_1pr_{\rm min},pc_1/2, c_1pr_{\rm min}/(1+r_{\rm min}) \}=c_1pr_{\rm min}/(1+r_{\rm min})$ (since $r_{\rm min}\leq 1/2$).  

Finally, letting $N_0=\{N\geq 2:K\sqrt{\log(N^2)/N}<c_1pr_{\rm min}/(2(1+r_{\rm min}))\}$
and verifying that $\varepsilon_{N,T}<c_1pr_{\rm min}/(2(1+r_{\rm min}))$ for all $(N,T)$ such that $N\geq N_0$ and
\[
T\geq \Bigg(\frac{4(1+r_{\rm min})K_1 N\log(TN)}{c_1pr_{\rm min}}\Bigg)^2,
\]
the result follows from \eqref{proof:cor:extact_recovery:eq:1} with $\varepsilon=\varepsilon_{N,T}$.

\end{proof}

\section{Misclassification rate for aggregated and spectral methods}
\label{sec:proof:missclassification:rate}

In this section, we prove Theorems \ref{thm:conver:hat:sigma:ag:towards:bar:sigma:ag:in:l2}, \ref{thm:missclasification:rate:aggregated:method}, \ref{thm:estimation:spectral:method} and \ref{thm:missclasification:rate:spectral:method}, besides the auxiliary result stated in Lemma \ref{lem:l2norm:controls:MR} below.
We start by proving Theorem \ref{thm:conver:hat:sigma:ag:towards:bar:sigma:ag:in:l2}. 

\begin{proof}[Proof of Theorem \ref{thm:conver:hat:sigma:ag:towards:bar:sigma:ag:in:l2}]
Recall that $\widehat{\sigma}^{\rm ag}=(\widehat{\Sigma}^{(1)})^{\top}1_N$ and $\overline{\sigma}^{\rm ag}=(\overline{\Sigma}^{(1)})^{\top}1_N$. Using that $\|1_N\|_2=N^{1/2}$ and the triangle inequality, we have that 
\begin{align*}
\|\widehat{\sigma}^{\rm ag}-\overline{\sigma}^{\rm ag}\|_2&\leq \vvvert (\widehat{\Sigma}^{(1)}-\overline{\Sigma}^{(1)})^{\top}\vvvert_2 \|1_N\|_2\\
&= \vvvert \widehat{\Sigma}^{(1)}-\overline{\Sigma}^{(1)}\vvvert_2 N^{1/2}\\
&\leq \Big(\vvvert \widehat{\Sigma}^{(1)}-\Sigma^{(1)}\vvvert_2+\vvvert \Sigma^{(1)}-\overline{\Sigma}^{(1)}\vvvert_2\Big)N^{1/2}.
\end{align*}
Combining the above inequality and Theorems \ref{cor:concentration:Sigma:operator:norm} (with $\delta=N^{-1}$) and \ref{thm:error:between:empirical:Sigma1:and:Sigma1}, the result follows. 
\end{proof}

Next, we prove Theorem \ref{thm:missclasification:rate:aggregated:method}.

\begin{proof}[Proof of Theorem \ref{thm:missclasification:rate:aggregated:method}]

The first step in the proof is to show that
\begin{equation}
\label{eq0:proof:thm:missclasification:rate:aggregated:method}
\operatorname{MR}(\widehat{\mathcal{P}}^{\mathrm{ag, \bullet}}_+, \widehat{\mathcal{P}}^{\mathrm{ag, \bullet}}_-)\leq \frac{K_{\bullet} \| \widehat{\sigma}^{\rm ag} - \overline{\sigma}^{\rm ag} \|_2^2}{|\overline{\sigma}_+-\overline{\sigma}_-|^2 \cdot (r_+^N\wedge r_-^N)^2} N^{-1},    
\end{equation}
with $K_{\rm mt} = 2$ and $K_{\rm km} = 20$.

First, Inequality \eqref{eq0:proof:thm:missclasification:rate:aggregated:method} for the estimated communities $(\widehat{{\cal P}}_+^{\rm ag, mt},\widehat{{\cal P}}_-^{\rm ag, mt})$ is a direct consequence of Lemma \ref{lem:l2norm:controls:MR} with $\overline{\sigma}=\overline{\sigma}^{\rm ag}$, $b_+=\overline{\sigma}_+$, $b_-=\overline{\sigma}_-$ and $\widehat{\sigma}=\widehat{\sigma}^{\rm ag}$.

Next, we prove \eqref{eq0:proof:thm:missclasification:rate:aggregated:method} for the estimated communities $(\widehat{{\cal P}}_+^{\rm ag, km},\widehat{{\cal P}}_-^{\rm ag, km})$.
To that end, denote $\widehat{\sigma}^{\rm km}_a=|\widehat{\cal P}_a^{\rm ag, km}|^{-1}\sum_{i\in \widehat{\cal P}_a^{\rm ag, km}}\widehat{\sigma}^{\rm ag}_i$, for each $a\in\{-,+\}$, and consider the vector 
\[
\widehat{\sigma}^{\rm km}=\widehat{\sigma}^{\rm km}_+ 1_{\widehat{\cal P}_+^{\rm ag, km}}
+\widehat{\sigma}^{\rm km}_- 1_{\widehat{\cal P}_-^{\rm ag, km}}.
\]
By the definition of $k$-means clustering, we have that for any partition $(S,S^c)$ of $[N]$, the following inequality holds:
\begin{equation}
\label{eq1:proof:thm:missclasification:rate:aggregated:method}
\|\widehat{\sigma}^{\rm ag}-\widehat{\sigma}^{\rm km}\|^2_2\leq \|\widehat{\sigma}^{\rm ag}-(\widehat{\mu}_S1_S+\widehat{\mu}_{S^c}1_{S^c})\|^2_2,   \end{equation}
where $\widehat{\mu}_V=|V|^{-1}\sum_{i\in V}\widehat{\sigma}^{\rm ag}_i$ for any subset $V \subseteq [N]$. Choosing $(S,S^c)=({\cal P}_+,{\cal P}_-)$ in \eqref{eq1:proof:thm:missclasification:rate:aggregated:method} and then using  triangle and Jensen inequalities, we obtain that 
\begin{align}
\label{eq2:proof:thm:missclasification:rate:aggregated:method}
\|\widehat{\sigma}^{\rm ag}-\widehat{\sigma}^{\rm km}\|^2_2&\leq \|\widehat{\sigma}^{\rm ag}-(\widehat{\mu}_{{\cal P}_+}1_{{\cal P}_+}+\widehat{\mu}_{{\cal P}_-}1_{{\cal P}_-})\|^2_2 \nonumber \\
&\leq 2\Big(\|\widehat{\sigma}^{\rm ag}-\overline{\sigma}^{\rm ag}\|^2_2+\|\overline{\sigma}^{\rm ag}-(\widehat{\mu}_{{\cal P}_+}1_{{\cal P}_+}+\widehat{\mu}_{{\cal P}_-}1_{{\cal P}_-})\|^2_2\Big)\nonumber \\
&=2\Big(\|\widehat{\sigma}^{\rm ag}-\overline{\sigma}^{\rm ag}\|^2_2+
|{\cal P}_+|(\overline{\sigma}_+-\widehat{\mu}_{{\cal P}_+})^2+|{\cal P}_-|(\overline{\sigma}_- - \widehat{\mu}_{{\cal P}_-})^2\Big)
\end{align}
Now, using once again Jensen inequality, we deduce for each $a\in\{-,+\}$ that
\begin{equation*}
|{\cal P}_a| (\widehat{\mu}_{{\cal P}_a} - \overline{\sigma}_a)^2 \leq 
|{\cal P}_a| \Big(|{\cal P}_a|^{-1}\sum_{i\in {\cal P}_a}(\widehat{\sigma}^{\rm ag}_i-\overline{\sigma}_a)^2\Big) = 
\sum_{i\in {\cal P}_a}(\widehat{\sigma}^{\rm ag}_i-\overline{\sigma}_a)^2,
\end{equation*}
so that it follows from \eqref{eq2:proof:thm:missclasification:rate:aggregated:method} that
\begin{align*}
\|\widehat{\sigma}^{\rm ag}-\widehat{\sigma}^{\rm km}\|^2_2&\leq 4\|\widehat{\sigma}^{\rm ag}-\overline{\sigma}^{\rm ag}\|^2_2.
\end{align*}
Using Jensen inequality a third time and then the above inequality, we deduce that  
\begin{align}
\label{eq3:proof:thm:missclasification:rate:aggregated:method}
\|\widehat{\sigma}^{\rm km}-\overline{\sigma}^{\rm ag}\|^2_2&\leq 2\Big(\|\widehat{\sigma}^{\rm ag}-\widehat{\sigma}^{\rm km}\|^2_2+\|\widehat{\sigma}^{\rm ag}-\overline{\sigma}^{\rm ag}\|^2_2\Big)\leq 10\|\widehat{\sigma}^{\rm ag}-\overline{\sigma}^{\rm ag}\|^2_2.
\end{align}

Now, note that $N^{-1} \sum_{j=1}^N \widehat{\sigma}^{\rm km}_j = 
N^{-1}\left( \widehat{\sigma}^{\rm km}_+ |\widehat{\cal P}_+^{\rm ag, km}| +
\widehat{\sigma}^{\rm km}_- |\widehat{\cal P}_-^{\rm ag, km}| \right) \in ( \widehat{\sigma}^{\rm km}_-, \widehat{\sigma}^{\rm km}_+)$. This remark implies that the estimated partition defined in Lemma \ref{lem:l2norm:controls:MR} from the estimator $\widehat{\sigma}^{\rm km}$ is in fact $(\widehat{\cal P}_+^{\rm ag, km},\widehat{\cal P}_-^{\rm ag, km})$. Hence, Lemma \ref{lem:l2norm:controls:MR} (with $\overline{\sigma}=\overline{\sigma}^{\rm ag}$, $b_+=\overline{\sigma}_+$, $b_-=\overline{\sigma}_-$ and $\widehat{\sigma}=\widehat{\sigma}^{\rm km}$) and \eqref{eq3:proof:thm:missclasification:rate:aggregated:method} imply that
\[
\operatorname{MR}(\widehat{\mathcal{P}}^{\mathrm{ag, km}}_+, \widehat{\mathcal{P}}^{\mathrm{ag, km}}_-)\leq \frac{2 \| \widehat{\sigma}^{\rm km} - \overline{\sigma}^{\rm ag} \|_2^2}{|\overline{\sigma}_+-\overline{\sigma}_-|^2 \cdot (r_+^N\wedge r_-^N)^2} N^{-1}\leq \frac{20 \| \widehat{\sigma}^{\rm ag} - \overline{\sigma}^{\rm ag} \|_2^2}{|\overline{\sigma}_+-\overline{\sigma}_-|^2 \cdot (r_+^N\wedge r_-^N)^2} N^{-1},
\]
proving \eqref{eq0:proof:thm:missclasification:rate:aggregated:method} for $(\widehat{{\cal P}}_+^{\rm ag, km},\widehat{{\cal P}}_-^{\rm ag, km})$.

Finally, using \eqref{eq0:proof:thm:missclasification:rate:aggregated:method}, observe that for all $N\geq N_1:=\lceil 2/(r_+\wedge r_-)\rceil$ (which ensures that $r_+^N\wedge r_-^N\geq (r_+\wedge r_-)/2$), we have 
\[
\operatorname{MR}(\widehat{\mathcal{P}}^{\mathrm{ag},\bullet}_+, \widehat{\mathcal{P}}^{\mathrm{ag}, \bullet}_-)\leq \frac{4K_{\bullet}\| \widehat{\sigma}^{\rm ag} - \overline{\sigma}^{\rm ag} \|_2^2}{|\overline{\sigma}_+-\overline{\sigma}_-|^2 \cdot (r_+\wedge r_-)^2} N^{-1}.
\]
Therefore, the result follows by combining the above inequality with Theorem \ref{thm:conver:hat:sigma:ag:towards:bar:sigma:ag:in:l2}.

\end{proof}

Now, we prove Theorem \ref{thm:estimation:spectral:method}.

\begin{proof}[Proof of Theorem \ref{thm:estimation:spectral:method}]

Recall that $\widecheck{v}$ is a leading right singular vector of $\widehat{\Sigma}^{(1)}$ and that $\widecheck{{\cal P}}$ is a preliminary community defined by \eqref{eq:def:preliminary:partition}. More generally, the notation $\widecheck{\eta}$ highlights an estimator of $\eta$ for which the sign ambiguity must be resolved. 
Moreover, recall that $\overline{\sigma}^{\rm sp}$ is a singular vector of the matrix $\overline{\Sigma}^{(1)}$, associated with the largest singular value $\sigma_1(\overline{\Sigma}^{(1)})$ defined in \eqref{def:leading:singular:value:bar:sigma1}. Assume that $r^N_+ \wedge r_-^N >0$, so that $\sigma_1(\overline{\Sigma}^{(1)})>0$.

Denote $E:=\{\vvvert \widehat{\Sigma}^{(1)}-\overline{\Sigma}^{(1)} \vvvert_{2}<(1-1/\sqrt{2}) \sigma_1(\overline{\Sigma}^{(1)})\}$. All the proof takes place on the event $E$. By Wedin's $\sin\Theta$ theorem \cite[Theorem 2.9 and Inequality 2.26a]{Chen2021Spectral}, we have that $\widecheck{v}$ is close to $\overline{\sigma}^{\rm sp}$ (up to its sign) in the sense that:
\begin{equation}
\label{ineq:1:proof:thm:community:detection}
\|\widecheck{v}-\overline{\sigma}^{\rm sp}\|_2\wedge \|\widecheck{v}+\overline{\sigma}^{\rm sp}\|_2\leq \frac{2\vvvert \widehat{\Sigma}^{(1)}-\overline{\Sigma}^{(1)} \vvvert_{2}}{\sigma_1(\overline{\Sigma}^{(1)})}.
\end{equation}

Let $\alpha\in\{-1,1\}$ such that $\|\widecheck{v}-\overline{\sigma}^{\rm sp}\|_2\wedge \|\widecheck{v}+\overline{\sigma}^{\rm sp}\|_2=\|\alpha \widecheck{v}-\overline{\sigma}^{\rm sp}\|_2$.
The sign $\alpha$ is unknown and the main objective is to prove that it can be estimated using $\hat{\sigma}^{\rm ag}$.

\medskip
{\bf Claim 1.} Whatever the value of $\alpha\in\{-1,+1\}$, $| |\widecheck{{\cal P}}| - |{\cal P}_\alpha| | \leq N \operatorname{MR}_\alpha$,
where $\operatorname{MR}_+ = \operatorname{MR}(\widecheck{{\cal P}},\,[N] \setminus \widecheck{\cal P})$ and $\operatorname{MR}_- = \operatorname{MR}([N] \setminus \widecheck{\cal P},\, \widecheck{{\cal P}})$.

\medskip
{\it Proof of Claim 1.} We only prove the case $\alpha=-1$. The other case is treated similarly.

Writing $|\widecheck{{\cal P}}|=|{\cal P}_+\cap \widecheck{{\cal P}}|+|{\cal P}_-\cap \widecheck{{\cal P}}|=|{\cal P}_+\cap \widecheck{{\cal P}}|+|{\cal P}_-|-|{\cal P}_-\cap ([N] \setminus \widecheck{\cal P})|$, we obtain that $|\widecheck{{\cal P}}|-|{\cal P}_-|\leq |{\cal P}_+\cap \widecheck{{\cal P}}|\leq N\operatorname{MR}(\widecheck{{\cal P}},\,[N] \setminus \widecheck{\cal P}).$ Arguing similarly, we can also check that that $|{\cal P}_-|-|\widecheck{{\cal P}}|\leq |{\cal P}_-\cap ([N] \setminus \widecheck{\cal P})|\leq N\operatorname{MR}(\widecheck{{\cal P}},\,[N] \setminus \widecheck{\cal P})$, concluding the proof. 
\medskip

In the following, assume that $N\geq N_1:=\lceil 2/(r_+\wedge r_-)\rceil$ so that $r_+^N\wedge r_-^N\geq (r_+\wedge r_-)/2$. Using this inequality on a combination of Claim 1 and Lemma \ref{lem:l2norm:controls:MR} applied to $\overline{\sigma}^{\rm sp}$ (hence $b_+ = \overline{\sigma}_+ / \|\overline{\sigma}^{\rm ag}\|_2$ and $b_- = \overline{\sigma}_{-} / \|\overline{\sigma}^{\rm ag}\|_2$) and $\alpha \widecheck{v}$, we have
\begin{equation*}
    N \operatorname{MR}_{\alpha} \leq  \frac{2 \| \alpha \widecheck{v} - \overline{\sigma}^{\rm sp} \|_2^2}{|b_+-b_-|^2 \cdot (r_+^N\wedge r_-^N)^2} 
    \leq \frac{8 \|\overline{\sigma}^{\rm ag}\|_2^2 \cdot\| \alpha\widecheck{v} - \overline{\sigma}^{\rm sp} \|_2^2}{|\overline{\sigma}_+-\overline{\sigma}_-|^2 \cdot (r_+\wedge r_-)^2}.
\end{equation*}
Using the fact that $\|\overline{\sigma}^{\rm ag}\|_2\leq N^{1/2} (|\overline{\sigma}_+|\vee |\overline{\sigma}_-|)$ and Inequality \eqref{ineq:1:proof:thm:community:detection}, we get
\begin{equation}\label{eq:control:MR:vvvert}
    \operatorname{MR}_{\alpha} \leq \kappa \frac{\vvvert \widehat{\Sigma}^{(1)}-\overline{\Sigma}^{(1)} \vvvert_{2}^2}{\sigma_1^2(\overline{\Sigma}^{(1)})},
\end{equation}
with $\kappa = 32 (|\overline{\sigma}_+|\vee |\overline{\sigma}_-|)^2 / (|\overline{\sigma}_+-\overline{\sigma}_-|^2 \cdot (r_+\wedge r_-)^2)$. It follows from the definitions of $\overline{\sigma}_+$ and $\overline{\sigma}_-$ that $|\overline{\sigma}_+|\vee |\overline{\sigma}_-| \geq c_1 p$ and $|\overline{\sigma}_+-\overline{\sigma}_-| = 2c_1 p$. In turn, we have $\kappa \geq 32$ since $r_+\wedge r_- \leq 1/2$. In particular, it implies that the event $F$ defined below is included in $E$.

Let $\widecheck{\sigma}_+$ and $\widecheck{\sigma}_-$ denote the two following averages,
\begin{equation}
\label{def:hatsigma}
\widecheck{\sigma}_+ = |\widecheck{{\cal P}}|^{-1}\sum_{i\in \widecheck{{\cal P}}}\widehat{\sigma}_i^{\rm ag}
\quad \text{and} \quad
\widecheck{\sigma}_- = |[N] \setminus \widecheck{\cal P}|^{-1}\sum_{i\in [N] \setminus \widecheck{\cal P}}\widehat{\sigma}_i^{\rm ag}.
\end{equation}
They are expected to be close to $\overline{\sigma}_+$ and $\overline{\sigma}_-$ up to sign ambiguity as stated below.

\medskip
{\bf Claim 2.} Let $N\geq N_1:=\lceil 2/(r_+\wedge r_-)\rceil$ and $\alpha \in\{-1,1\}$. On the event 
\[
F:=\Big\{\vvvert \widehat{\Sigma}^{(1)}-\overline{\Sigma}^{(1)} \vvvert_{2}< \sqrt{\frac{r_+ \wedge r_-}{2\kappa}} \, \sigma_1(\overline{\Sigma}^{(1)})\Big\} \subset E,
\]
we have 
$|\widecheck{\sigma}_+ - \overline{\sigma}_\alpha| \vee |\widecheck{\sigma}_- - \overline{\sigma}_{-\alpha}| \leq C \vvvert \widehat{\Sigma}^{(1)}-\overline{\Sigma}^{(1)}\vvvert_2$, where
\begin{equation*}
    C =  \frac{4 \kappa |\overline{\sigma}_+ - \overline{\sigma}_-|}{(r_+\wedge r_-) \sigma^2_1(\overline{\Sigma}^{(1)})} + \frac{2}{\sqrt{r_+\wedge r_-}}.
\end{equation*}

\medskip
{\it Proof of Claim 2.} 
We only prove the inequality regarding $\widecheck{\sigma}_+$. The other one is treated similarly. Observe that
\begin{align}
\label{eq:2:proof:thm:community:detection}
\widecheck{\sigma}_+&=\frac{1}{|\widecheck{{\cal P}}|}\langle \widehat{\sigma}^{\rm ag}, 1_{\widecheck{{\cal P}}}\rangle = \frac{1}{|\widecheck{{\cal P}}|}\langle \overline{\sigma}^{\rm ag}, 1_{\widecheck{{\cal P}}}\rangle+\frac{1}{|\widecheck{{\cal P}}|}\langle \widehat{\sigma}^{\rm ag}-\overline{\sigma}^{\rm ag}, 1_{\widecheck{{\cal P}}}\rangle\nonumber\\
&=\frac{(\overline{\sigma}_+ |{\cal P}_+\cap \widecheck{{\cal P}}|+\overline{\sigma}_- |{\cal P}_-\cap \widecheck{{\cal P}}|)}{|\widecheck{{\cal P}}|}+\frac{1}{|\widecheck{{\cal P}}|}\langle \widehat{\sigma}^{\rm ag}-\overline{\sigma}^{\rm ag}, 1_{\widecheck{{\cal P}}}\rangle.
\end{align}
Next, whatever the value of $\alpha\in\{-1,1\}$, we use Claim 1 to get
\begin{align}
\label{ineq:3:proof:thm:community:detection}
 \frac{(\overline{\sigma}_+ |{\cal P}_+\cap \widecheck{{\cal P}}|+\overline{\sigma}_- |{\cal P}_-\cap \widecheck{{\cal P}}|)}{|\widecheck{{\cal P}}|}&=\frac{(\overline{\sigma}_\alpha |{\cal P}_\alpha\cap \widecheck{{\cal P}}|+\overline{\sigma}_{-\alpha} |{\cal P}_{-\alpha}\cap \widecheck{{\cal P}}|)}{|\widecheck{{\cal P}}|} \nonumber\\
 &=\overline{\sigma}_\alpha+\frac{1}{|\widecheck{{\cal P}}|}|\widecheck{{\cal P}}\setminus {\cal P}_\alpha|(\overline{\sigma}_{-\alpha}-\overline{\sigma}_\alpha).
\end{align}
Combining identities \eqref{eq:2:proof:thm:community:detection} and \eqref{ineq:3:proof:thm:community:detection}, and then using Claim 1, 
we obtain that 
\begin{align}
\label{ineq:4:proof:thm:community:detection}
|\widecheck{\sigma}_+-\overline{\sigma}_\alpha|&\leq \frac{1}{|\widecheck{{\cal P}}|}|\widecheck{{\cal P}}\setminus {\cal P}_\alpha||\overline{\sigma}_{-\alpha}-\overline{\sigma}_\alpha|+ \frac{1}{|\widecheck{{\cal P}}|}\langle \widehat{\sigma}^{\rm ag}-\overline{\sigma}^{\rm ag}, 1_{\widecheck{{\cal P}}}\rangle \nonumber\\
&\leq \frac{1}{|\widecheck{{\cal P}}|}|{\cal M}||\overline{\sigma}_{-\alpha}-\overline{\sigma}_\alpha|+ \frac{1}{|\widecheck{{\cal P}}|}\langle \widehat{\sigma}^{\rm ag}-\overline{\sigma}^{\rm ag}, 1_{\widecheck{{\cal P}}}\rangle\nonumber \\
&= \frac{1}{|\widecheck{{\cal P}}|}|{\cal M}||\overline{\sigma}_{-}-\overline{\sigma}_+|+ \frac{1}{|\widecheck{{\cal P}}|}\langle \widehat{\sigma}^{\rm ag}-\overline{\sigma}^{\rm ag}, 1_{\widecheck{{\cal P}}}\rangle.
\end{align}
Now, it follows from  Claim 1 and Inequality \eqref{eq:control:MR:vvvert} that   
\begin{align}
\label{ineq:5:proof:thm:community:detection}
|\widecheck{{\cal P}}|&\geq |{\cal P}_\alpha| - N \operatorname{MR}_\alpha = N\Big( r^N_{\alpha}- \operatorname{MR}_\alpha \Big)\nonumber \\
&\geq N\Big( \frac{r_+\wedge r_-}{2} - \kappa \frac{\vvvert \widehat{\Sigma}^{(1)}-\overline{\Sigma}^{(1)} \vvvert_{2}^2}{\sigma_1^2(\overline{\Sigma}^{(1)})}\Big)\geq N \frac{r_+\wedge r_-}{4},
\end{align}
where in the last inequality we used the fact that we are working on the event $F$.
Hence, putting together \eqref{ineq:4:proof:thm:community:detection} and \eqref{ineq:5:proof:thm:community:detection}, and using  Inequality \eqref{eq:control:MR:vvvert} once more,  we conclude that on $F$ and for all $N\geq N_1$,   
\begin{align}
\label{ineq:5.5:proof:thm:community:detection}
|\widecheck{\sigma}_+-\overline{\sigma}_\alpha|&\leq \frac{4|\overline{\sigma}_--\overline{\sigma}_+|}{r_+\wedge r_-} \operatorname{MR}_\alpha+\frac{1}{|\widecheck{{\cal P}}|}\langle \widehat{\sigma}^{\rm ag}-\overline{\sigma}^{\rm ag}, 1_{\widecheck{{\cal P}}}\rangle\nonumber \\
&\leq \frac{4 \kappa |\overline{\sigma}_--\overline{\sigma}_+|}{(r_+\wedge r_-) \sigma^2_1(\overline{\Sigma}^{(1)})}\vvvert \widehat{\Sigma}^{(1)}-\overline{\Sigma}^{(1)} \vvvert_{2}^2+\frac{1}{|\widecheck{{\cal P}}|}\langle \widehat{\sigma}^{\rm ag}-\overline{\sigma}^{\rm ag}, 1_{\widecheck{{\cal P}}}\rangle.
\end{align}
Finally, using Cauchy-Schwarz inequality, the fact that $\|1_{V}\|_2 = |V|^{1/2}$ for all $V \subset [N]$
and then \eqref{ineq:5:proof:thm:community:detection},  we have that 
\begin{align}
\label{ineq:6:proof:thm:community:detection}
\frac{1}{|\widecheck{{\cal P}}|}\langle \widehat{\sigma}^{\rm ag}-\overline{\sigma}^{\rm ag}, 1_{\widecheck{{\cal P}}}\rangle& \leq |{\widecheck{{\cal P}}}|^{-1/2}\|\widehat{\sigma}^{\rm ag}-\overline{\sigma}^{\rm ag}\|_2\nonumber \\
&\leq \frac{2N^{-1/2}}{\sqrt{r_+\wedge r_-}}\|\widehat{\sigma}^{\rm ag}-\overline{\sigma}^{\rm ag}\|_2 = \frac{2N^{-1/2}}{\sqrt{r_+\wedge r_-}}\|(\widehat{\Sigma}^{(1)}-\overline{\Sigma}^{(1)})^{\top}1_N\|_2\nonumber \\
&\leq \frac{2}{\sqrt{r_+\wedge r_-}}\vvvert (\widehat{\Sigma}^{(1)}-\overline{\Sigma}^{(1)})^{\top}\vvvert_2=\frac{2}{\sqrt{r_+\wedge r_-}}\vvvert \widehat{\Sigma}^{(1)}-\overline{\Sigma}^{(1)} \vvvert_2. 
\end{align}
Combining \eqref{ineq:5.5:proof:thm:community:detection} and \eqref{ineq:6:proof:thm:community:detection} ends the proof of Claim 2.

\medskip
{\bf Claim 3}. Let $N_1$ and $C$ be defined as in Claim 2. For all $N\geq N_1$, on the event 
\[
G:=\Big\{\vvvert \widehat{\Sigma}^{(1)}-\overline{\Sigma}^{(1)} \vvvert_{2}<\Big(\sqrt{r_+\wedge r_-}\frac{\kappa}{4}\sigma_1(\overline{\Sigma}^{(1)})\Big)\wedge \Big(\frac{\overline{\sigma}_+ - \overline{\sigma}_-}{4C}\Big)\Big\}\subset F,
\]
the estimated sign is correct in the sense that $\widehat{\sigma}^{\rm sp} = \alpha \widecheck{v}$.

\medskip
{\it Proof of Claim 3}. 
By definitions of $\widecheck{\sigma}$ and $\widehat{\sigma}^{\rm sp}$, i.e. Equations \eqref{def:hatsigma} and \eqref{eq:def:hat:sigma:spectral}, we have $\widehat{\sigma}^{\rm sp} = \widecheck{v}$ if $\widecheck{\sigma}_+ \geq \widecheck{\sigma}_-$ and $\widehat{\sigma}^{\rm sp} = - \widecheck{v}$ else. Hence it suffices to prove that $\widecheck{\sigma}_+ \geq \widecheck{\sigma}_-$ if and only if $\alpha = 1$.

Suppose $\alpha = 1$. Since $G\subset F$ and $N\geq N_1$, it follows from Claim 2 that 
\begin{align*}
 \widecheck{\sigma}_+ - \widecheck{\sigma}_- \geq \overline{\sigma}_+ - \overline{\sigma}_- - 2C\vvvert \widehat{\Sigma}^{(1)}-\overline{\Sigma}^{(1)} \vvvert_2 
 \geq \overline{\sigma}_+ - \overline{\sigma}_- - (\overline{\sigma}_+ - \overline{\sigma}_-)/2
 = (\overline{\sigma}_+ - \overline{\sigma}_-)/2>0. 
\end{align*}
The same argument applies to the case $\alpha = -1$ and concludes the proof of Claim 3.

\medskip
Combining Claim 3 and Inequality \eqref{ineq:1:proof:thm:community:detection}, we deduce that $\|\widehat{\sigma}^{\rm sp} - \overline{\sigma}^{\rm sp}\|_2 \leq 2\vvvert \widehat{\Sigma}^{(1)}-\overline{\Sigma}^{(1)} \vvvert_{2} / \sigma_1(\overline{\Sigma}^{(1)})$ on the event $G$. First, remind that $\sigma_1\big(\overline{\Sigma}^{(1)}\big)=\sqrt{(\overline{\sigma}_+)^2r^N_++(\overline{\sigma}_-)^2r^N_-}$ so that $\sigma_1\big(\overline{\Sigma}^{(1)}\big) \geq C_1 :=  (\overline{\sigma}_+ \vee \overline{\sigma}_-) \cdot 2^{-1/2} \cdot (r_+ \wedge r_-)^{1/2} >0$ for all $N \geq N_1$.
Let $K_1$, $K_2$ and $K_3$ be the positive constants depending  on $\lambda$ given by Theorem \ref{thm:error:between:empirical:Sigma1:and:Sigma1} and $K$ be the positive constant depending on $\lambda$, $p$ and $r_+$ given by Corollary \ref{cor:concentration:Sigma:operator:norm} and let  
\[
N_2=\min\{N\geq 1: (K_1\vee (2K))\sqrt{\log(N)/N}\leq C_2/2\},
\]
where $C_2=\Big(\sqrt{r_+\wedge r_-}\frac{\kappa}{4}\sigma_1(\overline{\Sigma}^{(1)})\Big)\wedge \Big(\frac{|\overline{\sigma}_+|\wedge |\overline{\sigma}_-|}{2C}\Big).$
Also, increase $K_3$ if necessary so that $(K_1\vee (2K))(K_3)^{-1/2}\leq C_2/2$ and observe that for such $K_3$ it holds that
\[
(K_1\vee (2K))\sqrt{\frac{N\log(T)\log(N\log(T))}{T}}\leq  C_2/2,
\]
for all pairs $(N,T)$ satisfying $T\geq K_3N\log (T)\log(N\log(T))$, i.e. condition \eqref{def:1:over:SNRn}.

As a consequence, for all pairs $(N,T)$ for which $N\geq N_0:=N_1\vee N_2$ and 
condition \eqref{def:1:over:SNRn} is satisfied with $K_3$ as above, we have that 
$$
H:=\Bigg\{\vvvert \widehat{\Sigma}^{(1)}-\overline{\Sigma}^{(1)} \vvvert_{2}\leq (K_1\vee (2K))\Big(\sqrt{\frac{N}{T}\log(T)\log(N\log(T))}+\sqrt{\frac{\log(N)}{N}}\Big)\Bigg\}\subset G.
$$ 
Combining the inequality $\|\widehat{\sigma}^{\rm sp} - \overline{\sigma}^{\rm sp}\|_2 \leq 2\vvvert \widehat{\Sigma}^{(1)}-\overline{\Sigma}^{(1)} \vvvert_{2} / C_1$ with Corollary \ref{cor:concentration:Sigma:operator:norm} (with $\delta=N^{-1}$) and Theorem \ref{thm:error:between:empirical:Sigma1:and:Sigma1}, the result follows.
\end{proof}

Below, we prove Theorem \ref{thm:missclasification:rate:spectral:method}. 

\begin{proof}[Proof of Theorem \ref{thm:missclasification:rate:spectral:method}]
The proof follows from Theorem \ref{thm:estimation:spectral:method} and Lemma \ref{lem:l2norm:controls:MR}.
    
\end{proof}

Finally, we prove Lemma \ref{lem:l2norm:controls:MR}. 

\begin{lemma}\label{lem:l2norm:controls:MR}
    Let $\overline{\sigma} = b_+ 1_{{\cal P}_+} + b_- 1_{{\cal P}_-}$ for some $b_+ \neq b_-$ and $\widehat{\sigma}$ be an estimator of $\overline{\sigma}$. Let $(\widehat{\cal P}_+, \widehat{\cal P}_-)$ be the partition defined as
\begin{equation*}
\widehat{\cal P}_+= \left\{i\in[N]: \widehat{\sigma}_i \ge N^{-1} \sum_{j=1}^N \widehat{\sigma}_j\right\}
\quad \text{and} \quad
\widehat{\cal P}_-=[N]\setminus \widehat{\cal P}_+.
\end{equation*}    
Then the following inequality holds:
    \begin{equation*}
        \operatorname{MR}(\widehat{\mathcal{P}}_+, \widehat{\mathcal{P}}_-) \leq \frac{2 \| \widehat{\sigma} - \overline{\sigma} \|_2^2}{|b_+-b_-|^2 \cdot (r_+^N\wedge r_-^N)^2} N^{-1}.
    \end{equation*}
\end{lemma}
\begin{proof}
Let $v$ and $\overline{v}$ be the centered versions of $\widehat{\sigma}$ and $\overline{\sigma}$, that is
\begin{equation*}
    x_i = \hat{\sigma}_i - N^{-1}\sum_{j=1}^N \hat{\sigma}_j
    \quad \text{and} \quad
    \overline{x}_i = \overline{\sigma}_i - N^{-1}\sum_{j=1}^N \overline{\sigma}_j.
\end{equation*}
Remark that the set of misclassified components is ${\cal M}=\{i\in [N]:\text{sign}(v_i)\neq \text{sign}(\overline{v}_i) \}$. Note that the coordinates of $v$ and $\overline{v}$ have different signs in ${\cal M}$. As a consequence, for all $i\in {\cal M}$, 
\[
|v_i - \overline{v}_i| \geq 
| \overline{v}_i| \geq 
|b_+-b_-| (r_+^N \wedge r_-^N),
\]
and so $\| v - \overline{v} \|_2^2 \geq |\mathcal{M}| \cdot |b_+-b_-|^2 \cdot (r_+^N\wedge r_-^N)^2$. Yet, by triangle and Cauchy-Schwarz inequalities,
\begin{eqnarray}
    \| v - \overline{v} \|_2 
    &=&    \| \hat{\sigma} - \overline{\sigma} - N^{-1} \langle \hat{\sigma} - \overline{\sigma}, 1_N \rangle 1_N \|_2 \nonumber\\
    &\leq& \| \hat{\sigma} - \overline{\sigma} \|_2 + N^{-1} |\langle \hat{\sigma} - \overline{\sigma}, 1_N \rangle| \cdot \| 1_N \|_2 \nonumber\\
    &\leq& \| \hat{\sigma} - \overline{\sigma} \|_2 + N^{-1} \| \hat{\sigma} - \overline{\sigma} \|_2 \cdot \| 1_N \|_2^2 = 2 \| \hat{\sigma} - \overline{\sigma} \|_2. \label{eq:example:centered:l2}
\end{eqnarray}
To conclude, recall that $\operatorname{MR}(\widehat{\mathcal{P}}_+, \widehat{\mathcal{P}}_-) = |\mathcal{M}|/N$.
\end{proof}

\section{Simulation study}
\label{sec:simulation}

The numeric experiments were made using Julia programming language and are available in the package \href{https://github.com/jucheval/MeanFieldGraph.jl}{MeanFieldGraph.jl}, as well as the material used to produce the figures below.

The simulation framework follows the one of \cite{Chevallier2024inferring}. In particular, we consider the parameter $\beta = \mu/\lambda$ instead of $\mu$ because the set of admissible values of $\beta$ is independent of $\lambda$ (which is not the case for $\mu$). If not specified otherwise, the following parameter values are used:
\begin{equation*}
    N = 50, \quad 
    r_+ = .5, \quad 
    \beta = .5, \quad 
    \lambda = .5, \quad 
    p = .5, \quad
    N_{\rm simu} = 1000.
\end{equation*}
Furthermore, the fractions of excitatory and inhibitory components are chosen as $r_+^N = \lceil r_+N\rceil$ and $r_-^N=N-|\mathcal{P}_+|$.

The performance of any community detection method is evaluated via two metrics: the probability of exact recovery \big(PER = $\P(\{\widehat{{\cal P}}_{+}={\cal P_{+}}\}\cap \{\widehat{{\cal P}}_{-}={\cal P_{-}}\} )$\big) and the mean misclassification rate \big(MMR = $\E[ \operatorname{MR}(\widehat{{\cal P}}_{+},\widehat{{\cal P}}_{-})]$\big), recall the definition of the misclassification rate $\operatorname{MR}(\widehat{{\cal P}}_{+},\widehat{{\cal P}}_{-})$ given in \eqref{def:discrepancy:measure}.
Both metrics are estimated via Monte Carlo method ($N_{\rm simu}$ indicates the number of simulations). 

We evaluate 8 methods, obtained as the combination of the two methods (aggregated and spectral) developed in the paper with four clustering algorithms ($k$-means, mean threshold and hierarchical clustering with two linkage functions).
Before comparing the performance of the different methods in terms of PER and MMR, we highlight that the spectral method is computationally more expensive than the aggregated method. Indeed, the spectral method requires estimating the full 1-lagged covariance matrix rather than just its column sums and finding its leading singular vector. For instance, with $N=100$ and $T=10000$, the computation of $\widehat{\sigma}^{\rm sp}$ takes around $130$ms, whereas $\widehat{\sigma}^{\rm ag}$ takes only $0.5$ms.

Tables \ref{tab:probability:exact:recovery} and \ref{tab:mean:misclassification:rate} give a comparison of these 8 methods in terms of PER and MMR respectively.
We selected specific couples $(N,T)$ to produce the tables. 
The aggregated method combined with mean threshold clustering (denoted by ag\_threshold in the following) is always among the best methods. This is also true for the full dataset. To be exhaustive, let us mention that, out of 420 couples $(N,T)$, there are:
\begin{itemize}
    \item 25 couples where ag\_threshold is not among the best methods in terms of MMR (this number goes down to 13 couples if MMR is rounded up to 1\%),
    \item 7 couples where ag\_threshold is not among the best methods in terms of PER (this number goes down to 1 couple if PER is rounded up to 1\%).
\end{itemize}
However, this is an artifact of the balanced populations ($r_+ = 0.5$) considered in the simulation framework. The performance of the mean threshold clustering drops drastically in unbalanced settings (see Appendix \ref{app:plots}). The aggregated method combined with $k$-means clustering does not suffer form this drawback and yet has comparable performances. Let us mention that if we remove the two methods based on mean threshold clustering there are:
\begin{itemize}
    \item 27 couples $(N,T)$ where ag\_kmeans is not among the best methods in terms of MMR (this number goes down to 14 couples if MMR is rounded up to 1\%),
    \item 11 couples where ag\_kmeans is not among the best methods in terms of PER (this number goes down to 2 couples if PER is rounded up to 1\%).
\end{itemize}
For these two reasons, ag\_kmeans is the recommended method and is the only one considered until the end of this section.

\begin{table}[ht]
    \centering
    \scriptsize
    \begin{tabular}{c|c||c|c|c|c||c|c|c|c}
    \toprule
    \multirow{2}{*}{$N$} & \multirow{2}{*}{$T$} & \multicolumn{4}{c||}{Aggregated method} & \multicolumn{4}{c}{Spectral method}\\
        \cline{3-10}
     & & kmeans & threshold & avg & ward & kmeans & threshold & avg & ward\\
    \midrule
    $34$ & $5272$ & $84 \pm 2$ & $\mathbf{87 \pm 2}$ & $78 \pm 3$ & $78 \pm 3$ & $77 \pm 3$ & $80 \pm 2$ & $70 \pm 3$ & $68 \pm 3$\\
    $94$ & $15796$ & $86 \pm 2$ & $\mathbf{88 \pm 2}$ & $79 \pm 3$ & $81 \pm 2$ & $78 \pm 3$ & $80 \pm 2$ & $68 \pm 3$ & $68 \pm 3$\\
    $142$ & $26320$ & $\mathbf{90 \pm 2}$ & $\mathbf{90 \pm 2}$ & $82 \pm 2$ & $85 \pm 2$ & $82 \pm 2$ & $84 \pm 2$ & $72 \pm 3$ & $76 \pm 3$\\
    $190$ & $36844$ & $89 \pm 2$ & $\mathbf{90 \pm 2}$ & $81 \pm 2$ & $84 \pm 2$ & $82 \pm 2$ & $83 \pm 2$ & $71 \pm 3$ & $75 \pm 3$\\
    $250$ & $50000$ & $\mathbf{92 \pm 2}$ & $\mathbf{92 \pm 2}$ & $81 \pm 2$ & $86 \pm 2$ & $85 \pm 2$ & $86 \pm 2$ & $73 \pm 3$ & $77 \pm 3$\\
    \bottomrule
    \end{tabular}
    \caption{Probability of exact recovery (as percentages) $\pm$ radius of the 95\% confidence interval, computed over $N_{\rm simu}$ simulations. The couples $(N,T)$ are arbitrarily chosen so that the highest PER of each line is around 90\%. The columns avg and ward correspond to hierarchical clustering with, respectively, the average and Ward linkage functions.}
    \label{tab:probability:exact:recovery}
\end{table}

\begin{table}[ht]
    \centering
    \scriptsize
    \begin{tabular}{c|c||c|c|c|c||c|c|c|c}
    \toprule
    \multirow{2}{*}{$N$} & \multirow{2}{*}{$T$} & \multicolumn{4}{c||}{Aggregated method} & \multicolumn{4}{c}{Spectral method}\\
        \cline{3-10}
     & & kmeans & threshold & avg & ward & kmeans & threshold & avg & ward\\
    \midrule
    $34$ & $2641$ & $3 \pm 3$ & $\mathbf{2 \pm 3}$ & $4 \pm 6$ & $4 \pm 4$ & $5 \pm 4$ & $4 \pm 3$ & $7 \pm 9$ & $6 \pm 5$\\
    $94$ & $7903$ & $2 \pm 1$ & $\mathbf{1 \pm 1}$ & $2 \pm 3$ & $3 \pm 3$ & $3 \pm 2$ & $3 \pm 2$ & $5 \pm 5$ & $4 \pm 3$\\
    $142$ & $10534$ & $\mathbf{2 \pm 1}$ & $\mathbf{2 \pm 1}$ & $3 \pm 4$ & $3 \pm 3$ & $4 \pm 2$ & $4 \pm 2$ & $6 \pm 6$ & $6 \pm 3$\\
    $190$ & $13165$ & $\mathbf{2 \pm 1}$ & $\mathbf{2 \pm 1}$ & $4 \pm 5$ & $4 \pm 3$ & $4 \pm 2$ & $4 \pm 2$ & $8 \pm 9$ & $6 \pm 4$\\
    $250$ & $18427$ & $\mathbf{2 \pm 1}$ & $\mathbf{2 \pm 1}$ & $3 \pm 3$ & $3 \pm 2$ & $4 \pm 1$ & $4 \pm 1$ & $6 \pm 7$ & $5 \pm 3$\\
    \bottomrule
    \end{tabular}
    \caption{Mean misclassification rate (as percentages) $\pm$ standard deviation, computed over $N_{\rm simu}$ simulations. The couples $(N,T)$ are arbitrarily chosen so that the lowest MMR of each line is around 2\%. The columns avg and ward correspond to hierarchical clustering with, respectively, the average and Ward linkage functions.}
    \label{tab:mean:misclassification:rate}
\end{table}

Figure \ref{fig:proba:exact:recovery} illustrates the asymptotic regimes of Corollary \ref{cor:extact_recovery} and Theorem \ref{thm:missclasification:rate:spectral:method} for the ag\_kmeans method. The left plot shows that there is a separation between couples $(T,N)$ for which exact recovery never occurs (dark color) and those for which exact recovery always occurs (light color). Furthermore, the red curve demarcating this separation is compatible (modulo a log factor) with our condition $(N/T^{1/2})\log(NT) \to 0$ implying exact recovery. 

The right plot shows a smoother transition between couples $(T,N)$ for which the misclassification rate is close to random clustering (dark color) and those for which the misclassification rate is low (light color). Furthermore, the level set corresponding to a MMR value of 0.1 seems to be a line which is compatible with the fact that our convergence rate is mainly $N/T$ (modulo log factors).

\begin{figure}[ht]
    \includegraphics[width=.49\textwidth]{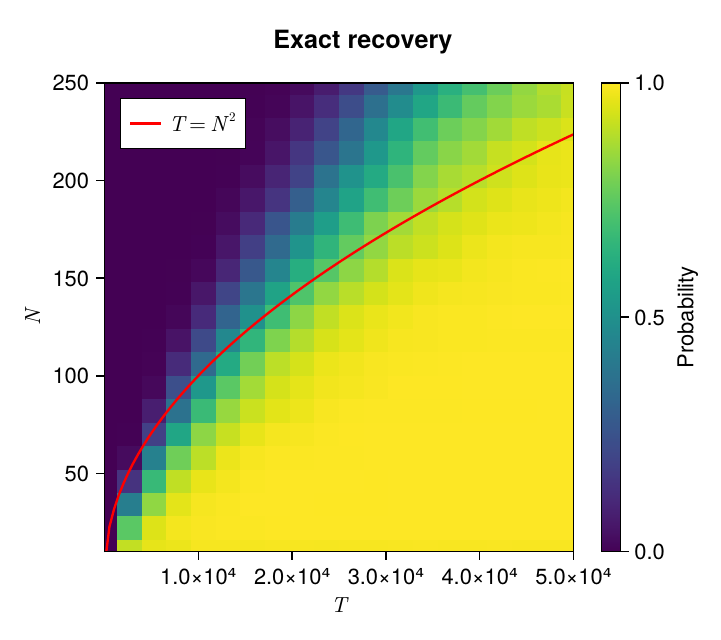}%
    \includegraphics[width=.49\textwidth]{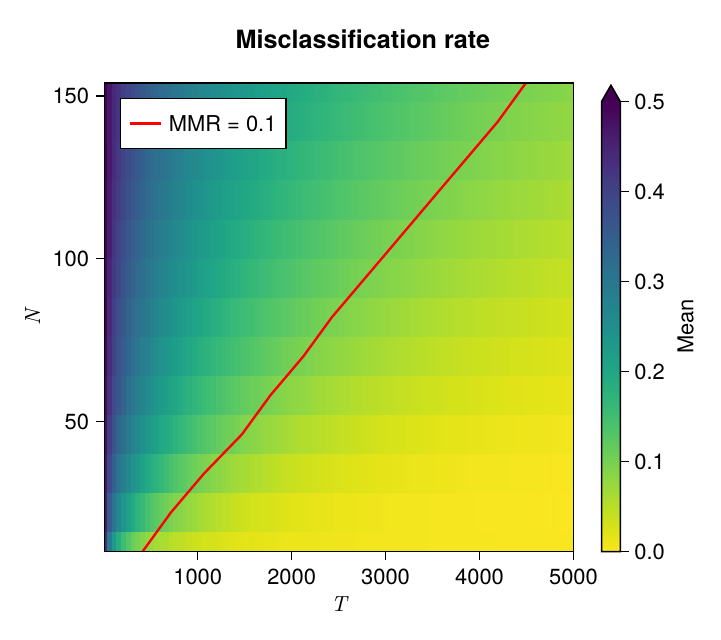}
    \caption{\label{fig:proba:exact:recovery}
    Estimated PER (left) and MMR (right) as a function of $N$ and $T$ (computed over $N_{\rm simu}$ simulations) for ag\_kmeans method. Left: The red curve correspond to the equation $T = N^2$ which is the critical condition of our asymptotic regime. Right: The red curve correspond to the couples $(N,T)$ for which the MMR is equal to 0.1.
    \emph{The scales are different from one plot to the other.}
    }
\end{figure}

Finally, Appendix \ref{app:plots} gives an overview of the performance of ag\_kmeans and ag\_threshold methods as one of the parameters ($N$, $r_+$, $\beta$, $\lambda$ or $p$) varies.

\begin{acks}[Acknowledgments]
The authors would like to thank the anonymous referees for their valuable comments and suggestions. 
This research has been conducted while J.C. was in Statify team at Centre Inria de l'Université Grenoble Alpes.
\end{acks}

\begin{funding}
G.O. was supported by the Serrapilheira Institute (grant number Serra – 2211-42049), FAPERJ (grants E-26/201.397/2021 and E-26/204.532/2024) and CNPq (grants 303166/2022-3). J.C. was supported by ANR-19-CE40-0024 (CHAllenges in MAthematical NEuroscience).
\end{funding}

\bibliographystyle{imsart-nameyear}
\bibliography{Ref}

\clearpage

\appendix

\section{Lemmas based on Hoeffding inequality}
\label{sec:proof:max:inequality}

The two lemmas stated here rely on Hoeffding concentration inequality.

\begin{lemma}\label{lem:Hoeffding:theta}
    Let $(\theta_{ij})_{1\leq i,j \leq N}$ be i.i.d. entries distributed as $\operatorname{Ber}(p)$ and $A, B \subset [N]$ be two subsets. Then,
    \begin{equation}\label{eq:Hoeffding:theta:1}
        \mathbb{P}\left( \max_{i\in [N]} \left| \frac{1}{N} \sum_{k \in A} (\theta_{ij} - p) \right| \geq \sqrt{\frac{|A|\log(2N/\delta)}{2N^2}}\right) \leq \delta,
    \end{equation}
    \begin{equation}\label{eq:Hoeffding:theta:2}
        \mathbb{P}\left( \max_{i,j\in [N]: \ i\neq j}\left|\frac{1}{N}\sum_{k=1}^{N}\theta_{ik}\theta_{jk}-p^2\right| \geq 2\sqrt{\frac{\log(4N/\delta)}{2N}}\right) \leq \delta,
    \end{equation}
    \begin{equation}\label{eq:Hoeffding:theta:3}
        \mathbb{P}\left(\max_{i,j\in [N]: \ i\neq j}\left|\frac{1}{N^2}\sum_{(k,\ell)\in A\times B}(\theta_{ik}\theta_{j\ell}-p^2)\right| \geq 2\sqrt{\frac{\log(4N/\delta)}{2N}}\right) \leq \delta.
    \end{equation}
\end{lemma}
\begin{proof}
    The independent and centered random variables $\theta_{ij} - p$ have range equal to 1 so Hoeffding inequality combined with the union bound gives
    \begin{equation*}
        \mathbb{P}\left( \max_{i\in [N]} \left| \frac{1}{N} \sum_{k \in B} (\theta_{ik} - p) \right| \geq z \right) \leq 2 N e^{-2N^2z^2/|A|}.
    \end{equation*}
    Taking $z = (|A|\log(2N/\delta)/(2N^2))^{1/2}$ proves~\eqref{eq:Hoeffding:theta:1}. 
    
    The proofs of the two last statements are similar. We only prove the last one which is a bit more complex.
    Let $i\neq j$. First of all, observe that we can write 
    \begin{equation*}
        \frac{1}{N^2}\sum_{(k,\ell)\in A\times B}(\theta_{ik}\theta_{j\ell}-p^2)=\left(\frac{1}{N}\sum_{k\in A}\theta_{ik}\right)\sum_{\ell\in B}\frac{(\theta_{j\ell}-p)}{N}+\frac{p|B|}{N}\sum_{k\in A}\frac{(\theta_{ik}-p)}{N}.
    \end{equation*}
   Next, using that $(|A|\vee |B|)/N\leq 1$, $0\leq p\leq 1$ and $|\theta_{mn}|\leq 1$, we obtain from the previous inequality the following upper bound  
    \begin{equation*}
        \max_{i,j\in [N]:i\neq j}\left|\frac{1}{N^2}\sum_{(k,\ell)\in A\times B}(\theta_{ik}\theta_{j\ell}-p^2)\right|\leq 
        \max_{j\in [N]}\left|\sum_{\ell\in B}\frac{(\theta_{j\ell}-p)}{N}\right|
        +\max_{i\in [N]}\left|\sum_{k\in A}\frac{(\theta_{ik}-p)}{N}\right|.
    \end{equation*}
    Using Equation~\eqref{eq:Hoeffding:theta:1}, the union bound and substituting $\delta$ by $\delta/2$ conclude the proof.
\end{proof}

The lemma below explicitly contains Proposition \ref{prop:asymptotics:concentration}.

\begin{lemma}\label{lem:asymptotics:concentration}
Remind that $D$ is defined in Equation~\eqref{def:mean_activity}. There exists a constant $K>0$ depending only on $\lambda$ such that for all $b\in\{-,+\}$, $\delta\in (0,1]$ and $N\geq 2$,
\begin{gather}
    \P\left(\left\| L^{N,\bullet b} - bpr_b1_N \right\|_{\infty} \geq K\sqrt{\frac{\log(N/\delta)}{N}}\right)\leq \delta,
    \label{eq:concentration:Lb}\\
    \P\left(\left\| L^{N} - p(r_+ - r_-)1_N  \right\|_{\infty} \geq K\sqrt{\frac{\log(N/\delta)}{N}}\right)\leq \delta,
    \label{eq:concentration:L}\\
    \P\left(\left\| C^{N} - p(1_{\mathcal{P}_+} - 1_{\mathcal{P}_-})  \right\|_{\infty} \geq K\sqrt{\frac{\log(N/\delta)}{N}}\right)\leq \delta,
    \label{eq:concentration:C}\\
    \P\left(\left\| \ell^{N,\bullet b} - \left( 1_{\mathcal{P}_b} + \frac{b(1-\lambda) p r_b}{1 - D} 1_N \right)  \right\|_{\infty} \geq K \sqrt{\frac{\log(N/\delta)}{N}}  \right)\leq \delta,
    \label{eq:concentration:ellb}\\
    \P\left(\left\| \ell^{N} - \frac{1}{1 - D} 1_N   \right\|_{\infty} \geq K \sqrt{\frac{\log(N/\delta)}{N}}  \right)\leq \delta,
    \label{eq:concentration:ell}\\
    \P\left(\left\| m^{N} - m 1_N  \right\|_{\infty} \geq K \sqrt{\frac{\log(N/\delta)}{N}}  \right)\leq \delta,
    \label{eq:concentration:m}\\
    \P\left(\left\| v^N - m(1-m)1_N   \right\|_{\infty} \geq K \sqrt{\frac{\log(N/\delta)}{N}}  \right)\leq \delta.
    \label{eq:concentration:v}
\end{gather}
\end{lemma}

\begin{remark}
    The asymptotics above are related to \cite[Lemmas F.1. and F.4.]{chevallier2026supp}. There are three differences: 1) control in expectation there is replaced by control in probability here, 2) the 2-norm there is replaced by the infinity one here, 3) the rate of order 1 there is replaced by a rate of order $(\log(N)/N)^{1/2}$ here (remind the natural upper bound $\| x \|_2 \leq \sqrt{N}\| x \|_\infty$ for $x\in \mathbb{R}^N$).
\end{remark}

\begin{proof}
Let us start with the proof of \eqref{eq:concentration:Lb}. By definition, for all $i\in [N]$ and $b\in\{-,+\}$,
\begin{equation*}
    L^{N,\bullet b}_i - bpr_b^N = \sum_{j\in \mathcal{P}_b} (A^N_{ij} - bp) = \frac{b}{N} \sum_{j\in \mathcal{P}_b} (\theta_{ij} - p).
\end{equation*}
Hence, Lemma~\ref{lem:Hoeffding:theta} and the fact that $r^N_b\leq 1$, imply that
\begin{equation*}
    \P\left(\left\| L^{N,\bullet b} - bpr^N_b1_N \right\|_{\infty} \geq \sqrt{\frac{\log(2N/\delta)}{2N}} \right)\leq \delta.
\end{equation*}
Combining the above inequality with the fact that $|r^N_b - r_b|\leq KN^{-1}$, and then adjusting the constant $K$ suitably, we obtain ~\eqref{eq:concentration:Lb}.
The proofs of Equations~\eqref{eq:concentration:L} and~\eqref{eq:concentration:C} follow the same lines and are therefore omitted.

We now turn to \eqref{eq:concentration:ell}. Let us denote $\varepsilon^N = L^{N} - p(r_+ - r_-)1_N$ and remind that $\ell^N = Q^N 1_N$. Starting from $1_{N} = Q^N[I_N - (1-\lambda)A^N]1_{N}$, one gets that
\begin{equation*}
    \ell^{N} = 1_{N} + (1-\lambda)Q^N L^{N} = 1_{N} + (1-\lambda)Q^N [p(r_+ - r_-) 1_N + \varepsilon^N],
\end{equation*}
so that
\begin{equation*}
    \ell^N = \frac{1}{1 - D} \left( 1_N + (1-\lambda)Q^N \varepsilon^N \right).
\end{equation*}
Hence,
\begin{equation*}
    \left\| \ell^{N} - \frac{1}{1 - D} 1_N \right\|_{\infty} \leq \frac{1-\lambda}{1-D} \vvvert Q^N \vvvert_{\infty} \| \varepsilon^N \|_{\infty} \leq \frac{1-\lambda}{\lambda(1-D)} \| \varepsilon^N \|_{\infty},
\end{equation*}
since $\vvvert Q^N \vvvert_{\infty} \leq \lambda^{-1}$ by Equation~\eqref{eq:control:QN}. Equation~\eqref{eq:concentration:ell} then follows from~\eqref{eq:concentration:L}.

Along the same lines, one can prove that 
\begin{equation*}
    \ell^{N,\bullet b} = 1_{\mathcal{P}_b} + b(1-\lambda)pr_b \ell^N + (1-\lambda)Q^N \varepsilon^{N,\bullet b},
\end{equation*}
with $\varepsilon^{N,\bullet b} = L^{N,\bullet b} - bpr_b^N1_N$,
and in turn that \eqref{eq:concentration:ellb} follows from \eqref{eq:concentration:Lb} and~\eqref{eq:concentration:ell}.

Starting from \eqref{eq:mN:ell:ell-} and using the fact that 
\begin{equation*}
    \frac{\mu}{1 - D}1_{N} + 1_{\mathcal{P}_-} - \left( 1_{\mathcal{P}_-} - \frac{(1-\lambda) p r_-}{1 - D} 1_N \right) = m 1_N,
\end{equation*}
one easily deduces \eqref{eq:concentration:m} from \eqref{eq:concentration:ellb} and \eqref{eq:concentration:ell}.

Finally, remark that
$$
\|(m^N)^2 - m^2 1_N\|_{\infty}\leq 2 \|m^N - m 1_N\|_{\infty},
$$
so that \eqref{eq:concentration:v} follows from \eqref{eq:concentration:m}.
\end{proof}

\section{Convergence of $\widehat{\sigma}^{\rm ag}$}
\label{sec:proof:conv:hat:sigma:ag:towards:sigma:ag}

In this section, we prove the convergence of $\widehat{\sigma}^{\rm ag}$ towards $\sigma^{\rm ag}$ stated in the result below.

\begin{theorem}
\label{thm:conv:hat:sigma:ag:towards:sigma:ag}
There exist positive constants $K_1,K_2$ and $K_3$ depending on $\lambda$ such that for all $N\geq 2$ and $\theta\in\{0,1\}^{N\times N}$, 
$$
\P_{\theta}\left(\|\widehat{\sigma}^{\rm ag}-\sigma^{\rm ag}\|_{\infty} \geq K_1\frac{N\log(TN)}{\sqrt{T}}\right) \leq  K_2T^{-1},
$$
provided that $T\geq K_3\log(N)$.
\end{theorem}

Note that Theorem \ref{thm:conv:hat:sigma:ag:towards:sigma:ag} holds for all realizations of the random matrix $\theta$. Clearly, integrating with respect to $\theta$ leads to a similar result with respect to the probability measure $\P$, under which the environment is random. 
The proof relies on \cite[Theorem 3]{ost2020sparse} to control the convergence of the empirical mean and empirical covariance.

\begin{proof}
Along this proof, we use the following notation: for a vector $v$ in $\mathbb{R}^N$, $\overline{v}$ denotes the spatial mean $\overline{v} = N^{-1} \sum_{i=1}^N v_i$.
First of all, observe that
\begin{align*}
\sigma^{\rm ag}_j=\sum_{i=1}^N\Sigma^{(1)}_{ij}&=\sum_{i=1}^N\cov_{\theta}(X_{i,1},X_{j,0})\\&
=N\cov_{\theta}(\overline{X_{1}},X_{j,0})=N\left(\E_{\theta}(\overline{X_{1}}X_{j,0})-\overline{m^N}m^N_j\right),
\end{align*}
and 
\begin{align*}
\widehat{\sigma}^{\rm ag}_j&=\frac{N}{T-1}\sum_{t=2}^{T}\overline{X}_tX_{j,t-1}-N\frac{\overline{Z}_T}{T}\frac{Z_{j,T}}{T}.
\end{align*}
Hence,
\begin{equation*} 
\| \sigma^{\rm ag}-\widehat{\sigma}^{\rm ag}\|_{\infty}=N\max_{j\in [N]}\left|\frac{1}{T-1}\sum_{t=2}^{T}\overline{X}_tX_{j,t-1}-\E_{\theta}(\overline{X_{1}}X_{j,0})+\overline{m^N}m^N_j-\frac{\overline{Z}_T}{T}\frac{Z_{j,T}}{T}\right|.
\end{equation*}
By the triangle inequality and the facts that $\overline{m}^N\leq 1$ and $T^{-1}Z_{j,T} \leq 1$, we have that 
$$
\max_{j\in [N]}\left|\overline{m^N}m^N_j-\frac{\overline{Z}_T}{T}\frac{Z_{j,T}}{T}\right|\leq \max_{j\in [N]}\left|m^N_j-\frac{Z_{j,T}}{T}\right| + \left|\overline{m^N}-\frac{\overline{Z}_T}{T}\right|,
$$
which together with the previous inequality gives
\begin{multline}
\label{proof_main_result_3_ineq_1}
\|\sigma^{\rm ag}-\widehat{\sigma}^{\rm ag}\|_{\infty}\leq N\max_{j\in [N]}\left|\frac{1}{T-1}\sum_{t=2}^{T}\overline{X}_tX_{j,t-1}-\E_{\theta}(\overline{X_{1}}X_{j,0})\right|
\\+N\max_{j\in [N]}\left|m^N_j-\frac{Z_{j,T}}{T}\right|+N\left|\overline{m^N}-\frac{\overline{Z}_T}{T}\right|.
\end{multline}

Now, we use Theorem 3 of \citep{ost2020sparse} to control each terms in the right hand side above. 

First, for each $j\in [N]$, let $f_j:\{0,1\}^N\to \{0,1\}$ be the projection on the $j$-th coordinate: $f_j(x)=x_j$ and denote
\begin{equation*}
    V(f_j) = \frac{1}{T}\sum_{t=1}^T(f_j(X_t)-\E_{\theta}(f_j(X_t))) = \frac{Z_{j,T}}{T} - m^N_j.    
\end{equation*}
By Theorem 3 of \citep{ost2020sparse} with $m=M=1$, ${\cal F}=\{f_j: j
\in [N]\}$, $F=[N]$ and $\log(1/(1-\lambda))/2$  as their $\theta$, we have that there exists constants $C_1,C_2>0$ depending only on $\lambda$ such that for all realization of $\theta\in\{0,1\}^{N\times N},$
\begin{equation}
\label{proof_main_result_3_ineq_2}    
\P_{\theta}\left(\cup_{j\in [N]}\left\{|V(f_j)|>\sqrt{C_1\frac{(1+\log(TN))\xi}{T}}\right\}\right)\leq C_2T^{-1}+2Ne^{-\xi},
\end{equation}

for all $\xi>0$ and $T$ such that
$$
T\geq 1+\frac{2}{\log((1-\lambda)^{-1})}(2\log(T)+\log(N)).
$$
Since, $\left|\overline{m^N}-\frac{\overline{Z}_T}{T}\right|\leq \max_{j\in [N]}|Z_{j,T}/T-m^N_j|$, Inequality \eqref{proof_main_result_3_ineq_2} implies that 

\begin{equation}
\label{proof_main_result_3_ineq_3}
\P_{\theta}\left(\left|\overline{m^N}-\frac{\overline{Z}_T}{T}\right|\leq \sqrt{C_1\frac{(1+\log(TN))\xi}{T}}\right)\geq 1-( C_2T^{-1}+2Ne^{-\xi}),   \end{equation}
for all realization of $\theta\in\{0,1\}^{N\times N}$.

Second, let us define $g_j:\{0,1\}^N\times \{0,1\}^N \to \{0,1\}$ for each $j\in [N]$ as $g_j(x,y)=x_j\overline{y}$ and denote
\begin{align*}
    W(g_j)&=\frac{1}{T-1}\sum_{t=2}^T(g_j(X_{t-1},X_t)-\E_{\theta}(g_j(X_{-1},X_{t})))\\
    &=\frac{1}{(T-1)}\sum_{t=2}^T\left(\overline{X}_tX_{j,t-1}-\E_{\theta}(\overline{X}_1X_{j,0})\right).
\end{align*}
By Theorem 3 of \citep{ost2020sparse} with $M=1$, $m=2$, ${\cal F}=\{g_j: j\in [N]\}$, $F=[N]$  and $\log(1/(1-\lambda))/2$  as their $\theta$, there exist constants $C^{\prime}_1,C^{\prime}_4>0$ depending only on $\lambda$ such that
\begin{equation}
\label{proof_main_result_3_ineq_4}
\P_{\theta}\left(\cup_{j\in [N]}\left\{W(g_j)\leq \sqrt{C^{\prime}_1\frac{(1+\log(TN))\xi}{T}}\right\}\right)\leq 1-(C^{\prime}_2T^{-1}+2Ne^{-\xi}),
\end{equation}
for all realization of $\theta\in\{0,1\}^{N\times N}$, $\xi>0$ and $T$ such that
$$
T\geq 2+\frac{2}{\log((1-\lambda)^{-1})}(2\log(T)+\log(N)).
$$

Putting together Equations \eqref{proof_main_result_3_ineq_1}-\eqref{proof_main_result_3_ineq_4} with $\xi=\log(TN)$, we show that the conclusion of the theorem is true for all $N\geq 2$ and $T\geq C_3 \log(TN)$ for some $C_3$ depending only on $\lambda$, with $K_i$ depending on $C_i$ and $C^{\prime}_i$, for each $i\in\{1,2\}$, which in turn implies that $K_i$ depends only on $\lambda$.
Hence, to conclude the proof, it remains to show that the conclusion of the theorem is also true for all $N\geq 2$ and $T\geq K_3 \log(N)$ for some $K_3$ depending only on $\lambda$. To see that, take $T_0=T_0(C_3)$ such that $T/2\geq C_3\log(T)$ for all $T\geq T_0$, define $K_3=(T_02C_3)/\log(2)$, and observe that if $T\geq K_3\log(N)>T_0$ then $T=T/2+T/2\geq C_3\log(T)+(K_3/2)\log(N)>C_3\log(T)+(T_0C_3/\log(2))\log(N)>C_3\log(T)+C_3\log(N)$.   
\end{proof}

\section{Proof of Proposition \ref{prop:concentration:asymp_approximation_simulta_cov_matrix}}
\label{sec:proof:simult:cov:matrix}

\subsection{Vectorization and Kronecker product}

Two mathematical objects are known to be very useful to deal with Stein-type matrix equations: vectorization and Kronecker product. Here are their definition and useful properties.

The {\it vectorization} $\vec(M)$ of a $N$-by-$N$ matrix $M$ is the $N^2$-dimensional vector constructed column by column. More precisely, $M_{ij}$ associated with its $k$-th coordinate, where $k=i+(j-1)N$ with $i,j\in [N]$.
With slight abuse of notation, will write $\vec(M)_{ij}$ to refer to the coordinate $k=i+(j-1)N$ of the vector $\vec(M)$.
Recall that the $\vec$ operator is an invertible linear operator on the set of all matrices with real value entries.
Given two $N$-by-$N$ matrices $A$ and $B$, we denote $A\bigotimes B$ the {\it Kronecker product} between $A$ and $B$. It has the following block representation:
$$
{A} \otimes  {B} ={\begin{bmatrix}A_{11} {B} &\cdots &A_{1n} {B} \\\vdots &\ddots &\vdots \\A_{m1} {B} &\cdots &A_{mn} {B} \end{bmatrix}}.
$$
More precisely, this is a $N^2$-by-$N^2$ matrix having value $A_{ik}B_{j\ell}$ associated with entry $(u,v)$, where $u=N(i-1)+j$ and $v=N(k-1)+\ell$  with $i,j,k,\ell \in [N]$. With slight abuse of notation, will write $(A\bigotimes B)_{ij,k\ell}$ to refer to the entry $(u,v)=(N(i-1)+j,N(k-1)+\ell)$ of the matrix $A\bigotimes B$.

Their most useful property regarding matrix equations is that, for all $N$-by-$N$ matrices $A,B,M$,
\begin{equation}\label{eq:property:vec:kronecker}
    \vec(AMB) = \left( B^\top \bigotimes A \right) \vec(M).
\end{equation}
Furthermore, remark that $\vvvert M \vvvert_{\max} = \|\vec(M)\|_{\infty}$.

\begin{proof}[Proof of Proposition \ref{prop:concentration:asymp_approximation_simulta_cov_matrix}]
The main idea is to apply the vectorization operator to the matrix equation of Proposition \ref{prop:Sigma:0:matrix:equation}. In order to emphasize the difference between vectors of size $N$ (for instance $m^N$, $v^N$) or size $N^2$ (for instance $\vec(\Sigma^{(0)})$), we use bold notation for the latter. The same convention is used for $N^2$-by-$N^2$ matrices. 
Let us then introduce the notation:
\begin{equation}
    \mathbf{\Sigma^{(0)}} := \vec(\Sigma^{(0)}), \ 
    \mathbf{v}^N := \vec(\diag(v^N)), \  
    \modif{\mathbf{1}^N := \vec(\diag(1_N))}, \  
    \mathbf{a}^N := \vec(a^N)
\end{equation}
and $\mathbf{I}$ denotes the $N^2$-by-$N^2$ identity matrix. Obviously, the abuse of notation described above for the coordinates of the vectorization and Kronecker product is inherited to the bold notation introduced here.

According to Equation \eqref{eq:property:vec:kronecker}, the starting idea to tackle the matrix equation satisfied by $\Sigma^{(0)}$ is to consider its vectorization and the Kronecker product $(1-\lambda)^2 (A^N\bigotimes A^N)$. Here is a slight modification in order to tackle the $\d0$ operator. 
Let $\mathbf{L}=(\mathbf{L}_{ij,k\ell})_{i,j,k,\ell\in [N]}$ be the $N^2$-by-$N^2$ matrix defined as
$$
\mathbf{L}_{ij,k\ell}=
    \begin{cases}
    (1-\lambda)^2 (A^N\bigotimes A^N)_{ij,k\ell}, \ \text{if} \ i\neq j, \ k,\ell \in [N]\\
    0, \ \text{otherwise}.
    \end{cases}
$$
In particular, remark that 
\begin{equation}\label{eq:L:diag:property}
    \forall \mathbf{x}\in \mathbb{R}^{N^2},\, i\in [N],\quad [\mathbf{L} \mathbf{x}]_{ii} = 0.
\end{equation}
Moreover, $\vvvert \mathbf{L} \vvvert_\infty \leq (1-\lambda)^2 < 1$ so that $\mathbf{I} - \mathbf{L}$ is invertible. Finally, remark that the expectation of $\mathbf{L}$ is given by
$$
\E\left[ \mathbf{L}_{ij,k\ell} \right]=
    \begin{cases}
    ((1-\lambda)pN^{-1})^2 , \ \text{if} \ i\neq j, \ (k,\ell) \in \mathcal{P}_+^2 \cup \mathcal{P}_-^2,\\
    -((1-\lambda)pN^{-1})^2 , \ \text{if} \ i\neq j, \ (k,\ell) \notin \mathcal{P}_+^2 \cup \mathcal{P}_-^2,\\
    0, \ \text{otherwise}.
    \end{cases}
$$

\smallskip
{\bf Step 1.} Here we prove that
$$
\mathbf{\Sigma^{(0)}}=(\mathbf{I} - \mathbf{L})^{-1} \mathbf{v}^N = \sum_{k=0}^{\infty} \mathbf{L}^k \mathbf{v}^N.
$$

\smallskip
{\it Proof of Step 1.} By Proposition \ref{prop:Sigma:0:matrix:equation} and linearity of $\vec$, we have
\begin{equation}
\label{eq_1_prof_prop:asymp_approximation_simulta_cov_matrix}
\mathbf{\Sigma^{(0)}} = (1-\lambda)^2 \vec\left(\d0\left( A^N\Sigma^{(0)}(A^N)^{\top} \right)\right) + \mathbf{v}^N.
\end{equation}
On the one hand, observe that for any $i,j\in [N]$ with $i\neq j$,
\begin{align*}
    (1-\lambda)^2 \vec\left(\d0\left( A^N\Sigma^{(0)}(A^N)^{\top} \right)\right)_{ij}
    &= (1-\lambda)^2 [A^N\Sigma^{(0)}(A^N)^{\top}]_{ij}\\
    &=(1-\lambda)^2\sum_{k=1}^N\sum_{\ell=1}^NA^N_{ik}\Sigma^{(0)}_{k\ell}A^N_{j\ell}\\
    &=\sum_{k=1}^N\sum_{\ell=1}^N \mathbf{L}_{ij,k\ell} \mathbf{\Sigma^{(0)}}_{k\ell} = [\mathbf{L} \mathbf{\Sigma^{(0)}}]_{ij}.
\end{align*}
On the other hand, for $i=j$, the definition of $\d0$ implies
$$
(1-\lambda)^2 \vec\left(\d0\left( A^N\Sigma^{(0)}(A^N)^{\top} \right)\right)_{ii} = 0 = [\mathbf{L} \mathbf{\Sigma^{(0)}}]_{ii},
$$
where the second equality comes from \eqref{eq:L:diag:property}. All in all, we have proved that
$$
\mathbf{\Sigma^{(0)}} = \mathbf{L} \mathbf{\Sigma^{(0)}} + \mathbf{v}^N,
$$
which, in turn, concludes Step 1.

\medskip
The rest of the proof is devoted to the study of the iterates of the form $\mathbf{L}^k \mathbf{v}^N$. Since we expect that $\mathbf{L}$ is close to $\E[\mathbf{L}]$ and $\mathbf{v}^N$ is close to $m(1-m)\mathbf{1}^N$, let us look at the iterates $\E[\mathbf{L}]^k \mathbf{1}^N$ as a preliminary step. It turns out that their study is simple since $\E[\mathbf{L}] \mathbf{1}^N$ is colinear to $\mathbf{a}^N$ and $\mathbf{a}^N$ is an eigenvector of $\E[\mathbf{L}]$ as proven in Step 2 below.

\smallskip
{\bf Step 2.} 
Here we prove that
$$
\E[\mathbf{L}] \mathbf{1}^N = \frac{(1-\lambda)^2p^2}{N} \mathbf{a}^N
\quad \text{and} \quad
\E[\mathbf{L}] \mathbf{a}^N = (1-\lambda)^2p^2 \left( (r_+^N - r_-^N)^2 - N^{-1} \right) \mathbf{a}^N.
$$

\smallskip
{\it Proof of Step 2.} 
Let $i,j\in [N]$. Let us first study $\E[\mathbf{L}] \mathbf{1}^N$. On the one hand, $(\E[\mathbf{L}] \mathbf{1}^N)_{ii} = 0 = \mathbf{a}^N_{ii}$ by \eqref{eq:L:diag:property} and the definition of $\mathbf{a}^N$. On the other hand, if $i\neq j$,
\begin{equation*}
(\E[\mathbf{L}] \mathbf{1}^N)_{ij} = \sum_{k=1}^N\sum_{\ell=1}^N \E\left[ \mathbf{L}_{ij,k\ell} \right] \mathbf{1}^N_{k\ell}
= \sum_{k=1}^N \E\left[ \mathbf{L}_{ij,kk} \right] 
= ((1-\lambda)p)^2N^{-1}.
\end{equation*}

We now turn to $\E[\mathbf{L}] \mathbf{a}^N$. Again, we have $(\E[\mathbf{L}] \mathbf{a}^N)_{ii} = 0 = \mathbf{a}^N_{ii}$. To consider the terms with $i\neq j$, let us denote
\begin{itemize}
    \item $k \leftrightarrow \ell$ if and only if $(k,\ell) \in \mathcal{P}_+^2 \cup \mathcal{P}_-^2$ and $k\neq \ell$,
    \item $k \not\leftrightarrow \ell$ if and only if $(k,\ell) \notin \mathcal{P}_+^2 \cup \mathcal{P}_-^2$.
\end{itemize}
Remark that the two relations above are not complementary. Indeed, $k=\ell$ does not satisfy $k \leftrightarrow \ell$ nor $k \not\leftrightarrow \ell$. The cardinality of the sets $\{k \leftrightarrow \ell\}$ and $\{k \not\leftrightarrow \ell\}$ are respectively $N^2[(r_+^N)^2 + (r_-^N)^2 - 1/N]$ and $2 N^2 r_+^N r_-^N$. Hence, if $i\neq j$, we have
\begin{eqnarray*}
    (\E[\mathbf{L}] \mathbf{a}^N)_{ij} &=& \sum_{(k,\ell)\in [N]^2: \ \ell \neq k} \E[\mathbf{L}_{ij,k\ell}] = 
    \frac{(1-\lambda)^2p^2}{N^2} N^2 \left([(r_+^N)^2 + (r_-^N)^2 - 1/N] - 2  r_+^N r_-^N\right)\\
    & = & (1-\lambda)^2p^2 \left( (r_+^N - r_-^N)^2 - N^{-1} \right)
\end{eqnarray*}
which concludes the proof of Step 2.

\medskip
According to Step 2 and the paragraph above it, we expect that the two following quantities are negligible:
\begin{equation*}
    \boldsymbol{\varepsilon}^{N,1} := \mathbf{L}\mathbf{v}^N - \frac{(1-\lambda)^2p^2m(1-m)}{N} \mathbf{a}^N
    \quad \text{and} \quad
    \boldsymbol{\varepsilon}^{N,2} := \mathbf{L}\mathbf{a}^N - D^2 \mathbf{a}^N,
\end{equation*}
where we remind that $D = (1-\lambda)p (r_+ - r_-)$.

\smallskip
{\bf Step 3.} Here we prove that there exists a constant $K>0$ depending only on $\lambda$ such that for all $\delta\in (0,1]$ and $N\geq 2$,
\begin{equation*}        
    \P\left( \| \boldsymbol{\varepsilon}^{N,1} \|_{\infty} \geq K  \sqrt{\frac{\log(N/\delta)}{N^3}} \right) \leq \delta.
\end{equation*}
\smallskip
{\it Proof of Step 3.} 
According to Step 2, we write and decompose $\boldsymbol{\varepsilon}^{N,1}$ as
\begin{equation*}
    \boldsymbol{\varepsilon}^{N,1} = \mathbf{L}\mathbf{v}^N - \E[\mathbf{L}](m(1-m)\mathbf{1}^N) = \mathbf{L}(\mathbf{v}^N - m(1-m)\mathbf{1}^N) + (\mathbf{L} - \E[\mathbf{L}]) (m(1-m)\mathbf{1}^N) 
\end{equation*}
On the one hand, since $(\mathbf{L}\mathbf{1}^N)_{ii} = (\E[\mathbf{L}]\mathbf{1}^N)_{ii} = 0$, we have 
\begin{eqnarray*}
    \| (\mathbf{L} - \E[\mathbf{L}]) (m(1-m)\mathbf{1}^N) \|_{\infty} 
    &=& \max_{i\neq j} \left| \sum_{k=1}^{N} \mathbf{L}_{ij,kk} - \E[\mathbf{L}_{ij,kk}] \right|\\ 
    &=& \frac{(1-\lambda)^2}{N} \max_{i\neq j} \left| \frac{1}{N} \sum_{k=1}^{N} \theta_{ik}\theta_{jk} - p^2 \right|.
\end{eqnarray*}
By Equation~\eqref{eq:Hoeffding:theta:2},
\begin{equation}\label{eq:concentration:step4}
    \P\left( \| (\mathbf{L} - \E[\mathbf{L}]) (m(1-m)\mathbf{1}^N) \|_{\infty} \geq 2(1-\lambda)^2 \sqrt{\frac{\log(4N/\delta)}{2N^3}} \right) \leq \delta.
\end{equation}

On the other hand, since $(\mathbf{L}\mathbf{v}^N)_{ii} = (\mathbf{L}\mathbf{1}^N)_{ii} = 0$, we have 
\begin{eqnarray*}
    \| \mathbf{L}(\mathbf{v}^N - m(1-m)\mathbf{1}^N) \|_{\infty} &=& \max_{i\neq j} \left| \sum_{k=1}^{N} \mathbf{L}_{ij,kk} \left( v_k^N - (m(1-m)) \right)\right| \\
    &\leq& \frac{(1-\lambda)^2}{N^2} N \| v^N - m(1-m)1_N \|_{\infty}.
\end{eqnarray*}
Using \eqref{eq:concentration:v:main} and gathering with~\eqref{eq:concentration:step4} conclude the proof of this step. 

\smallskip
{\bf Step 4.} Here we prove that there exists a constant $K>0$ depending only on $\lambda$ such that for all $\delta\in (0,1]$ and $N\geq 2$,
\begin{equation*}        
    \P\left( \| \boldsymbol{\varepsilon}^{N,2} \|_{\infty} \geq K \sqrt{\frac{\log(N/\delta)}{N}}\right) \leq \delta.
\end{equation*}
\smallskip
{\it Proof of Step 4.}
We decompose $\boldsymbol{\varepsilon}^{N,2}$ as
\begin{equation*}
    \boldsymbol{\varepsilon}^{N,2} = (\mathbf{L} - \E[\mathbf{L}]) \mathbf{a}^N + \left( \E[\mathbf{L}] \mathbf{a}^N - D^2 \mathbf{a}^N \right).
\end{equation*}

On the one hand, using Step 2, it is easy to check that
\begin{equation}\label{eq:EL-D2:aN}
    \| \E[\mathbf{L}] \mathbf{a}^N - D^2 \mathbf{a}^N \|_{\infty} = \left| (r_+^N - r_-^N)^2 - N^{-1} - (r_+ - r_-)^2 \right| \leq KN^{-1}.
\end{equation}

On the other hand, remind the notation $k \leftrightarrow \ell$ and $k \not\leftrightarrow \ell$ introduced in the proof of Step 2, and check that
\begin{eqnarray*}
    \| (\mathbf{L} - \E[\mathbf{L}]) \mathbf{a}^N \|_{\infty} = (1-\lambda)^2 \max_{i\neq j} \left| \frac{1}{N^2} \sum_{k \leftrightarrow \ell} (\theta_{ik}\theta_{j\ell} - p^2) - \frac{1}{N^2} \sum_{k \not\leftrightarrow \ell} (\theta_{ik}\theta_{j\ell} - p^2) \right|.
\end{eqnarray*}
By definition of $k \leftrightarrow \ell$, we have
\begin{align*}
    \sum_{k \leftrightarrow \ell} (\theta_{ik}\theta_{j\ell} - p^2) &= \sum_{(k,\ell) \in \mathcal{P}_+^2} (\theta_{ik}\theta_{j\ell} - p^2) + \sum_{(k,\ell) \in \mathcal{P}_-^2} (\theta_{ik}\theta_{j\ell} - p^2) - \sum_{k=1}^{N} (\theta_{ik}\theta_{jk} - p^2),\\
    \sum_{k \not\leftrightarrow \ell} (\theta_{ik}\theta_{j\ell} - p^2) &= \sum_{(k,\ell) \in \mathcal{P}_+ \times \mathcal{P}_-} (\theta_{ik}\theta_{j\ell} - p^2) + \sum_{(k,\ell) \in \mathcal{P}_- \times \mathcal{P}_+} (\theta_{ik}\theta_{j\ell} - p^2).
\end{align*}
Using Lemma~\ref{lem:Hoeffding:theta} on the five terms in the right hand side above and gathering with~\eqref{eq:EL-D2:aN} conclude the proof of this step
(after adjusting the constant $K$ to get a simpler form for the lower bound).

\medskip
We are now in position to end the proof of the proposition. Let us write
\begin{equation*}
    \mathbf{L}\mathbf{v}^N = \frac{\zeta}{N} \mathbf{a}^N + \boldsymbol{\varepsilon}^{N,1} \quad \text{and} \quad \mathbf{L}\mathbf{a}^N = D^2 \mathbf{a}^N + \boldsymbol{\varepsilon}^{N,2},
\end{equation*}
where $\zeta =(1-\lambda)^2p^2m(1-m)$. By standard linear algebra, we have $\mathbf{L}^0 \mathbf{v}^N = \mathbf{v}^N$, and by induction for all $k\geq 2$,
\begin{equation*}
    \mathbf{L}^k \mathbf{v}^N = \frac{\zeta}{N} \left( D^{2(k-1)}\mathbf{a}^N + \sum_{\ell = 0}^{k-2} D^{2\ell} \mathbf{L}^{(k-2)-\ell} \boldsymbol{\varepsilon}^{N,2} \right) + \mathbf{L}^{k-1} \boldsymbol{\varepsilon}^{N,1}.
\end{equation*}
In particular, denoting $\nu = \max\{D^2, 1-\lambda\} < 1$ and reminding that $\vvvert \mathbf{L} \vvvert_{\infty} \leq (1-\lambda)^2$, one has for all $k\geq 1$,
\begin{equation}
\label{ineq:behavior_of_constant_K}
    \left\| \mathbf{L}^k \mathbf{v}^N - \frac{\zeta}{N} D^{2(k-1)}\mathbf{a}^N \right\|_{\infty} \leq \zeta (k-1) \nu^{k} \frac{\| \boldsymbol{\varepsilon}^{N,2} \|_{\infty}}{N} + (1-\lambda)^{k-1} \| \boldsymbol{\varepsilon}^{N,1} \|_{\infty}.
\end{equation}
To conclude the proof, it suffices to remind the result of Step 1, remark that $\sum_{k=1}^{\infty} D^{2(k-1)} = 1/(1 - D^2)$, use the summability of the sequence $(k \nu^{k})_{k\geq 0}$ and the inequalities of Steps 3 and 4.
\end{proof}

\section{Convergence of $\widehat{\Sigma}^{(1)}$}
\label{sec:proof:of:thm:error:between:empirical:Sigma1:and:Sigma1}

The goal of this section is to prove the following result.

\begin{theorem}
\label{thm:error:between:empirical:Sigma1:and:Sigma1}
Let $\Sigma^{(1)}$ be the matrix defined in \eqref{def:lag_1_Cov_matrix_signed_random_enviroment} and $\widehat{\Sigma}^{(1)}$ its empirical estimate defined in \eqref{def:empirical:Sigma1}. 
There exist positive constants $K_1$, $K_2$ and $K_3$ depending only on $\lambda$ such that 
\begin{equation}
\label{rate:conv:empirical:ave:cross:cov}
\P_{\theta}\Bigg(\vvvert \widehat{\Sigma}^{(1)}-\Sigma^{(1)}\vvvert_2 \geq K_1\sqrt{\frac{N\log(T)\log(N\log(T))}{T}}\Bigg) \leq K_2\Big(\sqrt{\frac{N}{T}}+\frac{1}{N\log(T)}\Big),   
 \end{equation}
for all $\theta\in\{0,1\}^{N\times N}$, as long as the pair $(N,T)$ satisfies $N\geq 2$ and 
\begin{equation*}
T\geq (K_3N\log(T)\log(N\log(T)))\vee 2.
\end{equation*}
\end{theorem}
 An important feature of this result is that, like Theorem \ref{thm:conv:hat:sigma:ag:towards:sigma:ag}, it also holds for all realizations of the random matrix $\theta$. By integrating \eqref{rate:conv:empirical:ave:cross:cov} with respect to $\theta$, we deduce that the conclusion of Theorem  \ref{thm:error:between:empirical:Sigma1:and:Sigma1} is also valid for the probability measure $\P$, under which the environment is random.

The starting point of the proof of Theorem \ref{thm:error:between:empirical:Sigma1:and:Sigma1} is the following observation.
Let $\Tilde{\Sigma}^{(1)}$ be the matrix defined as
\begin{equation}
\label{def:empirical:Sigma1:centered}
\Tilde{\Sigma}^{(1)}=\frac{1}{T}\sum_{t=1}^{T}(X_{t+1}-m^N)(X_{t}-m^N)^{\top},
\end{equation} 
where $m^N$ is defined in \eqref{eq:def:mean:vector}. The matrix $\Tilde{\Sigma}^{(1)}$ can be seen as an intermediate matrix between the covariance matrix $\Sigma^{(1)}$ and its estimator $\widehat{\Sigma}^{(1)}$. Then, observe that 
\begin{equation}
\label{eq1:proof:thm:error:between:empirical:Sigma1:and:Sigma1}
\vvvert \widehat{\Sigma}^{(1)}-\Sigma^{(1)}\vvvert_2\leq \vvvert \widehat{\Sigma}^{(1)}-\tilde{\Sigma}^{(1)}\vvvert_2 + \vvvert \tilde{\Sigma}^{(1)}-\Sigma^{(1)}\vvvert_2. 
\end{equation}
Hence, the proof of Theorem \ref{thm:error:between:empirical:Sigma1:and:Sigma1} reduces to controlling each one of the terms on the RHS of \eqref{eq1:proof:thm:error:between:empirical:Sigma1:and:Sigma1}.

The first term on the RHS of \eqref{eq1:proof:thm:error:between:empirical:Sigma1:and:Sigma1} is controlled by bounding the difference between $\widehat{m}$ and the (conditional) mean vector $m^N$. This is done in Lemma \ref{lem:sigmahat:minus:sigmatilde:in:operator:norm}. 

To control the second term, 
the proof follows a similar strategy to that of Theorem \ref{thm:conv:hat:sigma:ag:towards:sigma:ag}.
Nevertheless, a direct application of \cite[Theorem 4]{ost2020sparse} would lead to a looser upper-bound here.
Indeed, their result relies on Hoeffding inequality, which requires no assumption on the covariances. Bernstein inequality is sharper in our case, since we know that the covariances are of order $N^{-1}$ (Lemma \ref{lem:covariance:produit}). Adapting their proof to our framework leads to Lemma \ref{lem:sigmatilde:minus:sigma:in:operator:norm} below. Finally, Lemma \ref{lemma:bound:on:sigma} is used to control a variance term appearing in Lemma \ref{lem:sigmatilde:minus:sigma:in:operator:norm}.

\begin{lemma}
\label{lem:sigmahat:minus:sigmatilde:in:operator:norm}
There exists a constant $K>0$ depending on $\lambda$ such that the following bounds hold for all $\theta\in\{0,1\}^{N\times N}$:
\begin{gather}
\E_{\theta}\left[\|\widehat{m}-m^N\|^2_2\right]\leq K\frac{N}{T}.
     \label{eq:limit:l2_m}\\
  \P_{\theta}\Bigg(\vvvert \widehat{\Sigma}^{(1)}-\tilde{\Sigma}^{(1)}\vvvert_2\geq \sqrt{\frac{N}{T}}\Bigg)\leq K\sqrt{\frac{N}{T}}.
  \label{eq:control:hatsigma:minus:tilde:sigma:opNorm}
\end{gather}
   
\end{lemma}

\begin{proof}
The proof of \eqref{eq:limit:l2_m} follows from the observation that
\begin{align*}
 \E_{\theta}[\|\widehat{m}-m^{N}\|_2^2]&=\sum_{i=1}^N\E_{\theta}[|\widehat{m}_i-m^{N}_i|_2^2]\\
 &=\sum_{i=1}^N\text{Var}_{\theta}\Big(\frac{1}{T}\sum_{t=1}^TX_{i,t}\Big)=\frac{1}{T^2}\sum_{i=1}^N\sum_{t=1}^T\sum_{s=1}^T\text{Cov}_{\theta}\Big(X_{i,t},X_{i,s}\Big)
 \end{align*}
and \ref{lem:covariance:produit}.
 
 For any vector $u\in\R^N$, $u(T^{-1}\sum_{t=1}^TX_t-\widehat{m})^{\top}=u(\widehat{m}-\widehat{m})^{\top}=0_{N\times N}$, where $0_{N\times N}$ is the matrix with all entries equal to 0, so that we can write
 \begin{align*}
 \widehat{\Sigma}^{(1)}&=\frac{1}{T}\sum_{t=1}^T(X_{t+1}-\widehat{m})(X_{t}-\widehat{m})^{\top}\\
 &=\frac{1}{T}\sum_{t=1}^T(X_{t+1}-m^N)(X_{t}-\widehat{m})^{\top}+(m^N-\widehat{m})\Big(\frac{1}{T}\sum_{t=1}^{T}X_{t}-\widehat{m}\Big)^{\top}\\
 &=\frac{1}{T}\sum_{t=1}^T(X_{t+1}-m^N)(X_{t}-m^N)^{\top}+\Big(\frac{1}{T}(X_{T+1}-X_1)+(\widehat{m}-m^N)\Big)(m^N-\widehat{m})^{\top}\\
 &=\tilde{\Sigma}^{(1)}+\frac{1}{T}(X_{T+1}-X_1)(m^N-\widehat{m})^{\top}-(\widehat{m}-m^N)(\widehat{m}-m^N)^{\top}.
 \end{align*}
As a consequence, we obtain that 
\begin{equation}
\label{eq2:proof:thm:error:between:empirical:Sigma1:and:Sigma1}
\vvvert\widehat{\Sigma}^{(1)}-\tilde{\Sigma}^{(1)}\vvvert_2\leq \frac{1}{T}\|X_{T+1}-X_1\|_2\|m^N-\widehat{m}\|_2+\|m^N-\widehat{m}\|_2^2.    
\end{equation}

By combining \eqref{eq:limit:l2_m}, \eqref{eq2:proof:thm:error:between:empirical:Sigma1:and:Sigma1} with the chain of inequalities $\|X_{T+1}-X_{1}\|_2\leq \|X_{T+1}\|_2+\|X_{1}\|_2\leq 2\sqrt{N}$ and then using Jensen inequality, we obtain for all $\theta\in\{0,1\}^{N\times N}$, that
\begin{align*}
\E_{\theta}\left[\vvvert\widehat{\Sigma}^{(1)}-\tilde{\Sigma}^{(1)}\vvvert_2\right]&\leq \frac{2\sqrt{N}}{T}\sqrt{\E_{\theta}\left[ \|\widehat{m}-m^N\|^2_2\right]}+\E_{\theta}\left[ \|\widehat{m}-m^N\|^2_2\right]\\
&\leq \frac{2\sqrt{K}N}{T^{3/2}}+K\frac{N}{T}\leq \frac{N}{T}(2\sqrt{K}+K).
\end{align*}
The inequality above combined with Chebyshev inequality implies Inequality \eqref{eq:control:hatsigma:minus:tilde:sigma:opNorm}, concluding the proof.
\end{proof}

Combining a coupling result (Lemma \ref{lemma:coupling}) with Matrix Bernstein inequality, we can show that $\Tilde{\Sigma}^{(1)}$ and $\Sigma^{(1)}$ are also uniformly close in the operator norm.

\begin{lemma}
\label{lem:sigmatilde:minus:sigma:in:operator:norm}
Take $1<\renewaltime\leq T$. For all $\theta\in\{0,1\}^{N\times N}$ and $\zeta>0$, the following inequality holds:
\begin{equation*}
 \P_{\theta}\Big(\vvvert \tilde{\Sigma}^{(1)}-\Sigma^{(1)}\vvvert_2>\frac{2N\renewaltime }{T}+\zeta\Big)\leq TN(1-\lambda)^{\renewaltime -2}+2N\renewaltime \exp\left\{-\frac{\lfloor T/\renewaltime \rfloor \zeta^2}{2(\sigma_{\theta}+2N\zeta/3)}\right\},   
\end{equation*}
where $\sigma_{\theta}=\sigma_{\theta,0}\vee \sigma_{\theta,1} $ and for $t\in\{0,1\}$,
\begin{equation*}
\sigma_{\theta,t}=\vvvert\E_{\theta}\Big[\|X_t-m^N\|_2^2(X_{1-t}-m^N)(X_{1-t}-m^N)^{\top}\Big]-(G_0 1_{\{t=0\}} + G_1 1_{\{t=1\}})\vvvert_2,  
\end{equation*}
with $G_0= \Sigma^{(1)}(\Sigma^{(1)})^{\top}$ and $G_1= (\Sigma^{(1)})^{\top}\Sigma^{(1)}$.
\end{lemma}

The proof of Lemma \ref{lem:sigmatilde:minus:sigma:in:operator:norm} is provided in Appendix \ref{sec:coupling:concentration}. 
The lemma below is used to control the constant $\sigma_\theta$ appearing in Lemma \ref{lem:sigmatilde:minus:sigma:in:operator:norm}.

\begin{lemma}
\label{lemma:bound:on:sigma}
For each $\theta\in\{0,1\}^{N\times N}$, let $\sigma_{\theta}$ be as in Lemma \ref{lem:sigmatilde:minus:sigma:in:operator:norm}. 
Then, there exists a constant $K>0$ depending only on $\lambda$ such that $\sigma_{\theta}\leq K N$ for all $\theta\in\{0,1\}^{N\times N}$. 
\end{lemma}
\begin{proof}
We need to prove that there exists a constant $K>0$ such that $\sigma_{\theta,t}\leq KN$ for all $\theta\in\{0,1\}^{N\times N}$ and $t\in \{0,1\}$. We will only prove that $\sigma_{\theta,0}\leq KN$ for all $\theta\in\{0,1\}^{N\times N}$ for some suitable constant $K>0$. The analysis for the case $t=1$ is very similar. 

In the rest of the proof, let $Y_{i,s}=X_{i,s}-\E_{\theta}[X_{i,s}]$ for $i\in [N]$ and $s\in\Z$. Also, denote $G=G_0$.

For all $i,j\in [N]$, observe that 
\[
\mathbb{E}_{\theta}\big[\|X_0-m^N\|_2^2(X_1-m^N)(X_1-m^N)^{\top}_{ij}\big]-G^{\top}_{ij}=\sum_{k=1}^N\text{Cov}_{\theta}[Y_{k,0}Y_{i,1},Y_{k,0}Y_{j,1}]. 
\]
By Items 2 and 5 of \ref{lem:covariance:produit}, we then deduce that for all $i\neq j$,
\begin{align}
\label{eq:1:proof:lemma:bound:on:sigma}
\vert\mathbb{E}_{\theta}\big[\|X_0-m^N\|_2^2(X_1-m^N)(X_1-m^N)^{\top}_{ij}\big]-G^{\top}_{ij}\vert &\leq \sum_{k=1}^N\vert \text{Cov}_{\theta}[Y_{k,0}Y_{i,1},Y_{k,0}Y_{j,1}]\vert  \nonumber\\
&\leq N\Big(\frac{K}{N}\Big)=K,    
\end{align}
for some constant $K>0$ depending only on $\lambda$, and that for all $i\in [N]$,
\begin{equation}
\label{eq:2:proof:lemma:bound:on:sigma}
|d_i|\leq N, \ \text{where} \ d_i:= \mathbb{E}_{\theta}\big[\|X_0-m^N\|_2^2(X_1-m^N)(X_1-m^N)^{\top}_{ii}\big]-G^{\top}_{ii}.    
\end{equation}
Let $d=(d_1,\ldots,d_N)$ and define $D=\text{Diag}(d)$. It follows from \eqref{eq:1:proof:lemma:bound:on:sigma} and \eqref{eq:2:proof:lemma:bound:on:sigma} that 
\begin{multline}
\label{eq:3:proof:lemma:bound:on:sigma}
\sigma_{\theta,0}=\vvvert\mathbb{E}_{\theta}\big[\|X_0-m^N\|_2^2(X_1-m^N)(X_1-m^N)^{\top}\big]-G^{\top}\vvvert_2\leq \vvvert D\vvvert_2\\
+\vvvert d_0(\mathbb{E}_{\theta}\big[\|X_0-m^N\|_2^2(X_1-m^N)(X_1-m^N)^{\top})\vvvert_2\leq \|d\|_{\infty}
\\
+N\vvvert d_0(\mathbb{E}_{\theta}\big[\|X_0-m^N\|_2^2(X_1-m^N)(X_1-m^N)^{\top})\vvvert_{\max}\leq N+NK=N(K+1),
\end{multline}
and the result follows. 
\end{proof}

We are now in position to prove Theorem \ref{thm:error:between:empirical:Sigma1:and:Sigma1}.

\begin{proof}[Proof of Theorem \ref{thm:error:between:empirical:Sigma1:and:Sigma1}]

Combining \eqref{eq1:proof:thm:error:between:empirical:Sigma1:and:Sigma1}, \eqref{eq:control:hatsigma:minus:tilde:sigma:opNorm} and Lemmas \ref{lem:sigmatilde:minus:sigma:in:operator:norm} and \ref{lemma:bound:on:sigma}, we obtain that there exists a constant $K>0$ depending only on $\lambda$ such that for all $1<\renewaltime \leq T$, $\zeta>0$ and $\theta\in\{0,1\}^{N\times N},$ it holds that
\begin{multline}
\label{eq3:proof:thm:error:between:empirical:Sigma1:and:Sigma1}
\P_{\theta}\Big(\vvvert \widehat{\Sigma}^{(1)}-\Sigma^{(1)}\vvvert_2\geq \sqrt{\frac{N}{T}}+\frac{2N\renewaltime }{T}+\zeta\Big)\leq \Bigg(K\sqrt{\frac{N}{T}}+TN(1-\lambda)^{\renewaltime -2}\\
+2N\renewaltime \exp\Big\{-\frac{n\zeta^2}{2N(K+2\zeta/3)}\Big\}\Bigg),    
\end{multline}
where $n=\lfloor T/\renewaltime \rfloor.$ 
Let $\renewaltime =2(1+\log(TN)/\log(1-\lambda)^{-1})$,  
and $\zeta=\sqrt{\frac{8KN\log(2N\renewaltime )}{n}}$ and suppose that the pair $(N,T)$ is such that $N\geq 2$ and 
\[
T\geq K_3 N\log(T)\log(N\log(T)) \vee 2,
\]
where $K_3>0$ is a constant satisfying the following properties:
\begin{enumerate}
    \item $K_3 \log(2)\log(2\log(2))\geq 1$.
    \item $K_3 2\log(2)\log(2\log(2))\geq (1-\lambda)^{-1/2}\vee T_1\vee \exp(8/\log((1-\lambda)^{-1}))$,
    where $T_1=\min\{t\geq 1: t\geq (16\log(t))/\log((1-\lambda)^{-1})\}$.
    \item $K_3\geq \frac{1024 K^{3}}{9\log((1-\lambda)^{-1})}$
\end{enumerate}

{\bf Claim.} For such a pair $(N,T)$, one can check that  following properties holds:
\begin{enumerate}[a)]
    \item $T\geq N$.
    \item $\renewaltime \leq 2\renewaltime \leq 16\log(T)/\log((1-\lambda)^{-1}))\leq T$.
   \item $N\log(T)\geq 16/\log((1-\lambda)^{-1})$.
    \item $n\geq T\log((1-\lambda)^{-1})/(16\log(T))$.
    \item $\zeta\leq 3/(2K)$.
\end{enumerate}

\medskip
{\it Proof of the Claim.} Item a) follows immediately from property 1.
Property 2 ensures that $T\geq (1-\lambda)^{-1/2}\vee T_1$. This, in turn, implies that $2\log(T)/\log((1-\lambda)^{-1})\geq 1$. Using this inequality and item a), we then deduce that 
\begin{align*}
\renewaltime \leq 2\renewaltime &\leq 4(1+\frac{1}{\log((1-\lambda)^{-1})}\log(TN))\\
&\leq  4(1+\frac{2}{\log((1-\lambda)^{-1})}\log(T))\leq \frac{16}{\log((1-\lambda)^{-1})}\log(T).   
\end{align*}
Using the definition of $T_1$ and the fact that $T\geq T_1$, we obtain the last inequality of item b).

Item c) follows from the fact that $N\geq 2$ and $T\geq \exp(8/\log((1-\lambda)^{-1}))$. As for item d), it follows from b) and the inequality $\lfloor z\rfloor \geq z/2$ which is valid for all $z\geq 2$. Finally, using first d), then b) and c), we have that 
\begin{align*}
\zeta=\sqrt{\frac{8KN\log(2N\renewaltime )}{n}}&\leq \sqrt{\frac{128 KN \log(T) \log(2N\renewaltime )}{T\log((1-\lambda)^{-1})}}\\
&\leq \sqrt{\frac{128 KN \log(T) \log(\frac{16N\log(T)}{\log((1-\lambda)^{-1})})}{T\log((1-\lambda)^{-1})}}\\
&\leq \sqrt{\frac{256 KN \log(T) \log(N\log(T))}{T\log((1-\lambda)^{-1})}}. 
\end{align*}
Hence, using the fact that $T\geq K_3N\log(T)\log(N\log(T))$ and Property 3, we deduce that 
\[
\zeta\leq \sqrt{\frac{256 KN \log(T) \log(N\log(T))}{T\log((1-\lambda)^{-1})}}\leq \sqrt{\frac{256 K}{\log((1-\lambda)^{-1})K_3}}\leq 3/(2K),
\]
proving e).
\medskip

We now conclude the proof using Claim 1.
By Items b) and e) of Claim 1, we have that $\renewaltime \leq T$ and $\zeta\leq 3/(2K)$, so that  
\begin{align}
\label{eq4:proof:thm:error:between:empirical:Sigma1:and:Sigma1}
K\sqrt{\frac{N}{T}}+TN(1-\lambda)^{\renewaltime -2}
&+2N\renewaltime \exp\Big\{-\frac{n\zeta^2}{2N(K+2\zeta/3)}\Big\}\nonumber\\
&\leq K\sqrt{\frac{N}{T}}+\frac{1}{TN}+\frac{1}{2N\renewaltime }\leq (K+1)\sqrt{\frac{N}{T}}+\frac{\log(1-\lambda)^{-1}}{4N\log(TN)},
\end{align}
where in the second inequality we have also used that $2\log(TN)/\log((1-\lambda)^{-1})\leq \renewaltime $.

Moreover, using the definition of $\renewaltime $, a), b), d) and c), we can also show that
\begin{align*}
\sqrt{\frac{N}{T}}+\frac{2N\renewaltime }{T}+\zeta&\leq  5\sqrt{\frac{N}{T}}+\frac{8}{\log(1-\lambda)^{-1}}\frac{N\log(T)}{T}+\sqrt{\frac{8KN\log(2N\renewaltime )}{n}}\\
&\leq  5\sqrt{\frac{N}{T}}+\frac{8N\log(T)}{T\log(1-\lambda)^{-1}}+\sqrt{\frac{8KN\log(\frac{16N\log(T)}{\log(1-\lambda)^{-1}})16\log(T)}{T\log(1-\lambda)^{-1}}}\\
&\leq  5\sqrt{\frac{N}{T}}+\frac{8N\log(T)}{T\log(1-\lambda)^{-1}}+\sqrt{\frac{256KN\log(N\log(T))\log(T)}{T\log(1-\lambda)^{-1}}}\\
&\leq K_1 \sqrt{\frac{N\log(T)\log(N\log(T))}{T}},   
\end{align*}
where $K_1>0$ depends only on $\lambda$ and $K$. 
Putting together \eqref{eq3:proof:thm:error:between:empirical:Sigma1:and:Sigma1}, \eqref{eq4:proof:thm:error:between:empirical:Sigma1:and:Sigma1} and the above inequality, the result follows. 
\end{proof}

\section{Coupling and Concentration}
\label{sec:coupling:concentration}

The goal of this section is prove the concentration result stated in Lemma \ref{lem:sigmatilde:minus:sigma:in:operator:norm}. To that end, we will use the
coupling result below. In words, we couple our model of interest with a similar Markov chain that independently starts anew every $\renewaltime$ time steps. 
We show that the coupling error vanishes exponentially with respect to $\renewaltime$.

\begin{lemma}
\label{lemma:coupling}
Let $\renewaltime >1$ and denote ${\cal R}=\{3+a\renewaltime : a\in\N\}$. Take $\theta\in\{0,1\}^{N\times N}$ and let $(Z_{t})_{t\in {\cal R}}$ be a sequence of i.i.d random vectors with distribution $\pi_{\theta}$, where  $\pi_{\theta}$ denotes the invariant distribution of the Markov chain with transition probability defined in \eqref{def:transition_prob_1}, conditionally on $\theta$.
Let $(U_{i,t})_{i\in[N],t\geq 2}$ be a sequence of i.i.d random variables uniformly distributed on $[0,1]$, independent of everything else. Let $(X_t,\tilde{X}_t)_{t\geq 1}$ be the bi-variate stochastic chain defined in the following way: for $t=1$, set $\tilde{X}_1=X_1\sim\pi_{\theta}$ and recursively for all $t\geq 2$ and $i\in [N]$, define 
\begin{equation*}
X_{i,t}={ 1}_{\big\{U_{i,t}\leq p_{\theta,i}(X_{t-1})\big\}} \ \text{and} \ \tilde{X}_{i,t}=\indiq_{\{t\not\in {\cal R}\}}{1}_{\big\{U_{i,t}\leq p_{\theta,i}(\tilde{X}_{t-1})\big\}}+\indiq_{\{t\in {\cal R}\}}Z_{i,t},   
\end{equation*}
where $p_{\theta,i}(x)$ is defined in  \eqref{def:transition_prob_2}. Then, the following properties hold:
\begin{enumerate}
    \item The chain $(X_t)_{t\geq 1}$ is a stationary version of Markov chain with transition probability defined in \eqref{def:transition_prob_1}, conditionally on $\theta$.
    \item The random vectors $(\tilde{X}_{t-2},\tilde{X}_{t-1})_{t\in {\cal R}}$ are independent.
    \item For all $t\in{\cal R}$, $(\tilde{X}_{t-2},\tilde{X}_{t-1})$ is distributed as $(X_{t-2},X_{t-1})$. Also, if $X_{t-2}=\tilde{X}_{t-2}$, then $X_{t-1}=\tilde{X}_{t-1}$.
    \item Denoting $\mathbb{Q}_{\theta}$ the joint distribution of $(X_{t}, \tilde{X}_t)_{t\geq 1}$ built as above, then for all $t\in {\cal R},$ 
    \begin{equation*}
\mathbb{Q}_{\theta}\big(X_{t-2}\neq \tilde{X}_{t-2}\big)\leq N(1-\lambda)^{\renewaltime -2}.   
\end{equation*}
\end{enumerate}
\end{lemma}

\begin{proof}
Items 1,2 and 3 follow from classical arguments of Markov chain theory. Hence, we only prove Item 4.

Throughout the proof, we write $\E_{\mathbb{Q}_{\theta}}$ to denote the expectation taken with respect to $\mathbb{Q}_{\theta}.$ 
Note that by construction, $p_{\theta,i}(X_{t-1})$ and $p_{\theta,i}(\tilde{X}_{t-1})$ are independent from $U_{t,j},j\in[n]$ for all $i\in[N]$ and $t\geq 2$, so that for all $i\in[N]$ and $t\not\in {\cal R},$
\begin{align}
\label{eq:1:proof:lemma:coupling}
 \E_{\mathbb{Q}_{\theta}}[|X_{i,t} - \tilde{X}_{i,t}|]&= \mathbb{Q}_{\theta}\{X_{i,t}\neq \tilde{X}_{i,t}\}\nonumber \\
 &= \mathbb{Q}_{\theta}\{U_{i,t}\mbox{ lies between }p_{\theta,i}(X_{t-1})\mbox{ and }p_{\theta,i}(\tilde{X}_{t-1})\}\nonumber\\
 &= \mathbb{E}_{\mathbb{Q}_{\theta}}[|p_{\theta,i}(X_{t-1}) - p_{\theta,i}(\tilde{X}_{t-1})|].   
\end{align}
Next using the definition of $p_{\theta,i}$, i.e. Equation \eqref{def:transition_prob_2},
we obtain that
\begin{align*}
|p_{\theta,i}(X_{t-1}) - p_{\theta,i}(\tilde{X}_{t-1})| & \leq (1-\lambda)N^{-1}\sum_{j=1}^{N}\theta_{ij}|X_{j,t-1}-\tilde{X}_{j,t-1}|.
\end{align*}
Taking the expectation in both sides of the above inequality and then using \eqref{eq:1:proof:lemma:coupling}, we deduce that for $t\not\in{\cal R}$ and $i\in [N]$,
\begin{align*}
\E_{\mathbb{Q}_{\theta}}[|X_{i,t} - \tilde{X}_{i,t}|]&\leq 
(1-\lambda)N^{-1}\sum_{j=1}^N\theta_{ij}\mathbb{E}_{\mathbb{Q}_{\theta}}[|X_{j,t-1}-\tilde{X}_{j,t-1}|]\\
&\leq 
(1-\lambda)\max_{j\in [N]}\mathbb{E}_{\mathbb{Q}_{\theta}}[|X_{j,t-1}-\tilde{X}_{j,t-1}|].
\end{align*}
In the second inequality above, we have also used that $N^{-1}\sum_{j=1}^N\theta_{ij}\leq 1$. Now, since $i\in[N]$ is arbitrary, it follows that for $t\not\in{\cal R}$,
\[
\max_{i\in[N]}\mathbb{E}_{\mathbb{Q}_{\theta}}[|X_{i,t} - \tilde{X}_{i,t}|]\leq (1-\lambda)\max_{j\in[N]}\mathbb{E}_{\mathbb{Q}_{\theta}}[|X_{t-1,j}-\tilde{X}_{t-1,j}|].
\]

Finally, take $t\in{\cal R}\setminus\{3\}$, we observe that $t-2-s\not\in{\cal R}$ for all $s=0,\ldots, \renewaltime -3$, so that applying the above inequality recursively yields
\begin{align*}
\max_{i\in[N]}\mathbb{E}_{\mathbb{Q}_{\theta}}[|X_{i,t-2} - \tilde{X}_{i,t-2}|]&\leq (1-\lambda)^{(\renewaltime -2)}\max_{i\in[N]}\mathbb{E}_{\mathbb{Q}_{\theta}}[|X_{i,t-2-(\renewaltime -3)} - \tilde{X}_{i,t-2-(\renewaltime -3)}|]\nonumber\\
&\leq (1-\lambda)^{(\renewaltime -2)}.    
\end{align*}
which, in turn, implies 
\begin{align*}
\mathbb{Q}_{\theta}\big(X_{t-2}\neq \tilde{X}_{t-2}\big)\leq \sum_{i=1}^N\mathbb{E}_{\mathbb{Q}_{\theta}}[|X_{i,t-2} - \tilde{X}_{i,t-2}|] \leq N (1-\lambda)^{\renewaltime -2}.    
\end{align*}
For $t=3$, we have $X_{t-2}=X_{1}=\tilde{X}_1=\tilde{X}_{t-2}$ by construction. Hence the result holds automatically also for $t=3$.
\end{proof} 

Once we have Lemma \ref{lemma:coupling} at our disposal, we are in position to prove Lemma \ref{lem:sigmatilde:minus:sigma:in:operator:norm}.

\begin{proof}[Proof of Lemma \ref{lem:sigmatilde:minus:sigma:in:operator:norm}] 
Let $n = \lfloor T/\renewaltime \rfloor$. For each $1\leq r\leq \renewaltime $, denote
$\mathbb{T}_{r}=\{r+a\renewaltime : a\in\{0,1,\ldots, n-1\}\}$ and define
 \[
M_r=\frac{1}{n}\sum_{t\in\mathbb{T}_r}(X_{t+1}-m^N)(X_{t}-m^N)^{\top}-\Sigma^{(1)}.
 \]
 Also, define
 \[
E=\begin{cases}
\frac{1}{T}\sum_{t=n\renewaltime +1}^{T} \left( (X_{t+1}-m^N)(X_{t}-m^N)^{\top}-\Sigma^{(1)} \right), \ \text{if} \ n\renewaltime <T,\\
0_{N\times N}, \ \text{if} \ n\renewaltime =T,
\end{cases}
 \]
 where $0_{N\times N}$ is the $N$-by-$N$ matrix with all entries equal to $0$.
 
With this notation, we can write
\[
\tilde{\Sigma}^{(1)}-\Sigma^{(1)}=\frac{n\renewaltime }{T}\Big(\frac{1}{\renewaltime }\sum_{r=1}^\renewaltime M_r\Big)+E.
\]
Since $n\renewaltime /T \leq 1$ we then deduce from the above identity that 
\begin{equation}
\label{eq1:proof:lem:sigmatilde:minus:sigma:in:operator:norm}\vvvert\tilde{\Sigma}^{(1)}-\Sigma^{(1)}\vvvert_2\leq \Big\vvvert \frac{1}{\renewaltime }\sum_{r=1}^\renewaltime M_r \Big\vvvert_2+\vvvert E\vvvert_2.
\end{equation}
To deal with the second term on the RHS of \eqref{eq1:proof:lem:sigmatilde:minus:sigma:in:operator:norm}, we first observe that for all $t\in\Z$,
\[
\vvvert (X_{t+1}-m^N)(X_{t}-m^N)^{\top}\vvvert_2\leq N \vvvert (X_{t+1}-m^N)(X_{t}-m^N)^{\top}\vvvert_{\max}\leq N,
\]
which also implies that $\vvvert\Sigma^{(1)}\vvvert_2\leq N$, and then combine these inequalities with the fact that $T-n\renewaltime \leq \renewaltime $ to obtain $\vvvert E\vvvert_2\leq 2N\renewaltime /T.$ 
To deal with the first term on the RHS of \eqref{eq1:proof:lem:sigmatilde:minus:sigma:in:operator:norm}, we proceed as follows. By the union bound and the stationarity of the chain conditionally on $\theta$, we have that for all $\zeta>0$ and $\theta\in\{0,1\}^{N\times N}$,
\[
\P_{\theta}\Big(\Big\vvvert \frac{1}{\renewaltime }\sum_{r=1}^\renewaltime M_r \Big\vvvert_2>\zeta\Big)\leq \sum_{r=1}^\renewaltime \P_{\theta}\Big(\vvvert M_r\vvvert_2>\zeta\Big)=\renewaltime \P_{\theta}\Big(\vvvert M_1\vvvert_2>\zeta\Big).
\]
Putting together \eqref{eq1:proof:lem:sigmatilde:minus:sigma:in:operator:norm} and the above upper bounds, we obtain for all $\theta\in\{0,1\}^{N\times N}$ and $\zeta>0$,  
\begin{equation}
\label{eq2:proof:lem:sigmatilde:minus:sigma:in:operator:norm}\P_{\theta
}\Big(\vvvert \tilde{\Sigma}^{(1)}-\Sigma^{(1)}\vvvert_2 >\frac{2N\renewaltime }{T}+\zeta\Big) \leq \P_{\theta
}\Big(\Big\vvvert \frac{1}{\renewaltime }\sum_{r=1}^\renewaltime M_r \Big\vvvert_2>\zeta\Big)\leq \renewaltime \P_{\theta}\Big(\vvvert M_1\vvvert_2>\zeta\Big).
\end{equation}

Let $(\tilde{X}_t)_{t\geq 1}$ the chain built in Lemma \ref{lemma:coupling} and define 
\[
\tilde{M}_{1}=n^{-1}\sum_{t\in \mathbb{T}_1}(\tilde{X}_{t+1}-m^N)(\tilde{X}_{t}-m^N)^{\top}-\Sigma^{(1)}.  
\]
By Items 1, 3 and 4 of Lemma  \ref{lemma:coupling}, we have that 
\begin{align}
\label{eq3:proof:lem:sigmatilde:minus:sigma:in:operator:norm}
\P_{\theta}(\vvvert M_1 \vvvert_2 >\zeta)&=\mathbb{Q}_{\theta}(\vvvert M_1 \vvvert_2 >\zeta)\nonumber\\
&\leq \mathbb{Q}_{\theta}(\vvvert \tilde{M}_1 \vvvert_2 >\zeta)+\mathbb{Q}_{\theta}\Big(\cup_{t\in \mathbb{T}_1}\{Y_{t}\neq X_{t}\}\Big)\nonumber\\
&\leq  \mathbb{Q}_{\theta}(\vvvert \tilde{M}_1 \vvvert_2 >\zeta)+nN(1-\lambda)^{\renewaltime -2}.
\end{align}
Items 2 and 3 of Lemma \ref{lemma:coupling} imply that the matrix $\tilde{M}_1$ is a sum of i.i.d random matrices. Therefore, we can apply matrix Bernstein inequality \cite[Theorem 6.1.1]{tropp2015introduction}, to deduce that
\[
\mathbb{Q}_{\theta}(\vvvert \tilde{M}_1 \vvvert_2 >\zeta)\leq 2N\exp\left\{-\frac{n\zeta^2}{2(\sigma_{\theta}+2N\zeta/3)}\right\}. 
\]
Combining \eqref{eq2:proof:lem:sigmatilde:minus:sigma:in:operator:norm}, \eqref{eq3:proof:lem:sigmatilde:minus:sigma:in:operator:norm} and the above inequality, the result follows.  
\end{proof}

\section{Minimax lower bounds}
\label{sec:proof:lower:bounds}

\subsection{Misclassification rate}
\label{sec:lower:bound:missclassification}
The goal of this section is to establish the minimax lower bound for the misclassification rate stated in Theorem \ref{thm:minmax:lower:bound}.

Throughout this section, we denote $N_+=\lceil Nr_+ \rceil$ and $N_-=N-N_+$. In what follows, we will assume that $1/2\leq r_+<1$. The analysis is similar for the case $0<r_+<1/2$.
For later use, let us observe that $N_+\geq N/2\geq N_-$ under the regime $1/2\leq r_+<1$. 

Recall that ${\cal S}=\{S\subset[N]: \ |S|= N_+\}$. In the sequel, each subset $S\in {\cal S}$ plays the role of the community $\mathcal{P}_+$.
Recall that for each $\theta\in\{0,1\}^{N\times N}$ and $S\in {\cal S}$, we write $\P_{\theta,S}$ to denote the probability measure under which the sequence of random vectors $X=(X_t)_{t\in\Z}$ is distributed as 
our model of interest, where the communities $({\cal P}_+,{\cal P}_-)=(S,S^c).$ In particular, for all $t\in\Z$ and $x,y\in\{0,1\}^N$, it holds that 
\begin{equation}
\label{def1:transition:prob:S:theta}
P_{\theta,S}(y|x):=\P_{\theta,S}(X_{t}=y|X_{t-1}=x)=\prod_{i=1}^N (p_{\theta,S,i}(x))^{y_i}(1-p_{\theta,S,i}(x))^{1-y_i},    
\end{equation}
where for each $i\in [N]$, 
\begin{equation}
\label{def2:transition:prob:S:theta}
p_{\theta,S,i}(x)=\mu+\frac{(1-\lambda)}{N}\Big(\sum_{j\in S}\theta_{ij}x_j+\sum_{j\in S^c}\theta_{ij}(1-x_j)\Big).    
\end{equation}

Moreover, let $\pi_{\theta,S}$ denote the unique stationary distribution of the sequence $X$ under $\P_{\theta,S}$, that is, the distribution such that $\pi_{\theta,S}(y)=\sum_{x\in\{0,1\}^N}\pi_{\theta,S}(x)P_{\theta,S}(y|x)$ for all $y\in\{0,1\}^N$, where $P_{\theta,S}(y|x)$ is the transition matrix defined in \eqref{def1:transition:prob:S:theta}.

Furthermore, let $\P_S$ denote the probability measure under which $\theta=(\theta_{ij})_{i,j\in [N]}$ is an i.i.d. sequence of random variables with distribution $\text{Ber}(p)$ and the sequence $X$ is distributed according $\P_{\theta,S}$ conditionally on $\theta$, i.e., $\P_S(X\in \cdot|\theta)=\P_{S,\theta}(X\in\cdot)$.   
Also, for $T\geq 1$, we denote $\P_{\theta,S}^{(1:T)}$ and $\P_{S}^{(1:T)}$ the law of $(X_1,\ldots, X_{T})$ under $\P_{\theta,S}$ and $\P_S$ respectively.

Finally, we note that $|S\setminus V|=|V\setminus S|$ for any $S,V\in\mathcal{S}$. To see this, first write $|S|=|S\setminus V|+|S\cap V|$ and $|V|=|V\setminus S|+|S\cap V|$ and then use the fact that $|V|=N_+=|S|.$ 

The key technical step of this section is to study the modifications of our model as $S$ varies in $\mathcal{S}$. These modifications are measured in terms of total variation distance $d_{TV}$ or Kullback-Leibler divergence.

Recall that for each $K>0$ and $u\geq 0$, we denote $\tau_K(u)=\sqrt{1-\exp(-Ku)}$.
The main step in deriving the lower bound stated in Theorem \ref{thm:minmax:lower:bound} is to prove the following result.

\begin{proposition} 
\label{prop:upper:bound:total:variation}
Let $T\geq 0$ be an integer. There exist positive constants $K_1$ and $K_2$ depending only on $\mu$ and $\lambda$ such that 
\[
d_{TV}\Big(\P_{S}^{(1:T+1)}, \P_{V}^{(1:T+1)}\Big)\leq \tau_{K_1}\Big(\frac{T}{N}\Big)+2\tau_{K_2}\Big(N^{-1} \Big\lceil\frac{\log(2N)}{\log((1-\lambda)^{-1})}\Big\rceil\Big).
\]
for all $S,V\in\mathcal{S}$ such that $|S\setminus V|=1$. 
\end{proposition}

To show Proposition  \ref{prop:upper:bound:total:variation}, we will use three auxiliary results which involve the Kullback-Leibler divergence denoted by $\operatorname{KL}(\cdot \,\|\, \cdot)$.

\begin{lemma} 
\label{lem:upper:bound:kulllback:leibler:divergence}
There exists a positive constant $K$ depending only $\mu$ and $\lambda$ such that 
\[
\operatorname{KL}\big(P_{\theta,S}(\cdot|x) \,\big\|\, P_{\theta,V}(\cdot|x)\big)\leq KN^{-1}.
\]
for all $x\in\{0,1\}^{N}$, $\theta\in\{0,1\}^{N\times N}$ and $S,V\in\mathcal{S}$ such that $|S\setminus V|=1$.
\end{lemma}
\begin{proof}
Fix $x\in\{0,1\}^N$, $\theta\in\{0,1\}^{N\times N}$ and $S,V\in\mathcal{S}$ such that $|S\setminus V|=1$.
Using that $P_{\theta,S}(\cdot|x)=\otimes_{i=1}^N\text{Ber}(p_{\theta,S,i}(x))$ and $P_{\theta,V}(\cdot|x)=\otimes_{i=1}^N\text{Ber}(p_{\theta,V,i}(x))$, we can check that
\begin{equation}
\label{eq1:proof:lem:upper:bound:kulllback:leibler:divergence}
\operatorname{KL}\big(P_{\theta,S}(\cdot|x) \,\big\|\, P_{\theta,V}(\cdot|x)\big)=\sum_{i=1}^N \operatorname{KL}\big(\text{Ber}(p_{\theta,S,i}(x)) \,\big\|\, \text{Ber}(p_{\theta,V,i}(x))\big). 
\end{equation}  
Next, using that for all $p,q\in (0,1)$,
\[
\operatorname{KL}\big(\text{Ber}(p) \,\big\|\, \text{Ber}(q)\big)\leq \frac{(p-q)^2}{q(1-q)},
\] 
and that, since $0<\mu\leq p_{\theta,S,i}(x) \leq 1-(\lambda-\mu)<1$ for all $\theta\in \{0,1\}^{N\times N}$, $S\in \mathcal{S}$ and $i\in [N]$,

\[
\min_{W\in \{S,V\}}\min_{i\in [N]}p_{\theta,W,i}(x)(1-p_{\theta,W,i}(x))\geq C:=\min\{\mu(1-\mu),(\lambda-\mu)(1-(\lambda-\mu))\}>0, 
\]
we obtain from \eqref{eq1:proof:lem:upper:bound:kulllback:leibler:divergence} that 
\begin{equation}
\label{eq2:proof:lem:upper:bound:kulllback:leibler:divergence}
\operatorname{KL}\big(P_{\theta,S}(\cdot|x) \,\big\|\,  P_{\theta,V}(\cdot|x)\big)\leq \frac{1}{C}\sum_{i=1}^N \big(p_{\theta,S,i}(x)-p_{\theta,V,i}(x)\big)^2. 
\end{equation}
Now, since $|V\setminus S|=|S\setminus V|=1$, we have that 
\begin{align}
\label{eq3:proof:lem:upper:bound:kulllback:leibler:divergence}
\Bigg|\sum_{j\in S}\theta_{ij}x_j-\sum_{j\in V}\theta_{ij}x_j\Bigg|&=  \Bigg|\sum_{j\in S\setminus V}\theta_{ij}x_j-\sum_{j\in V\setminus S}\theta_{ij}x_j\Bigg|\nonumber \\
&\leq \sum_{j\in S\setminus V}|\theta_{ij}x_j|+\sum_{j\in V\setminus S}|\theta_{ij}x_j|\leq |S\setminus V|+|V\setminus S|=2.
\end{align}
Similarly, we also have that 
\begin{equation}
\label{eq4:proof:lem:upper:bound:kulllback:leibler:divergence}
\Bigg|\sum_{j\in S^c}\theta_{ij}(1-x_j)-\sum_{j\in V^c}\theta_{ij}(1-x_j)\Bigg| \leq 2.
\end{equation}
Combining \eqref{def2:transition:prob:S:theta}, \eqref{eq3:proof:lem:upper:bound:kulllback:leibler:divergence} and \eqref{eq4:proof:lem:upper:bound:kulllback:leibler:divergence}, we then deduce that
\[
|p_{\theta,S,i}(x)-p_{\theta,V,i}(x)|\leq \frac{4(1-\lambda)}{N}.
\]
Finally, using the above inequality in  \eqref{eq2:proof:lem:upper:bound:kulllback:leibler:divergence}, we obtain that
\[
\operatorname{KL}\big(P_{\theta,S}(\cdot|x) \,\big\|\,  P_{\theta,V}(\cdot|x)\big)\leq \frac{1}{C}N \Big(16(1-\lambda)^2N^{-2}\Big) = \frac{16(1-\lambda)^2}{C} N^{-1}.
\]
\end{proof}

In the sequel, let us denote
\begin{equation*}
 \nu^{(t)}_{\theta,S|x}(y)=\P_{\theta,S}(X_t=y|X_0=x),   
\end{equation*}
for all $x, y\in\{0,1\}^N, S\in {\cal S}, t\geq 0$ and $\theta\in\{0,1\}^{N\times N}$. 
The second auxiliary result reads as follows.

\begin{lemma}
\label{lem:upper:bound:Dtv:and:kL:div:Xt:given:x:theta:S}
 There exists a constant $K>0$ depending only on $\mu$ and $\lambda$ such that 
 \begin{gather}
 d_{TV}\Big( \nu^{(t)}_{\theta,S|x} , \nu^{(t)}_{\theta,S|y} \Big)\leq N(1-\lambda)^t \indiq_{\{x\neq y\}} \label{ineq:dTV:P:theta:S:Xt:given:x:and:P:theta:S:Xt:given:y},\\
 \label{ineq:KL:P:theta:S:Xt:given:x:and:P:theta:V:Xt:given:x}
 \operatorname{KL}\Big( \nu^{(t)}_{\theta,S|x} \,\|\, \nu^{(t)}_{\theta,V|x} \Big)\leq K\frac{t}{N},
 \end{gather}
for all $\theta\in\{0,1\}^{N\times N}$, $x,y\in\{0,1\}^{N}$, $t\geq 0$ and $S,V\in \mathcal{S}$ such that $|S\setminus V|=1$.
\end{lemma}
\begin{proof}
We start proving \eqref{ineq:dTV:P:theta:S:Xt:given:x:and:P:theta:S:Xt:given:y} with a coupling argument. Take $\theta\in\{0,1\}^{N\times N}$, $S\in\mathcal{S}$ and $y,x\in\{0,1\}^{N}$. Next, consider $(U_{i,t})_{i\in[N],t\geq 1}$ a sequence of i.i.d random variables uniformly distributed on $[0,1]$ and let $(X_t,\widehat{X}_t)_{t\geq 0}$ be the bi-variate stochastic chain defined in the following way. For $t=0$, set $X_0=x$ and $\widehat{X}_0=y$, and recursively for all $t\geq 1$ and $i\in [N]$, define 
\begin{equation*}
X_{i,t}={ 1}_{\big\{U_{i,t}\leq p_{\theta,S,i}(X_{t-1})\big\}} \ \text{and} \ \widehat{X}_{i,t}={1}_{\big\{U_{i,t}\leq p_{\theta,S,i}(\widehat{X}_{t-1})\big\}},   
\end{equation*}
where $p_{\theta,S,i}(\cdot)$ is defined in \eqref{def2:transition:prob:S:theta}. 

Clearly, the chains $(X_t)_{t\geq 1}$ and $(\widehat{X}_t)_{t\geq 1}$ are Markov chains with transition probabilities defined in \eqref{def1:transition:prob:S:theta}, starting from $x$ and $y$ respectively.  
Moreover, denoting $\mathbb{Q}_{\theta,S,y,x}$ the joint distribution of $(X_t,\widehat{X}_t)_{t\geq 0}$ and proceeding similar as in the proof of Item 4 of Lemma \ref{lemma:coupling}, we can show that for all $t\geq 0$,
\begin{equation}
\label{proof:lem:upper:bound:Dtv:and:kL:div:Xt:given:x:theta:S}
\mathbb{Q}_{\theta,S,y,x}\Big(X_{t}\neq \widehat{X}_t\Big)\leq N(1-\lambda)^t \indiq_{\{x\neq y\}}.
\end{equation}
Now, since  
\[
d_{TV}\Big( \nu^{(t)}_{\theta,S|x} , \nu^{(t)}_{\theta,S|y} \Big)\leq \mathbb{Q}_{\theta,S,y,x}\Big(X_{t}\neq \widehat{X}_t\Big),  
\]
and $S\in\mathcal{S}$, $\theta$, $x$, $y$ and $t\geq 0$ are arbitrary, Inequality \eqref{ineq:dTV:P:theta:S:Xt:given:x:and:P:theta:S:Xt:given:y} follows from the above inequality and \eqref{proof:lem:upper:bound:Dtv:and:kL:div:Xt:given:x:theta:S}.

Hence, it only remains to prove \eqref{ineq:KL:P:theta:S:Xt:given:x:and:P:theta:V:Xt:given:x}. 
The key step to prove \eqref{ineq:KL:P:theta:S:Xt:given:x:and:P:theta:V:Xt:given:x} is to establish the following result.
\smallskip

{\bf Claim.} There exists $K>0$ depending only on $\lambda$ and $\mu$ such that 
\[
\operatorname{KL}\big(\nu^{(t+1)}_{\theta,S|x} \,\|\, \nu^{(t+1)}_{\theta,V|x}\big)\leq \operatorname{KL}\big(\nu^{(t)}_{\theta,S|x} \,\|\, \nu^{(t)}_{\theta,V|x}\big)+KN^{-1},
\]
for all $\theta\in\{0,1\}^{N\times N}$, $S,V\in \mathcal{S}$ such that $|S\setminus V|=1$, 
 $x\in\{0,1\}^{N}$ and $t\geq 0$.

\medskip
{\it Proof of the Claim}.
First, take $\theta\in\{0,1\}^{N\times N}$, $S,V\in \mathcal{S}$ such that $|S\setminus V|=1$, $x\in\{0,1\}^{N}$ and $t\geq 0$, and denote
\begin{equation*}
 \nu^{(t,t+1)}_{\theta,S|x}(y,z)=\P_{\theta,S}(X_t=y,X_{t+1}=z | X_0=x), \ y,z\in\{0,1\}^N.   
\end{equation*}
We can write  
\begin{align}
\nu^{(t,t+1)}_{\theta,S|x}(y,z)&=\nu^{(t)}_{\theta,S|x}(y)P_{\theta,S}(z|y)
\label{ineq2:proof:lem:upper:bound:Dtv:and:kL:div:Xt:given:x:theta:S}
\\
\label{ineq3:proof:lem:upper:bound:Dtv:and:kL:div:Xt:given:x:theta:S}
&=\nu^{(t+1)}_{\theta,S|x}(z)\P_{\theta,S}(X_{t}=y|X_{t+1}=z,X_0=x),
\end{align}
for all $y,z\in\{0,1\}^N$.

Next, using \eqref{ineq2:proof:lem:upper:bound:Dtv:and:kL:div:Xt:given:x:theta:S}, check that 
\begin{multline}
\label{ineq4:proof:lem:upper:bound:Dtv:and:kL:div:Xt:given:x:theta:S}
\operatorname{KL}\big(\nu^{(t,t+1)}_{\theta,S|x} \,\|\,  \nu^{(t,t+1)}_{\theta,V|x}\big) =\operatorname{KL}\big(\nu^{(t)}_{\theta,S|x} \,\|\, \nu^{(t)}_{\theta,V|x}\big) \\
+ \sum_{y\in\{0,1\}^N}\nu^{(t)}_{\theta,S|x}(y)\operatorname{KL}\big(P_{\theta,S}(\cdot|y) \,\|\,  P_{\theta,V}(\cdot|y)\big).    
\end{multline}
Then, using now \eqref{ineq3:proof:lem:upper:bound:Dtv:and:kL:div:Xt:given:x:theta:S}, check also that
\begin{multline}
\label{ineq5:proof:lem:upper:bound:Dtv:and:kL:div:Xt:given:x:theta:S}
\operatorname{KL}\big(\nu^{(t,t+1)}_{\theta,S|x} \,\|\,  \nu^{(t,t+1)}_{\theta,V|x}\big)
=\operatorname{KL}\big(\nu^{(t+1)}_{\theta,S|x} \,\|\,  \nu^{(t+1)}_{\theta,V|x}\big)+\\
\sum_{z\in\{0,1\}^N}\nu^{(t+1)}_{\theta,S|x}(z)\operatorname{KL}\big(\P_{\theta,S}(X_{t}=\cdot|X_{t+1}=z,X_0=x) \,\|\,  \P_{\theta,V}(X_{t}=\cdot|X_{t+1}=z,X_0=x)\big).    
\end{multline}
Combining \eqref{ineq4:proof:lem:upper:bound:Dtv:and:kL:div:Xt:given:x:theta:S}, \eqref{ineq5:proof:lem:upper:bound:Dtv:and:kL:div:Xt:given:x:theta:S} and using the positivity of the Kullback-Leibler divergence, it then follows that   
\begin{multline*}
\operatorname{KL}\big(\nu^{(t+1)}_{\theta,S|x} \,\|\,  \nu^{(t+1)}_{\theta,V|x}\big)\\
\leq \operatorname{KL}\big(\nu^{(t)}_{\theta,S|x} \,\|\,  \nu^{(t)}_{\theta,V|x}\big)+\sum_{y\in\{0,1\}^N}\nu^{(t)}_{\theta,S|x}(y)\operatorname{KL}\big(P_{\theta,S}(\cdot|y) \,\|\,  P_{\theta,V}(\cdot|y)\big).    
\end{multline*}
Applying Lemma \ref{lem:upper:bound:kulllback:leibler:divergence} to the above inequality and then using the fact that $\nu^{(t)}_{\theta,S|x}$ is a probability measure, we conclude the proof of the claim.

\smallskip
Finally, starting from $\nu^{(0)}_{\theta,S|x}=\nu^{(0)}_{\theta,V|x}$ (and thus $\operatorname{KL}\big(\nu^{(0)}_{\theta,S|x} \,\|\, \nu^{(0)}_{\theta,V|x}\big)=0$) and applying the Claim iteratively, we obtain \eqref{ineq:KL:P:theta:S:Xt:given:x:and:P:theta:V:Xt:given:x}.

\end{proof}

Combining Lemma \ref{lem:upper:bound:Dtv:and:kL:div:Xt:given:x:theta:S} with the theory of perturbation of Markov chains,  we can obtain an upper bound for $d_{TV}(\pi_{\theta,S} , \pi_{\theta,V})$. This is our last auxiliary result. 

\begin{lemma}
\label{lem:control:dtv:piStheta:piVtheta}
There exists a constant $K$ depending on $\mu$ and $\lambda$ such that 
\[
d_{TV}(\pi_{\theta,S} , \pi_{\theta,V})\leq 2\tau_{K}(N^{-1}\log(N)),
\]
for all $\theta\in\{0,1\}^{N\times N}$ and $S,V\in\mathcal{S}$ such that $|S\setminus V|=1$.
\end{lemma}
\begin{proof}
Let $t_N=\Big\lceil \log(2N)/\log((1-\lambda)^{-1})\Big\rceil$.
On one hand, Lemma \ref{lem:upper:bound:Dtv:and:kL:div:Xt:given:x:theta:S} yields
\[
d_{TV}\Big(\nu^{(t_N)}_{\theta,S|x},\nu^{(t_N)}_{\theta,S|y}\Big)\leq N(1-\lambda)^{t_N} \indiq_{\{x\neq y\}}\leq 2^{-1} \indiq_{\{x\neq y\}},
\]
for all $x\in\{0,1\}^N$, $\theta\in\{0,1\}^{N\times N}$ and $S\in {\cal S}$.

On the other hand, using Bretagnolle-Huber inequality and the fact that there exists some positive constant $C$ depending only on $\lambda$ such that $t_N \leq C \log(N)$ (since $N\geq 2$), we obtain that
\begin{equation*}
d_{TV}\Big(\nu^{(t_N)}_{\theta,S|x},\nu^{(t_N)}_{\theta,V|x}\Big) 
\leq \sqrt{1-\exp\big(-\operatorname{KL}(\nu^{(t_N)}_{\theta,S|x} \,\|\, \nu^{(t_N)}_{\theta,V|x})\big)} 
\leq \tau_{K}(N^{-1}\log(N)),
\end{equation*}
for all $x,y\in\{0,1\}^{N}$, $\theta\in\{0,1\}^{N\times N}$ and $S,V\in\mathcal{S}$ such that $|S\setminus V|=1$, where $K$ is a positive constant depending only on $\mu$ and $\lambda$.
Hence, applying Theorem 19.2.1 of \cite{rudolf2024perturbations} with $\alpha=1/2$, $\epsilon=\tau_{K}(N^{-1}\log(N))$, $Q(x,\cdot)=\nu^{(t_N)}_{\theta,S|x}(\cdot)$ and $K(x,\cdot)=\nu^{(t_N)}_{\theta,V|x}(\cdot)$, the result follows.
\end{proof}

Given Lemmas \ref{lem:upper:bound:kulllback:leibler:divergence} and \ref{lem:control:dtv:piStheta:piVtheta}, we are in position to prove Proposition
\ref{prop:upper:bound:total:variation}.

\begin{proof}[Proof of Proposition \ref{prop:upper:bound:total:variation}]
First, observe that 
\[
d_{TV}\Big(\P_{S}^{(1:T+1)} , \P_{V}^{(1:T+1)}\Big)\leq \max_{\theta\in\{0,1\}^{N\times N}}d_{TV}\Big(\P_{\theta,S}^{(1:T+1)} , \P_{\theta,V}^{(1:T+1)}\Big).
\]
For each $\theta\in\{0,1\}^{N\times N}$ and $S,V\in \mathcal{S}$, let us denote 
\[
\overline{\P}_{\theta,S,V}^{(1:T+1)}(x_1,\ldots, x_{T+1})=\pi_{\theta,S}(x_1)\prod_{t=2}^{T+1}P_{\theta,V}(x_t|x_{t-1}),
\]
for all $x_1,\ldots, x_{T+1}\in \{0,1\}^N$, where $P_{\theta,V}$ is the transition probability defined in \eqref{def1:transition:prob:S:theta} and $\pi_{\theta,S}$ is the stationary distribution associated with the transition probability $P_{\theta,S}$.
Observe that $\overline{\P}_{\theta,S,V}^{(1:T+1)}$ and $\P_{\theta,V}^{(1:T+1)}$ differ only from their initial distributions, namely, $\pi_{\theta,S}$ and $\pi_{\theta,V}$ respectively.  Using this fact, we can write     
\begin{align}
\label{eq1:proof:prop:upper:bound:total:variation}
d_{TV}\Big(\P_{\theta,S}^{(1:T+1)} , \P_{\theta,V}^{(1:T+1)}\Big)&\leq d_{TV}\Big(\P_{\theta,S}^{(1:T+1)} , \overline{\P}_{\theta,S,V}^{(1:T+1)}\Big)+d_{TV}\Big(\overline{\P}_{\theta,S,V}^{(1:T+1)} , \P_{\theta,V}^{(1:T+1)}\Big)\nonumber \\
& \leq d_{TV}\Big(\P_{\theta,S}^{(1:T+1)} , \overline{\P}_{\theta,S,V}^{1:(T+1)}\Big)+d_{TV}\big(\pi_{\theta,S} , \pi_{\theta,V}\big),
\end{align}
for all $\theta\in \{0,1\}^{N\times N}$.

Let us deal with the first term on the right-hand size of \eqref{eq1:proof:prop:upper:bound:total:variation}. To that end, we can use the stationarity of $(X_{t})_{t\in\Z}$ under $\P_{\theta,S}$ to show that  
\begin{align*}
\operatorname{KL}\Big(\P_{\theta,S}^{(1:T+1)}\|\overline{\P}_{\theta,S,V}^{(1:T+1)}\Big)&=\E_{\theta,S}\Bigg[\sum_{t=2}^{T+1}\log\Big(\frac{P_{\theta,S}(X_t|X_{t-1})}{P_{\theta,V}(X_t|X_{t-1})}\Big)\Bigg]\\
&=T\E_{\theta,S}\Bigg[\log\Big(\frac{P_{\theta,S}(X_1|X_{0})}{P_{\theta,V}(X_1|X_{0})}\Big)\Bigg]\\
&=T\E_{\theta,S}\Bigg[\operatorname{KL}\Big(P_{\theta,S}(\cdot|X_{0}) \,\|\, P_{\theta,V}(\cdot|X_{0})\Big)\Bigg].
\end{align*}
Combining the above identity with Lemma \ref{lem:upper:bound:kulllback:leibler:divergence} and then applying Bretagnolle-Huber inequality, we obtain that 
\[
d_{TV}(\P_{\theta,S}^{(1:T+1)} , \overline{\P}_{\theta,S,V}^{1:(T+1)})\leq \tau_K(TN^{-1}),
\]
for some constant $K$ depending only on $\mu$ and $\lambda$.

Plugging above inequality into \eqref{eq1:proof:prop:upper:bound:total:variation} and then applying Lemma \ref{lem:control:dtv:piStheta:piVtheta}, we conclude the proof.
\end{proof}

Finally, we prove Theorem \ref{thm:minmax:lower:bound}. In the proof, we will consider excitatory communities $S$ that belong to the following subset of ${\cal S}$:
\begin{equation}
\label{def:sub:class:of:excitatory:community}
\mathcal{S}^{\prime}=\{S'\cup\{2N_-+1,\ldots, N\}: S'\subset [2N_-] \ \text{with} \ |S'|=N_- \},
\end{equation}
with the convention that $\{2N_-+1,\ldots, N\}=\emptyset$ when $2N_-=N$. In words, it means that whatever $S\in \mathcal{S}^{\prime}$ is, the last $N_+ - N_-$ components are excitatory and $N_-$ components among the $2N_-$ first ones are chosen as excitatory. In particular, $|S| = N_+$, proving that $\mathcal{S}^{\prime}\subset {\cal S}$.

\begin{proof}[Proof of Theorem \ref{thm:minmax:lower:bound}]

We use the framework of \cite[Theorem 2.12]{Tsybakov2009Introduction} where the statistical model is parametrized by binary vectors $\omega$.
For each vector $\omega\in\{0,1\}^{N_-}$, let $S'_\omega$ be the subset of $[2N_-]$ defined as
\begin{equation}
\label{def:Swprime}
S'_\omega=\{i\in [N_-]:\omega_i=1\}\cup\{i+N_-\in [2N_-]\setminus [N_-]:\omega_i=0\}.
\end{equation}
Clearly, $|S'_\omega|=N_-$ and $S_\omega:=S'_\omega\cup\{2N_-+1,\ldots, N\}\in \mathcal{S}^{\prime}$ for each $\omega\in \{0,1\}^{N_-}.$
Next, observe that $|S_\omega^\prime\setminus S_\omega|=1$ for all $\omega,\omega^\prime\in\{0,1\}^{N_-}$ such that 
\[
d_{H}(\omega^\prime,\omega):=\sum_{i=1}^{N_-} \indiq_{\{\omega^\prime_i\neq \omega_i\}}=1.
\]
Hence, it follows from Proposition \ref{prop:upper:bound:total:variation} and \cite[Theorem 2.12, Item (ii)]{Tsybakov2009Introduction} that there exists positive constants $K_1$ and $K_2$ depending only on $\mu$ and $\lambda$ such that
\begin{equation}
\label{ineq1:proof:thm:minmax:lower:bound}
\inf_{\widehat{\omega}}\max_{\omega\in\{0,1\}^{N_-}}\E_{S_\omega}[d_{H}(\widehat{\omega},\omega)]\geq \frac{N_-}{2}\Big(1-\tau_{K_1}\Big(\frac{T}{N}\Big)+2\tau_{K_2}\Big(N^{-1} \log(N)\Big)\Big),    
\end{equation}
where the infimum is taken with respect to all random vectors $\widehat{\omega}$ taking values in $\{0,1\}^{N_-}$ which are measurable with respect $\sigma(X_1,\ldots, X_{T+1})$.     

One can use \eqref{ineq1:proof:thm:minmax:lower:bound} to conclude the proof. Indeed, take $(\widehat{\cal P}_+,\widehat{\cal P}_-)$ any random partition of $[N]$ which is measurable with respect to  $\sigma(X_1,\ldots, X_{T+1})$ and observe that 
\begin{align}
\max_{S\in{\cal S}}\E_{S}\Big[\operatorname{MR}(\widehat{\cal P}_+,\widehat{\cal P}_-)\Big]&\geq \max_{S\in{\cal S}^{\prime}}\E_{S}\Big[\operatorname{MR}(\widehat{\cal P}_+,\widehat{\cal P}_-)\Big]\nonumber \\
&\geq  \max_{\omega\in \{0,1\}^{N_-}}\E_{S_\omega}\Big[\operatorname{MR}(\widehat{\cal P}_+,\widehat{\cal P}_-)\Big]\nonumber \\
\label{ineq2:proof:thm:minmax:lower:bound}
&=N^{-1}\max_{\omega\in \{0,1\}^{N_-}}\E_{S_\omega}\Big[|\widehat{\cal P}_+\setminus S_\omega|+|S_\omega\setminus\widehat{\cal P}_+|\Big].
\end{align}
Next, defining $\widehat{\omega}\in\{0,1\}^{N_-}$ by $\widehat{\omega}_i = 1$ if and only if $i \in \widehat{\mathcal{P}}_+$, we have that $\widehat{\omega}$ is measurable with respect to $\sigma(X_1,\ldots, X_{T+1})$ and  
\begin{equation}
|\widehat{\cal P}_+\setminus S_\omega|+|S_\omega\setminus\widehat{\cal P}_+|\geq |S'_{\widehat{\omega}}\setminus S'_\omega|+|S'_\omega \setminus S'_{\widehat{\omega}}|=d_H(\widehat{\omega},\omega).
\label{ineq3:proof:thm:minmax:lower:bound}
\end{equation}
Combining \eqref{ineq1:proof:thm:minmax:lower:bound}, \eqref{ineq2:proof:thm:minmax:lower:bound} and \eqref{ineq3:proof:thm:minmax:lower:bound}, we deduce that    
\[
\max_{S\in{\cal S}}\E_{S}\Big[\operatorname{MR}(\widehat{\cal P}_+,\widehat{\cal P}_-)\Big]\geq \frac{N_-}{2N}\Big(1-\tau_{K_1}\Big(\frac{T}{N}\Big)+2\tau_{K_2}\Big(N^{-1} \log(N)\Big)\Big),
\]
for all random partition $(\widehat{\cal P}_+,\widehat{\cal P}_-)$ of $[N]$ which are $\sigma(X_1,\ldots, X_{T+1})$ measurable.
Finally, the result follows from the fact that $N_-\geq Nr_--1$.  
\end{proof}

\subsection{Exact recovery}
\label{sec:lower:bound:exact:recovery}
The goal of this section is to establish a minimax lower bound for the problem of exactly recovering
the communities ${\cal P}_+$ and ${\cal P}_-$.   
Here, we consider versions of our stochastic model starting from fixed initial conditions. Specifically, for each $x\in \{0,1\}^{N}$, $\theta\in\{0,1\}^{N\times N}$ and $S\subset [N]$ with $|S|=N_+$ (recall that $N_+=\lceil Nr_+ \rceil$), we denote $\P_{x,S,\theta}$ the probability measure under which the sequence of random vectors $(X_{t})_{t\geq 1}$ is a Markov chain on $\{0,1\}^{N}$ with transition probabilities given by \eqref{def1:transition:prob:S:theta} and \eqref{def2:transition:prob:S:theta}, starting from $X_1 = x$. 
Similarly as in the previous section, we also write $\P_{x,S}$ to denote the probability measure under which $\theta=(\theta_{ij})_{i,j\in [N]}$ is an i.i.d. sequence of random variables distributed as $\text{Ber}(p)$ and the sequence $(X_{t})_{t\geq 1}$ is distributed as $\P_{x,S,\theta}$ conditionally on $\theta$. Moreover, for each integer $T\geq 1$, we write $\P_{x,S,\theta}^{(1:T)}$ and $\P_{x,S}^{(1:T)}$ to denote the distribution of $(X_1,\ldots, X_{T})$ under $\P_{x,S,\theta}$ and $\P_{x,S}$ respectively. Finally, denote $r_{\min} = r_+\wedge r_-$ and $\phi(N)=\frac{\sqrt{N r_{\min}-1}}{1 + \sqrt{N r_{\min}-1}}$, for each $N\geq \lceil r_{\min}^{-1}\rceil$, and note that $\phi(N)\to 1$ as $N\to\infty.$

With this notation, we can prove the following minimax lower bound for exact recovery. 

\begin{proposition}
\label{prop:lower:bound:exact:recovery}
There exists a constant $K>0$ depending on $\mu$ and $\lambda$ such that for all $x\in\{0,1\}^{N}$, $N\geq \lceil r_{\min}^{-1}\rceil$ and $T\geq 1$, 
\begin{equation*}
    \inf_{(\widehat{\cal P}_+,\widehat{ \cal P}_-)} \max_{S \in \mathcal{S}}
    \P_{x,S}((\widehat{\cal P}_+,\widehat{\cal P}_-)\neq (S,S^c))\geq \phi(N)
    \Bigg(1 - \frac{1}{\log(Nr_{\rm min}-1)}\Big(\frac{2KT}{N}+\sqrt{\frac{KT}{N}}\Big)\Bigg),
\end{equation*}    
where $r_{\rm min}=r_+\wedge r_-$, ${\cal S}=\{S\subset [N]: |S|= \lceil Nr_+ \rceil\}$ and the infimum is taken over all random partitions $(\widehat{\cal P}_+,\widehat{\cal P}_-)$ of $[N]$ which are measurable with respect to $\sigma(X_1,\ldots, X_{T+1}).$     
\end{proposition}

Proposition \ref{prop:lower:bound:exact:recovery} implies that there exists $K$ such that, if $T=T_N$ and $N$ diverge but $T_N/(N\log(N)) \to 0$
\[
\liminf_{N\to\infty} \min_{(\widehat{\cal P}_+,\widehat{\cal P}_-)} \max_{S\in{\cal S}} \P_{x,S}((\widehat{\cal P}_+,\widehat{\cal P}_-)\neq (S,S^c)) = 1,
\]
showing that exact recovery is impossible in the minimax sense if $T = o(N\log(N))$. As we will see below, the proof of Proposition \ref{prop:lower:bound:exact:recovery} is similar to the proof of Theorem \ref{thm:minmax:lower:bound} except that Assouad's Lemma is replaced by \cite[Theorem 2.5]{Tsybakov2009Introduction}.

\begin{proof}

The proof is similar to that of Theorem \ref{thm:minmax:lower:bound}.  We only consider the case $r_-\leq r_+$. Recall that $N_{-}=N-N_+$.   
Also, for each vector $\omega\in\{0,1\}^{N_-}$, recall the definition of the set $S'_{\omega}$ given in \eqref{def:Swprime} and set $S_{\omega}:=S'_{\omega}\cup\{2N_-+1,\ldots, N\}$. As observed in the proof of Theorem \ref{thm:minmax:lower:bound}, the set $S_{\omega}$ belongs to the collection ${\cal S}^{\prime}$ defined in \eqref{def:sub:class:of:excitatory:community}, for each $\omega\in\{0,1\}^{N_-}$.

In the rest of the proof, let $S_0:= S_{\omega_0}$, where $\omega_0 = (0, \ldots, 0)\in \{0,1\}^{N_-}$ and $S_i:=S_{\omega_i}$ for $i\in [N_-]$, where $\omega_i\in \{0,1\}^{N_-}$ has value $1$ on its $i$-th coordinate and all the others are $0$. Clearly, $|S_0\setminus S_i|=1$ for all $i\in [N_-]$.
 
Fix $x\in \{0,1\}^{N_-}$. Combining the fact that $\P_{x,S}^{(1:T+1)}$ is a mixture of the probability measures $\P_{x,S,\theta}^{(1:T+1)}$ with the convexity of the Kullback-Leibler divergence \cite[Theorem 2.7.2]{Cover2006Elements}, we obtain, for all $i\in [N_-]$, that
\[
    \operatorname{KL}\left(\P_{x, S_0}^{(1:T+1)} \Big\| \P_{x, S_i}^{(1:T+1)}\right)
    \leq
    \max_{\theta \in [0,1]^{N \times N}}
   \operatorname{KL}\left(\P_{x, S_0, \theta}^{(1:T+1)}\Big\|\P_{x, S_i, \theta}^{(1:T+1)}\right).
\]

Next, combining the chain rule for the Kullback-Leibler divergence with Lemma \ref{lem:upper:bound:kulllback:leibler:divergence}, we have that
\[
    \operatorname{KL}\left(\P_{x, S_0, \theta}^{(1:T+1)} \,\Big\|\, \P_{x, S_{i},\theta}^{(1:T+1)}\right)
    \;\leq\; \frac{KT}{N},
\]
for all $i\in [N_-]$ and $\theta \in \{0,1\}^{N\times N}$, where $K>0$ depends only on $\mu$ and $\lambda$.
 
Therefore, applying \cite[Theorem 2.5]{Tsybakov2009Introduction} with $M = N_-$, $P_j=\P^{(1:T+1)}_{x,S_j}$ for $j\in \{0,1,\ldots, N_-\}$, $s=1/2$ and $\alpha = \dfrac{KT}{N \log(N_-)}$, we obtain that
\[
    \inf_{\widehat{\omega}} \max_{w \in \{0,1\}^{N_-}}
    \P_{x, S_w}\left(d_H(\widehat{w}, w)\geq 1/2\right)
    \geq
    \frac{\sqrt{N_-}}{1 + \sqrt{N_-}}
    \left(1 - \frac{2KT}{N\log(N_-)} - \sqrt{\frac{KT}{N\log^2(N_-)}}\right),
\]
where the infimum is taken with respect to all random vectors $\widehat{\omega}$ taking values in $\{0,1\}^{N_-}$ which are measurable with respect to $\sigma(X_1,\ldots, X_{T+1})$.

Now, take $(\widehat{\cal P}_+,\widehat{\cal P}_-)$ any random partition of $[N]$ which is measurable with respect to  $\sigma(X_1,\ldots, X_{T+1})$. Observing that $\{(\widehat{\cal P}_+,\widehat{\cal P}_-)\neq (S_w,S^c_w)\}=\{|\widehat{\cal P}_+\setminus S_\omega|+|S_\omega\setminus\widehat{\cal P}_+|\geq 1\}$ and that $\{d_{H}(\omega,\omega')\geq 1\}=\{d_{H}(\omega,\omega')\geq 1/2\}$, we obtain from \eqref{ineq3:proof:thm:minmax:lower:bound} that 
\begin{align*}
 \max_{S \in {\cal S}^{'}}
    \P_{x,S}\left((\widehat{\cal P}_+,\widehat{\cal P}_-)\neq (S,S^c)\right)& \geq  \max_{w \in \{0,1\}^{N_-}}
    \P_{x, S_w}\left((\widehat{\cal P}_+,\widehat{\cal P}_-)\neq (S_w,S^c_w)\right)\\
    &\geq
    \frac{\sqrt{N_-}}{1 + \sqrt{N_-}}
    \left(1 - \frac{2KT}{N\log(N_-)} - \sqrt{\frac{KT}{N\log^2(N_-)}}\right).   
\end{align*}
Using the fact that $N_-\geq Nr_--1$ in the above inequality, the result follows.
\end{proof}

\section{Covariance lemma}
\label{app:covariance:lemma}

For completeness, \cite[Lemma 3.8]{Chevallier2024inferring} is stated below.

\begin{lemma}
    \label{lem:covariance:produit}
    Let $z_k = (i_k,t_k) \in [N] \times \mathbb{N},$ for $k=1,\dots, 4$. Denote 
    $$B = \cov_\theta \left( X_{i_1, t_1} X_{i_2,t_2} \,,\, X_{i_3, t_3} X_{i_4,t_4} \right),$$ 
    and $E = \{z_k: k=1,\dots, 4\}$.  There exists a constant $K$ only depending on $\lambda$ such that:
    \begin{enumerate}
       \item If $ \# E = 1,$ then
        $
        |B|\leq 1.
        $
        \item If $ \# E = 2$, $z_1\neq z_2$ and $z_3\neq z_4$, then
        $ 
        |B|\leq 1.
        $
         \item If $ \# E = 2$, $z_1=z_2$ and $z_3=z_4$, then
        $$ 
        |B|\leq K(1-\lambda)^{|t_1 - t_3|}N^{-1}.
        $$
        \item If $ \# E = 3$ and $z_1=z_2$ or $z_3=z_4$, 
        then
        $$ 
        |B|\leq K (1-\lambda)^{|s_3-s_2|+|s_2-s_1|}N^{-2},
        $$
        where $s_1\leq s_2\leq s_3$ is an ordering of $\{t_1,t_2,t_3,t_4\}$.

        \item If $ \# E = 3$, $z_1\neq z_2$ and $z_3\neq z_4$, for instance assume that $z_1=z_3$ (and so $z_2\neq z_4$), then
        $$ 
        |B|\leq K (1-\lambda)^{|t_2-t_4|}N^{-1}.
        $$
        \item If $ \# E = 4 $ and $t_1 \wedge t_2 \geq  t_3\vee t_4$, then
        $$
        |B| \le K (1-\lambda)^{t_1 \vee t_2 - t_3 \wedge t_4}N^{-2} .
        $$
        \item If $ \# E = 4 $, then
$$
    |B| \le K (1-\lambda)^{|s_2 - s_1| + |s_3 - s_4|}N^{-2},
$$
where $s_1\leq \dots \leq s_4$ is an ordering of $\{t_1,\dots,t_4\}$.
       \end{enumerate}
\end{lemma}

\section{Remarks on $k$-means clustering in dimension 1}
\label{sec:remarks:on:kmeans:ckustering:1d}
In this section, we collect some remarks concerning $k$-means clustering in dimension 1 that we use in the proof of Corollary \ref{cor:extact_recovery}. In what follows, for any vector $v\in \R^N$, we denote $v_{(1)}\leq v_{(2)}\leq \ldots \leq v_{(N)}$ the order statistics associated to its coordinates.
Recall the definition of the vector $\widehat{\sigma}^{\rm ag}$ given in \eqref{def:hat:sigma:ag}. For each $S\subseteq [N]$, we denote  $\widehat{\mu}_S=|S|^{-1}\sum_{i\in S}\widehat{\sigma}^{\rm ag}_i$ and $r_{\rm min} = r_+ \wedge r_-$.

\begin{lemma}
\label{lem:structural:properties:kmeans}
Consider the vectors $\widehat{\sigma}^{\rm ag}$ and  $\overline{\sigma}^{\rm ag}$ defined in \eqref{def:hat:sigma:ag} and \eqref{def:sigmaN:and:sigma:aggregated} respectively. There are at most $N-1$ possible outcomes for $k$-means clustering applied to $\widehat{\sigma}^{\rm ag}$ with $k=2$. Moreover, these outcomes have the form $(S^{(k)}_+,S^{(k)}_-)$ where $S^{(k)}_-=\Big\{i\in [N]:\widehat{\sigma}^{\rm ag}_{(i)}\leq \widehat{\sigma}^{\rm ag}_{(k)}\Big\}$ and $S^{(k)}_+=[N]\setminus S^{(k)}_-$, for $1\leq k\leq N-1$. 

Furthermore, for all $N\geq \lceil 2/r_{\rm \min}\rceil$, the following inequalities hold simultaneously on the event $E=\{\|\widehat{\sigma}^{\rm ag}-\overline{\sigma}^{\rm ag}\|_{\infty}\leq \epsilon\}$ for all $\epsilon\leq c_1p r_{\rm min}/(1+r_{\rm min})$: 
\begin{enumerate}
    \item For all $1\leq k<|{\cal P}_-|$, 
\begin{equation}
\label{kmeans:decreasing:as:function:of:k}
\| \widehat{\sigma}^{\rm ag}-(\widehat{\mu}_{S^{(k)}_+}1_{{S^{(k)}_+}}+\widehat{\mu}_{S^{(k)}_-}1_{{S^{(k)}_-}})\|^2_2
>\| \widehat{\sigma}^{\rm ag}-(\widehat{\mu}_{S^{(k+1)}_+}1_{{S^{(k+1)}_+}}+\widehat{\mu}_{S^{(k+1)}_-}1_{{S^{(k+1)}_-}})\|^2_2.    
\end{equation}
\item For all $|{\cal P}_-|\leq k\leq N-1$, 
\begin{equation}
\label{kmeans:increasing:as:function:of:k}
\|\widehat{\sigma}^{\rm ag}-(\widehat{\mu}_{S^{(k)}_+}1_{{S^{(k)}_+}}+\widehat{\mu}_{S^{(k)}_-}1_{{S^{(k)}_-}})\|^2_2
<\| \widehat{\sigma}^{\rm ag}-(\widehat{\mu}_{S^{(k+1)}_+}1_{{S^{(k+1)}_+}}+\widehat{\mu}_{S^{(k+1)}_-}1_{{S^{(k+1)}_-}})\|^2_2.
\end{equation}
\end{enumerate}
\end{lemma}

\begin{proof}
It is well known that the partition of $k$-means is a Voronoi diagram. In our case, Voronoi cells are intervals (dimension equals one) and the optimal partition of $k$-means, denoted by $(S_+,S_-)$, is of the form $S_- = \{i\in [N]:\widehat{\sigma}^{\rm ag}_{i}\leq c \}$ and $S_+ = [N] \setminus S_-$ for some threshold value $c \in (\min_i \widehat{\sigma}^{\rm ag}_{i}, \max_i \widehat{\sigma}^{\rm ag}_{i})$. This collection of partitions corresponds to $\{(S^{(k)}_+,S^{(k)}_-): 1 \leq k \leq N-1\}$

Hence, it remains to prove \eqref{kmeans:decreasing:as:function:of:k} and \eqref{kmeans:increasing:as:function:of:k}. Since both proofs are similar, we only prove \eqref{kmeans:decreasing:as:function:of:k}. Therefore, in the rest of the proof, we fix $1\leq k<|{\cal P}_-|$.

Using that $\min_{c\in \R}\sum_{j\in S}(\widehat{\sigma}^{\rm ag}_j-c)^2=\sum_{j\in S}(\widehat{\sigma}^{\rm ag}_j-\widehat{\mu}_S)^2$ for any $S\subseteq [N]$ and that $\sum_{j\in S^{(k)}_+}(\widehat{\sigma}^{\rm ag}_j-c)^2-\sum_{j\in S^{(k+1)}_+}(\widehat{\sigma}^{\rm ag}_j-c)^2=(\widehat{\sigma}^{\rm ag}_{(k+1)}-c)^2$, for all $c\in \R$, one can check that 
\begin{equation}
\label{eq1:proof:lem:structural:properties:kmeans}
\sum_{j\in S^{(k)}_{+}}(\widehat{\sigma}^{\rm ag}_j-\widehat{\mu}_{S^{(k)}_+})^2-\sum_{j\in S^{(k+1)}_{+}}(\widehat{\sigma}^{\rm ag}_j-\widehat{\mu}_{S^{(k+1)}_+})^2\geq (\widehat{\sigma}^{\rm ag}_{(k+1)}-\widehat{\mu}_{S^{(k)}_+})^2.    \end{equation}
Arguing similarly, we can also check that 
\begin{equation}
\label{eq2:proof:lem:structural:properties:kmeans}
\sum_{j\in S^{(k+1)}_{-}}(\widehat{\sigma}^{\rm ag}_j-\widehat{\mu}_{S^{(k+1)}_-})^2-\sum_{j\in S^{(k)}_{-}}(\widehat{\sigma}^{\rm ag}_j-\widehat{\mu}_{S^{(k)}_-})^2\leq (\widehat{\sigma}^{\rm ag}_{(k+1)}-\widehat{\mu}_{S^{(k)}_-})^2.    \end{equation}
Combining \eqref{eq1:proof:lem:structural:properties:kmeans} and \eqref{eq2:proof:lem:structural:properties:kmeans}, we obtain that \eqref{kmeans:decreasing:as:function:of:k} holds, whenever
\begin{equation}
\label{eq3:proof:lem:structural:properties:kmeans}
|\widehat{\sigma}^{\rm ag}_{(k+1)}-\widehat{\mu}_{S^{(k)}_+}|>|\widehat{\sigma}^{\rm ag}_{(k+1)}-\widehat{\mu}_{S^{(k)}_-}|.   
\end{equation}

To prove that \eqref{eq3:proof:lem:structural:properties:kmeans} holds on the event $E$, we argue as follows. On the one hand, 
\begin{equation}
\label{eq4:proof:lem:structural:properties:kmeans}
|\widehat{\sigma}^{\rm ag}_{(k+1)}-\widehat{\mu}_{S^{(k)}_-}|=\frac{1}{|S^{(k)}_-|}\sum_{j=1}^{k}(\widehat{\sigma}^{\rm ag}_{(k+1)}-\widehat{\sigma}^{\rm ag}_{(j)})\leq \epsilon.
\end{equation}
On the other hand,
\begin{align*}
|\widehat{\sigma}^{\rm ag}_{(k+1)}-\widehat{\mu}_{S^{(k)}_+}|&=\frac{1}{N-k}\sum_{j=k+1}^N(\widehat{\sigma}^{\rm ag}_{(j)}-\widehat{\sigma}^{\rm ag}_{(k+1)})\nonumber \\
&\geq \frac{1}{N-k}\sum_{j=|{\cal P}_-|+1}^N(\widehat{\sigma}^{\rm ag}_{(j)}-\widehat{\sigma}^{\rm ag}_{(k+1)})\nonumber \\
&\geq  \frac{|{\cal P}_+|}{N-k}(\overline{\sigma}_+-\overline{\sigma}_--2\varepsilon)\geq \frac{|{\cal P}_+|}{N-1}(\overline{\sigma}_+-\overline{\sigma}_--2\varepsilon)>r^N_+(\overline{\sigma}_+-\overline{\sigma}_--2\varepsilon). 
\label{eq5:proof:lem:structural:properties:kmeans}
\end{align*}
Assuming that $N\geq N_0:=\lceil 2/(r_+\wedge r_-)\rceil$ so that $r_+^N\wedge r_-^N\geq (r_+\wedge r_-)/2$, it gives
\[
|\widehat{\sigma}^{\rm ag}_{(k+1)}-\widehat{\mu}_{S^{(k)}_+}|>\frac{(\overline{\sigma}_+-\overline{\sigma}_--2\varepsilon)r_{\rm min}}{2}.
\]
Finally, combining the above inequality with \eqref{eq4:proof:lem:structural:properties:kmeans}, we see that \eqref{eq3:proof:lem:structural:properties:kmeans} holds on the event $E$, for all $N\geq N_0$ as soon as $\varepsilon\leq \frac{(\overline{\sigma}_+-\overline{\sigma}_-)r_{\rm min}}{2(1+r_{\rm min})}=\frac{c_1pr_{\rm min}}{(1+r_{\rm min})}$.
\end{proof}

\pagebreak
\section{Notation tables}
\label{app:notation:tables}

\begin{table}[ht]
    \centering
    \label{tab:main:notation:scalars}
    \begin{tabular}{ll}
        \toprule
        \textbf{Notation} & \textbf{Description} \\
        \midrule
        $\mu$, $\lambda$, $p$, $r_{\pm}$ & Model parameters \\
        \midrule
        $m = \frac{\mu+(1-\lambda)pr_-}{1-(1-\lambda)p(r_+-r_-)}$ & Asymptotic spatio-temporal mean \\[10pt]
        $D = (1-\lambda)p(r_+-r_-)$ & Denominator of $m$ \\[5pt]
        $c_1 = (1-\lambda) m (1-m)$ & Signal strength normalized by the edge probability $p$ \\[5pt]
        $c_2 = \frac{(1-\lambda)^2 p^2 (r_+ - r_-)}{1-D^2} c_1$ & Bias factor \\[10pt]
        $\overline{\sigma}_{\pm} = c_2 \pm c_1p$ & Limits discriminating the communities \\
        \bottomrule
    \end{tabular}
    \caption{Main scalar notation: model parameters and derived scalar quantities.}
\end{table}

\begin{table}[ht]
    \centering
    \label{tab:main:notation:3versions}
    \begin{tabular}{lccc}
        \toprule
        \textbf{Notation} & \textbf{Empirical} & \textbf{Quenched} & \textbf{Annealed} \\
        \midrule
        \multicolumn{4}{l}{\textit{Environment}} \\
        Signed adjacency & - & $A^N$ & $\frac{p}{N} 1_N (1_{\mathcal{P}_+}-1_{\mathcal{P}_-})^\top$ \\
        Row-wise sums & - & $L^N$ & $p(r_+^N - r_-^N)1_N$ \\
        Column-wise sums & - & $C^N$ & $p(1_{\mathcal{P}_+}-1_{\mathcal{P}_-})$ \\
        \midrule
        \multicolumn{4}{l}{\textit{Covariance matrices}} \\
        1-lagged covariance & $\widehat{\Sigma}^{(1)}$ & $\Sigma^{(1)}$ & $\overline{\Sigma}^{(1)}$ \\
        0-lagged covariance & $\widehat{\Sigma}^{(0)}$ & $\Sigma^{(0)}$ & - \\
        \midrule
        \multicolumn{4}{l}{\textit{Key vectors}} \\
        Aggregated & $\widehat{\sigma}^{\rm ag}$ & $\sigma^{\rm ag}$ & $\overline{\sigma}^{\rm ag}$ \\
        Spectral & $\widehat{\sigma}^{\rm sp}$ & - & $\overline{\sigma}^{\rm sp}$ \\
        \bottomrule
    \end{tabular}
    \caption{Main matrix notation: empirical, quenched, and annealed versions. The terms \emph{quenched} and \emph{annealed} come from the theory of stochastic processes in random environment: quenched refers to quantities computed conditionally on a fixed realization of the random environment, while annealed refers to quantities where the random environment has been averaged out.}
\end{table}

\section{Supplement to the simulation study}
\label{app:plots}

Figures \ref{fig:parameters:vary:threshold} and \ref{fig:parameters:vary:kmeans} give an overview of the performance of the aggregated method combined with either mean threshold or $k$-means clustering as one of the parameters ($N$, $r_+$, $\beta$, $\lambda$ or $p$) varies. Except the top right panel ($r_+$ varies), ag\_threshold and ag\_kmeans show similar performances. The following remarks apply to both Figures. 

In this framework and for each time horizon $T$ fixed, the parameter $\beta$ seems to have no influence on the performance. 

As $\lambda$ increases, the performance deteriorates as expected because the strength of interaction in the system decreases so that the sample contains less and less information on the underlying graph. Mathematically, this appears in the facts that: 1) $\overline{\sigma}_+ - \overline{\sigma}_- = 2c_1 p$ vanishes as $\lambda \to 1$; 2) the constants in our structural results (e.g. Theorem \ref{thm:concentration:Sigma}) diverge as $\lambda \to 1$. 

Despite the fact that these constants also diverge when $\lambda \to 0$ (see Remark \ref{rem:constants}), our method performs well for small values of $\lambda$. 

Moreover, as $p$ increases, the performance improves as expected because the strength of interaction in the system increases. This is reflected in the fact that $\overline{\sigma}_+ - \overline{\sigma}_- = 2c_1 p$ is an increasing function of $p$. 

Finally, by comparing the top right panel of Figures \ref{fig:parameters:vary:threshold} and \ref{fig:parameters:vary:kmeans}, it is clear that the mean threshold clustering is not robust to unbalanced communities. That is the main reason why we recommend to use the ag\_kmeans method.

\begin{figure}[ht]
    \includegraphics[width=.49\textwidth]{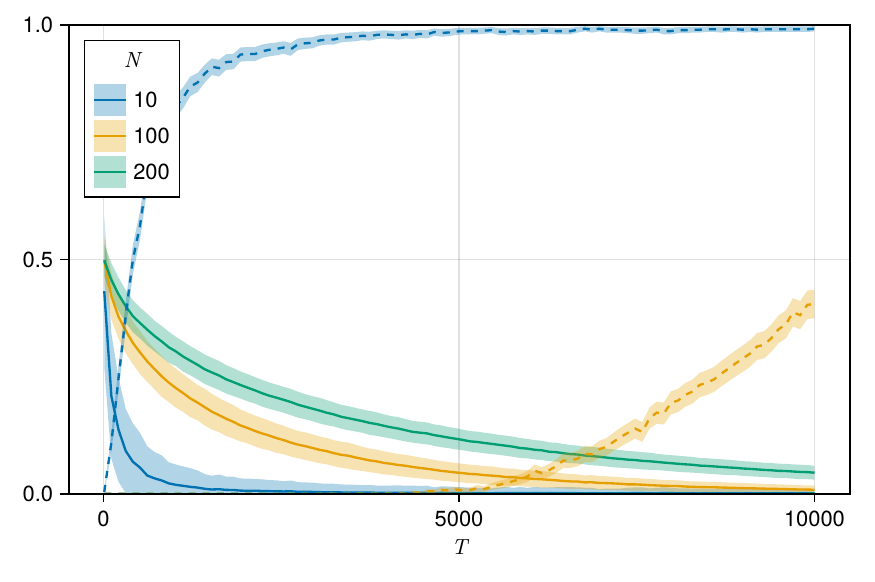}%
    \includegraphics[width=.49\textwidth]{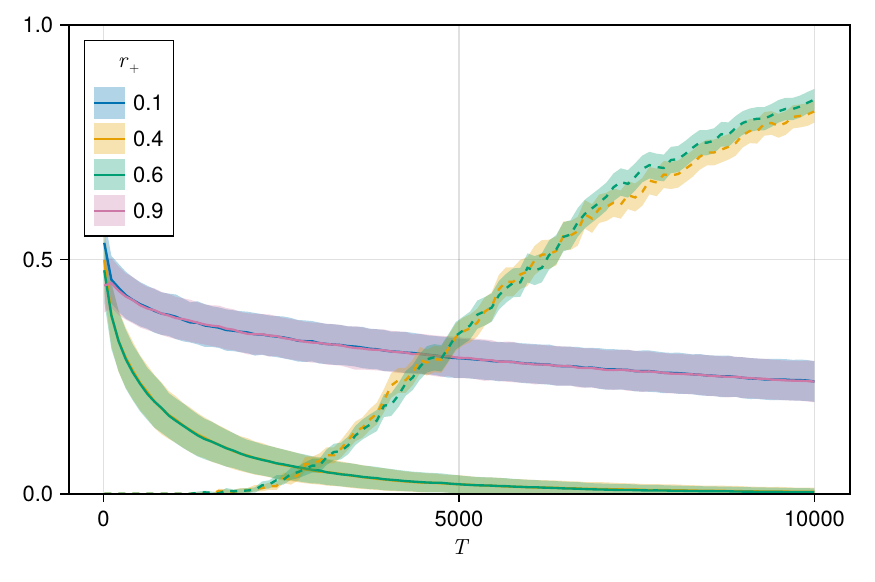}
    \includegraphics[width=.49\textwidth]{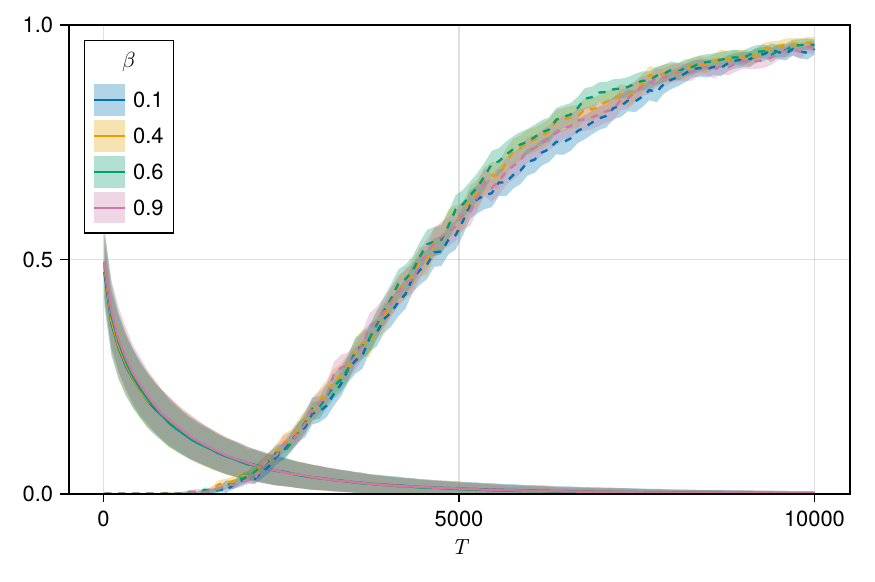}%
    \includegraphics[width=.49\textwidth]{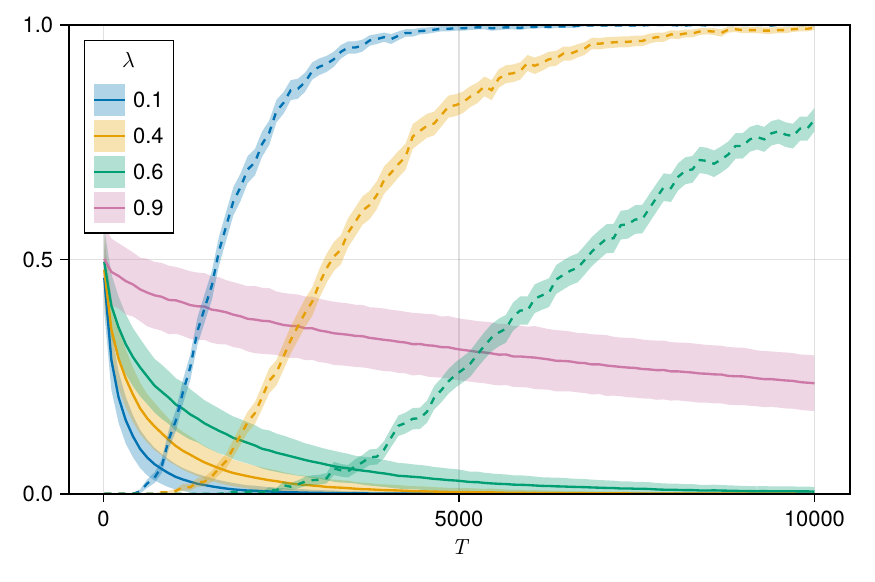}
    \includegraphics[width=.49\textwidth]{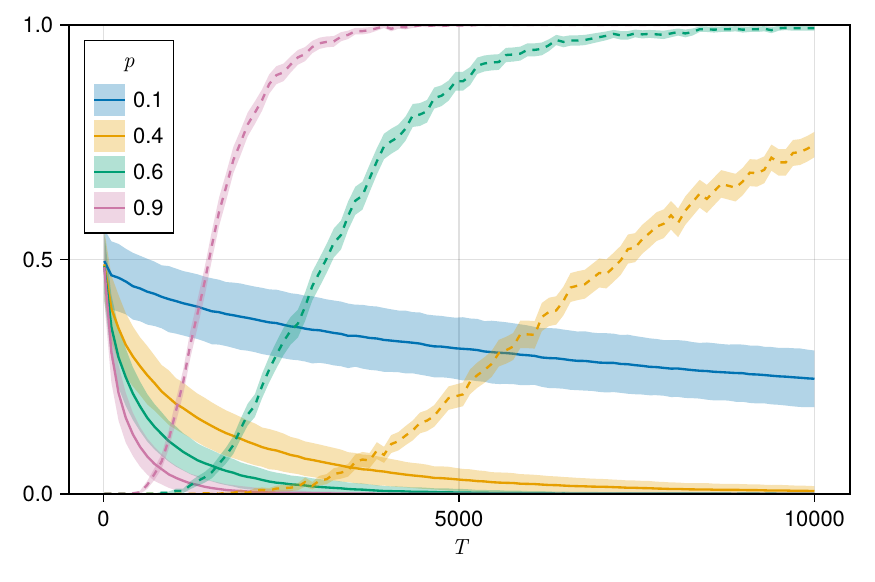}%
    \caption{\label{fig:parameters:vary:threshold}
    Estimated PER and MMR as function of $T$ for the ag\_threshold method. The increasing dashed lines correspond to the PER (the ribbon corresponds to the 95\% confidence interval). The decreasing solid lines correspond to the MMR (the width of the ribbon is equal to the standard deviation). The panels correspond to different choices of varying parameter (the non-varying parameters are chosen according to the default values given at the beginning of Section \ref{sec:simulation}). The values of the varying parameter are given by the color legends.
    }
\end{figure}

\begin{figure}[ht]
    \includegraphics[width=.49\textwidth]{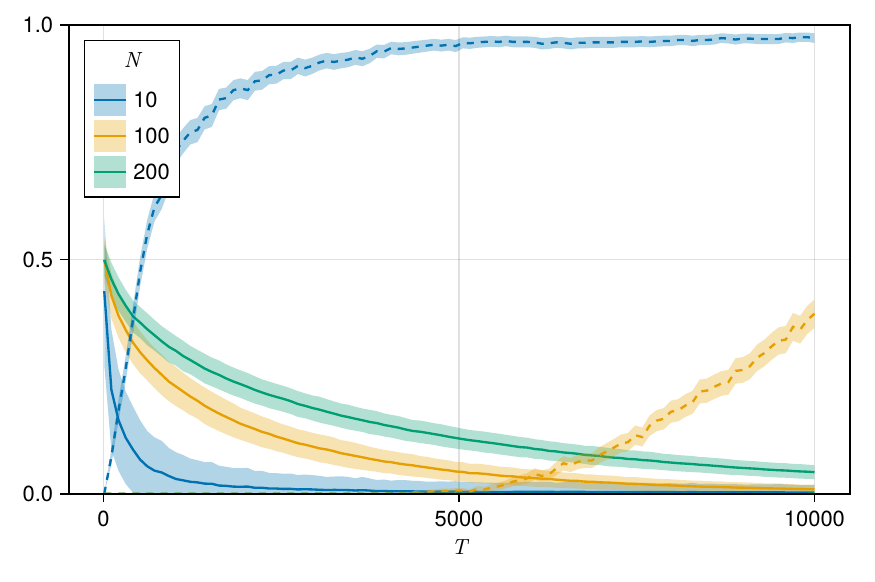}%
    \includegraphics[width=.49\textwidth]{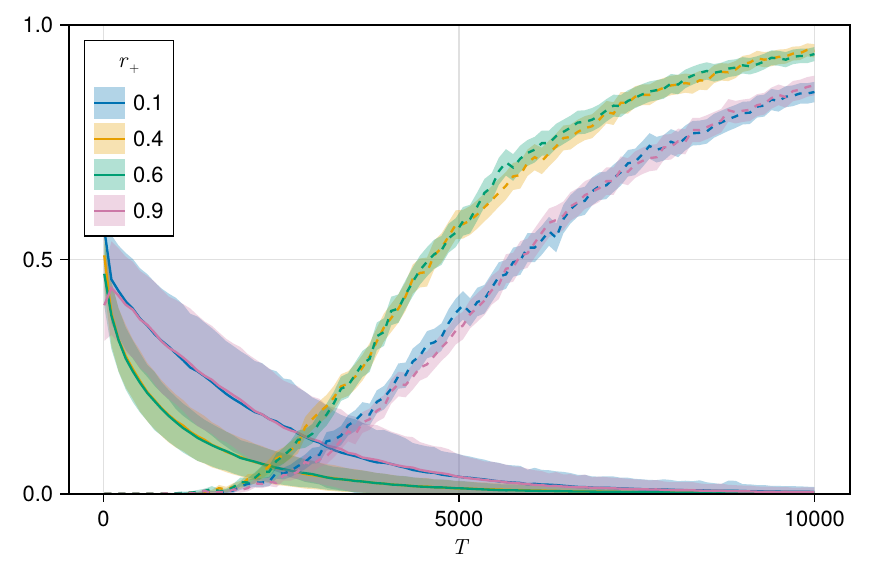}
    \includegraphics[width=.49\textwidth]{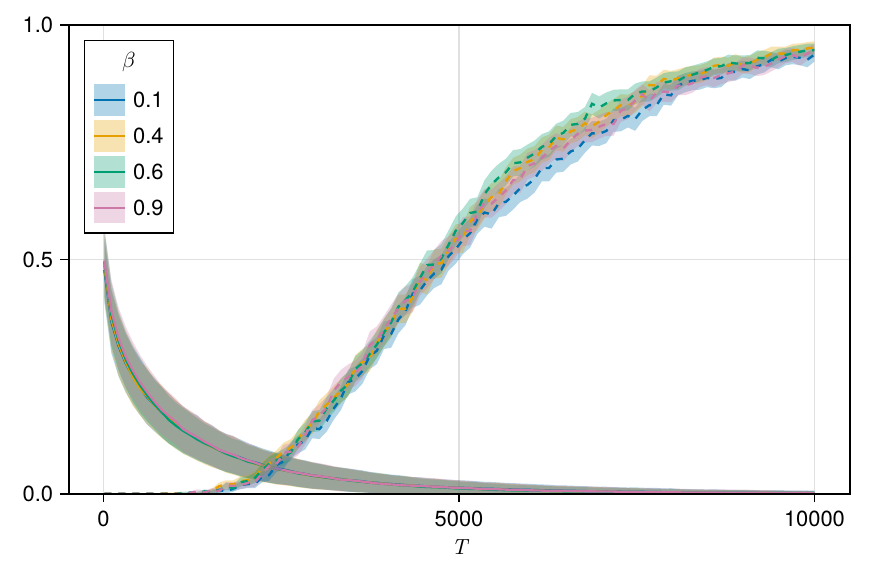}%
    \includegraphics[width=.49\textwidth]{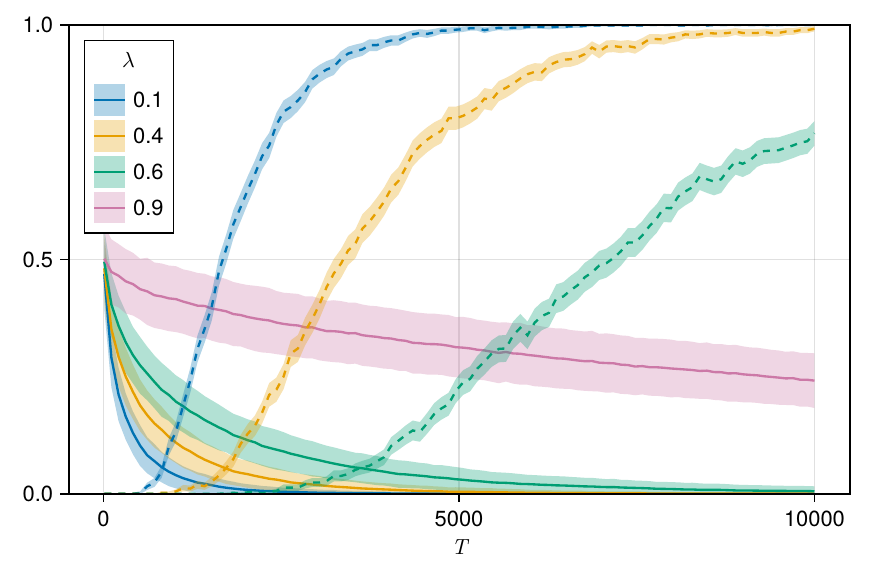}
    \includegraphics[width=.49\textwidth]{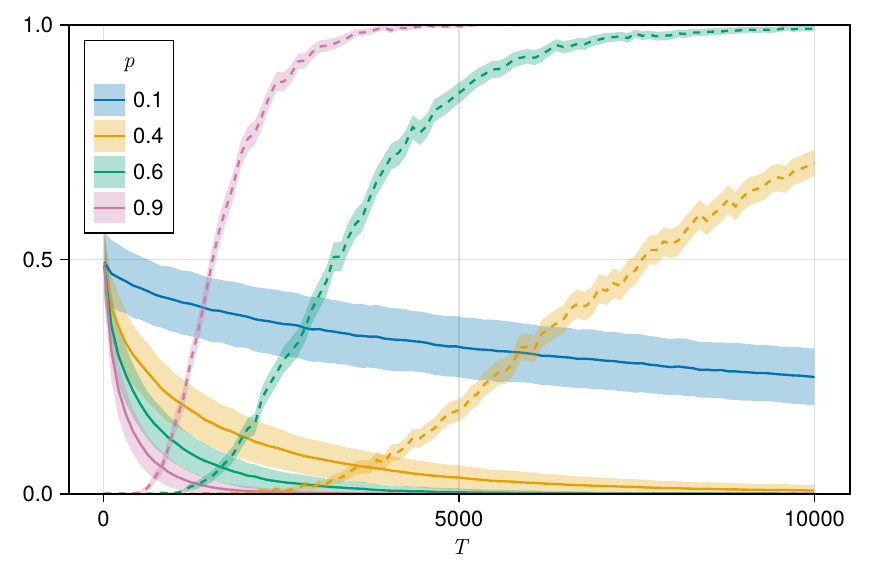}%
    \caption{\label{fig:parameters:vary:kmeans}
    Similar to Figure \ref{fig:parameters:vary:threshold} except that it corresponds to the ag\_kmeans method.
    }
\end{figure}

\end{document}